%% file: main.tex
\documentclass[10pt]{article}

\usepackage[a4paper,margin=1in]{geometry}

\usepackage{amsmath}
\usepackage{amsthm}
\usepackage{xcolor}
\usepackage{hyperref}
\hypersetup{
  colorlinks=true,
  linkcolor=black,
  citecolor=black,
  urlcolor=blue,
}
\usepackage[capitalize,nameinlink]{cleveref}   

\input{./setup/my_package_list}
\input{./setup/common_commands}

\input{./setup/commands_mesh}
\input{./setup/commands_splines}
\input{./tikz/myTikZCommands}
\theoremstyle{plain}
\newtheorem{theorem}{Theorem}[section]
\newtheorem{proposition}{Proposition}[section]
\newtheorem{lemma}{Lemma}[section]
\newtheorem{corollary}{Corollary}[section]

\theoremstyle{definition}
\newtheorem{definition}{Definition}[section]
\newtheorem{assumption}{Assumption}

\theoremstyle{remark}
\newtheorem{remark}{Remark}[section]

\title{Adaptively-refinable polar-spline discrete differential forms:\\
  hierarchical construction, exactness, and applications}

\author{%
  Diogo C. Cabanas\thanks{Delft Institute of Applied Mathematics, Delft
    University of Technology, The Netherlands
    (\texttt{d.costacabanas@tudelft.nl}).}
  \and Deepesh Toshniwal\thanks{Delft Institute of Applied Mathematics, Delft
    University of Technology, The Netherlands
    (\texttt{d.toshniwal@tudelft.nl}). Corresponding author.}
  \and Rafael Vazquez\thanks{Department of Applied Mathematics, University of
    Santiago de Compostela, Spain, and Galician Centre for Mathematical Research
    and Technology (CITMAga), University of Santiago de Compostela, Spain
    (\texttt{rafael.vazquez@usc.es}).}%
}

\date{\today}

\begin{document}

\maketitle

\begin{abstract}
	A common approach to representing disk-like and sphere-like geometries in isogeometric analysis is to use collapsed-edge singularities in the underlying parameterization.
	The resulting tensor-product B-spline spaces on such singular parameterizations lack the required smoothness to be used in isogeometric analysis.
	Polar splines (Toshniwal et al., CMAME 2017) rectify this issue by extracting a smoother subspace, and have been successfully used to discretize high-order partial differential equations, as well as to perform structure-preserving discretizations of the de Rham complex.
	The latter is particularly relevant for mixed formulations that appear in applications such as electromagnetism and fluid dynamics.

	The main contributions of this paper are twofold.
	We combine the polar spline construction with hierarchical splines to obtain adaptively-refinable polar splines.
	Moreover, we do so in the context of structure-preserving methods.
	That is, we show how to construct hierarchical polar spline $k$-form spaces, prove that the corresponding basis functions are linearly independent, and prove that they form a cohomologically-correct discrete de Rham complex.

	These contributions provide a mathematically sound foundation for adaptive structure-preserving simulations on polar geometries.
	We also provide numerical experiments that validate the theory and illustrate the practical behavior of the construction with problems where preserving the structure is mandatory.
	These include a study of spurious harmonics that highlights how cohomology-breaking refinements differ between the polar and tensor-product settings, a discretization of a Maxwell eigenvalue problem on a locally-refined polar mesh, and an adaptive convergence study for a Hodge--Laplace problem.
\end{abstract}



\input{./sections/introduction}
\input{./sections/preliminaries-bsplines}
\input{./sections/preliminaries-polarsplines}
\input{./sections/construction}

\input{./sections/exactness_polar}

\input{./sections/numerics}
\input{./sections/conclusions}

\appendix
\crefalias{section}{appendix}
\input{./sections/app-intermediate-extraction}
\input{./sections/app-extraction}
\input{./sections/app-extraction-refinement}

\bibliographystyle{plain}
\bibliography{bibliography}

\end{document}

%% file: setup/my_package_list.tex
\usepackage[T1]{fontenc}

\usepackage{amsfonts,amssymb}     
\usepackage{rotating}
\usepackage[normalem]{ulem}
\usepackage{graphicx}
\usepackage[labelformat=simple]{subcaption}   

\usepackage{wrapfig}
\usepackage{multirow}
\usepackage{xparse}
\usepackage{environ}
\usepackage{stackengine,wasysym}
\usepackage{mathdots}
\usepackage{booktabs}
\usepackage{siunitx}
\usepackage{csvsimple}

\usepackage{tensor}
\usepackage{varwidth,calc}

\usepackage{pgfplotstable}
\usepackage{pgfplots}
\pgfplotsset{compat=1.14}
\pgfplotsset{
	colormap={my basis colormap}{
		rgb255=(0, 114, 189);
		rgb255=(54, 106, 148);
		rgb255=(108, 98, 107);
		rgb255=(163, 91, 66);
		rgb255=(217, 83, 25);
	}
}
\pgfplotsset{
	colormap={my parula}{
		rgb255=(53.0655, 42.406, 134.9460);
		rgb255=(20.2336, 132.6061, 211.9511);
		rgb255=(55.5463, 184.8859, 157.9118);
		rgb255=(208.7187, 186.8470,  89.1966);
		rgb255=(248.9565, 250.6905, 13.7190);
	}
}
\pdfpageattr {/Group << /S /Transparency /I true /CS /DeviceRGB>>}
\usepackage{tikz}
\usepgfmodule{oo}
\usetikzlibrary{calc}
\usetikzlibrary{shapes}
\usetikzlibrary{plotmarks}
\usetikzlibrary{backgrounds}
\usetikzlibrary{decorations.pathmorphing,decorations.pathreplacing,decorations.markings}
\usetikzlibrary{arrows.meta,arrows,fit,matrix,positioning,shapes.geometric,positioning}
\usetikzlibrary{3d}
\usetikzlibrary{cd}
\providecommand{\tikzexternaldisable}{}

\providecommand{\tikzsetnextfilename}[1]{}

\usepackage{etoolbox}

\usepackage[inline]{enumitem}

\usepackage{scalerel,stackengine,wasysym}

\usepackage{accents}

\definecolor{lightgray}{gray}{0.80}

\usepackage{makecell}

\crefname{lemma}{Lemma}{Lemmas}
\crefname{remark}{Remark}{Remarks}
\crefname{theorem}{Theorem}{Theorems}
\crefname{proposition}{Proposition}{Propositions}
\crefname{corollary}{Corollary}{Corollaries}
\crefname{definition}{Definition}{Definitions}
\crefname{figure}{Fig.}{Figs.}
\crefname{table}{Table}{Tables}
\crefname{algorithm}{Algorithm}{Algorithms}
\crefname{section}{Section}{Sections}
\crefname{subsection}{Section}{Sections}
\crefname{appendix}{Appendix}{Appendices}
\crefname{assumption}{Assumption}{Assumptions}

%% file: setup/common_commands.tex
\newcommand{\Bezier}{B\'ezier~}

\usepackage{mathtools}

\newcommand{\mbf}[1]{{\boldsymbol{#1}}}
\newcommand{\RR}{{\mathbb{R}}}

\newcommand{\NN}{{\mathbb{N}}}
\newcommand{\PP}{{\mathfrak{P}}}

\DeclareMathOperator{\rank}{rank}

\newcommand{\dimwp}[1]{\dim{\left(#1\right)}}

\newcommand{\spaceSpan}[1]{\ensuremath{\text{span}}\left(#1\right)}
\newcommand{\isomorphic}{\cong}
\DeclareMathOperator{\trace}{tr}

\DeclareMathOperator{\supp}{supp}

\newcommand{\suppwp}[1]{\supp{\left(#1\right)}}

\stackMath
\newcommand\reallywidehat[1]{%
	\savestack{\tmpbox}{\stretchto{%
			\scaleto{%
				\scalerel*[\widthof{\ensuremath{#1}}]{\kern-.6pt\bigwedge\kern-.6pt}%
				{\rule[-\textheight/2]{1ex}{\textheight}}
			}{\textheight}%
		}{0.5ex}}%
	\stackon[1pt]{#1}{\tmpbox}%
}
\stackMath
\newcommand\reallywidecheck[1]{%
	\savestack{\tmpbox}{\stretchto{%
			\scaleto{%
				\scalerel*[\widthof{\ensuremath{#1}}]{\kern-.6pt\bigvee\kern-.6pt}%
				{\rule[-\textheight/2]{1ex}{\textheight}}
			}{\textheight}%
		}{0.5ex}}%
	\stackon[1pt]{#1}{\tmpbox}%
}

\ExplSyntaxOn
\NewEnviron{bmatrixT}
{
	\marine_transpose:V \BODY
}

\int_new:N \l_marine_transpose_row_int
\int_new:N \l_marine_transpose_col_int
\seq_new:N \l_marine_transpose_rows_seq
\seq_new:N \l_marine_transpose_arow_seq
\prop_new:N \l_marine_transpose_matrix_prop
\tl_new:N \l_marine_transpose_last_tl
\tl_new:N \l_marine_transpose_body_tl

\cs_new_protected:Nn \marine_transpose:n
{
	\seq_set_split:Nnn \l_marine_transpose_rows_seq { \\ } { #1 }
	\int_zero:N \l_marine_transpose_row_int
	\prop_clear:N \l_marine_transpose_matrix_prop
	\seq_map_inline:Nn \l_marine_transpose_rows_seq
	{
		\int_incr:N \l_marine_transpose_row_int
		\int_zero:N \l_marine_transpose_col_int
		\seq_set_split:Nnn \l_marine_transpose_arow_seq { & } { ##1 }
		\seq_map_inline:Nn \l_marine_transpose_arow_seq
		{
			\int_incr:N \l_marine_transpose_col_int
			\prop_put:Nxn \l_marine_transpose_matrix_prop
			{
				\int_to_arabic:n { \l_marine_transpose_row_int }
				,
				\int_to_arabic:n { \l_marine_transpose_col_int }
			}
			{ ####1 }
		}
	}
	\tl_clear:N \l_marine_transpose_body_tl
	\int_step_inline:nnnn { 1 } { 1 } { \l_marine_transpose_col_int }
	{
		\int_step_inline:nnnn { 1 } { 1 } { \l_marine_transpose_row_int }
		{
			\tl_put_right:Nx \l_marine_transpose_body_tl
			{
				\prop_item:Nn \l_marine_transpose_matrix_prop { ####1,##1 }
				\int_compare:nF { ####1 = \l_marine_transpose_row_int } { & }
			}
		}
		\tl_put_right:Nn \l_marine_transpose_body_tl { \\ }
	}
	\begin{bmatrix*}[r]
		\l_marine_transpose_body_tl
	\end{bmatrix*}
}
\cs_generate_variant:Nn \marine_transpose:n { V }
\cs_generate_variant:Nn \prop_put:Nnn { Nx }
\ExplSyntaxOff

\makeatletter
\catcode`! 3
\def\Transpose #1{\romannumeral0\expandafter
	\Mar@Transpose@a\romannumeral`^^@\Mar@DoOneRow #1\\!\\}

\def\Mar@DoOneRow #1\\{\Mar@DoOneRow@a {}#1&^^@&}%

\def\Mar@DoOneRow@a #1#2&{%
	\if^^@\detokenize{#2}\expandafter\@gobble\fi
	\Mar@DoOneRow@a {#1#2\\}%
}%

\def\Mar@Transpose@a #1#2\\{\ifx!#2\expandafter\Mar@FinishTranspose\fi
	\expandafter\Mar@Transpose@b\romannumeral`^^@\Mar@DoOneRow@a {}#2&^^@&#1}

\def\Mar@Transpose@b #1#2^^@\\{\Mar@Join {}#2^^@!#1}

\def\Mar@Join #1#2\\#3!#4\\%
{\if^^@\detokenize{#3}\expandafter\Mar@EndJoin\fi
	\Mar@Join {#1#2&#4\\}#3!}%

\def\Mar@EndJoin\Mar@Join #1^^@!^^@\\{\Mar@Transpose@a {#1^^@\\}}

\def\Mar@FinishTranspose
#1&^^@&#2\\^^@\\{ #2}

\catcode`! 12
\makeatother

%% file: setup/commands_mesh.tex
\newcommand{\basis}[1]{\mathcal{#1}}
\newcommand{\vSpace}[1]{\mathbb{#1}}
\newcommand{\complex}[1]{\mathit{\vSpace{#1}^\bullet}}

\newcommand{\mat}[1]{\mbf{\mathrm{#1}}}

\newcommand{\tpDomain}{\widehat{\Omega}}

\newcommand{\cU}{\theta}
\newcommand{\cV}{\rho}
\newcommand{\dcU}{\mathit{d\cU}}
\newcommand{\dcV}{\mathit{d\cV}}

\newcommand{\degU}{p^\cU}
\newcommand{\degV}{p^\cV}

\newcommand{\kntU}{\cU}
\newcommand{\kntV}{\cV}
\newcommand{\kntsU}{\mbf{\kntU}}
\newcommand{\kntsV}{\mbf{\kntV}}
\newcommand{\pkntsU}{\kntsU^\per}

\newcommand{\mult}[1]{\mu(#1)}

\newcommand{\ndofU}{m^\cU}
\newcommand{\ndofV}{m^\cV}

\newcommand{\bsp}{B}
\newcommand{\bspB}{\basis{\bsp}}
\newcommand{\bspS}{\vSpace{\bsp}}

\newcommand{\tpF}[2]{\bspS^{#1}}
\newcommand{\tpB}[2]{\bspB^{#1}}

\newcommand{\cS}{s}
\newcommand{\cT}{t}
\newcommand{\per}{\mathrm{per}}

\newcommand{\pSmU}{\nu^\per}
\newcommand{\pndofU}{m^{\cU,\per}}

\newcommand{\polar}{P}
\newcommand{\polarB}{\basis{\polar}}
\newcommand{\polarS}{\vSpace{\polar}}

\newcommand{\polarF}[1]{\polarS^{#1}}

\newcommand{\jet}[2]{J^{#1}\left(#2\right)}

\newcommand{\bary}{\lambda}

\newcommand{\tint}{\mathrm{int}}
\newcommand{\tpolar}{\mathrm{pol}}
\newcommand{\qDomain}{\Omega^{\per}}
\newcommand{\polarDomain}{\Omega}

\newcommand{\hSpace}[1]{\mathit{\vSpace{H}\vSpace{#1}}}
\newcommand{\hBasis}[1]{\mathit{\basis{H}\basis{#1}}}
\newcommand{\hbsp}[1]{\mathit{H #1}}

\newcommand{\hintComplex}{\complex{H\bsp}^{,\tint}}
\newcommand{\hintComplexZ}{\vSpace{H\bsp}^{\bullet,\tint}_0}
\newcommand{\hpolarComplex}{\complex{H\polar}}

\newcommand{\activeS}[2]{A^{#1}_{#2}}
\newcommand{\ractiveS}[2]{\tilde{A}^{#1}_{#2}}

%% file: tikz/myTikZCommands.tex
\definecolor{myBlue}{rgb} {0,0.4470,0.7410}
\definecolor{myRed}{rgb} {0.8500,0.3250,0.0980}
\definecolor{myGray1}{rgb} {0,0,0}
\definecolor{myGray2}{rgb} {0.6,0.6,0.6}
\definecolor{myGray3}{rgb} {0.45,0.45,0.45}
\definecolor{myGray4}{rgb} {0.8,0.8,0.8}
\definecolor{myGray5}{rgb} {0.55,0.55,0.55}

\definecolor{myCPlot1}{rgb} {0,    0.4470,    0.7410}
\definecolor{myCPlot2}{rgb} {0.8500,   0.3250,    0.0980}
\definecolor{myCPlot3}{rgb} {0.9290,    0.6940,    0.1250}
\definecolor{myCPlot4}{rgb} {0.4940,    0.1840,    0.5560}
\definecolor{myCPlot5}{rgb} {0.4660,    0.6740,    0.1880}
\definecolor{myCPlot6}{rgb} {0.3010,    0.7450,    0.9330}
\definecolor{myCPlot7}{rgb} {0.6350,    0.0780,    0.1840}

\definecolor{myCPColor}{rgb} {0.9255, 0.6941, 0.355}

\tikzstyle{scpColor}=[circle, fill=myCPColor]

\tikzset{
	show curve controls/.style={
		decoration={
			show path construction,
			curveto code={
				\draw[myGray1,dashed,eThickness]
				(\tikzinputsegmentfirst)
				-- (\tikzinputsegmentsupporta)
				-- (\tikzinputsegmentsupportb)
				-- (\tikzinputsegmentlast)
				;
				\fill[scpColor] (\tikzinputsegmentfirst) circle(3pt);
				\fill[scpColor] (\tikzinputsegmentsupporta) circle(3pt);
				\fill[scpColor] (\tikzinputsegmentsupportb) circle(3pt);
				\fill[scpColor] (\tikzinputsegmentlast) circle(3pt);
				\draw[#1,line width=1pt]
				(\tikzinputsegmentfirst)
				.. controls (\tikzinputsegmentsupporta)
				and (\tikzinputsegmentsupportb) ..
				(\tikzinputsegmentlast);
			}
		},decorate
	}
}

\tikzset{
	on each segment/.style={
		decorate,
		decoration={
			show path construction,
			moveto code={},
			lineto code={
				\path [#1]
				(\tikzinputsegmentfirst) -- (\tikzinputsegmentlast);
			},
			curveto code={
				\path [#1] (\tikzinputsegmentfirst)
				.. controls
				(\tikzinputsegmentsupporta) and (\tikzinputsegmentsupportb)
				..
				(\tikzinputsegmentlast);
			},
			closepath code={
				\path [#1]
				(\tikzinputsegmentfirst) -- (\tikzinputsegmentlast);
			},
		},
	},
	mid arrow/.style={postaction={decorate,decoration={
				markings,
				mark=at position .6 with {\arrow[#1]{Latex}}
	}}},
}

\tikzset{
	bThickness/.style={line width=#1\pgflinewidth},
	bThickness/.default={2},
}

\tikzset{
	eThickness/.style={line width=#1\pgflinewidth},
	eThickness/.default={0.5},
}

\tikzset{cross/.style={cross out, draw, 
		minimum size=2*(#1-\pgflinewidth), 
		inner sep=0pt, outer sep=0pt}}

\pgfooclass{bezier quad}{
	
	\attribute vertex 0;
	\attribute vertex 1;
	\attribute vertex 2;
	\attribute vertex 3;
	\attribute degree 0;
	\attribute degree 1;
	
	\method bezier quad(){
	}
	
	\method set vertices(#1,#2,#3,#4,#5,#6,#7,#8) {
		\pgfcoordinate{vertex 0}{\pgfpoint{#1}{#2}}
		\pgfcoordinate{vertex 1}{\pgfpoint{#3}{#4}}
		\pgfcoordinate{vertex 2}{\pgfpoint{#5}{#6}}
		\pgfcoordinate{vertex 3}{\pgfpoint{#7}{#8}}
	}
	
	\method set rectangle vertices(#1,#2,#3,#4) {
		\pgfmathsetmacro{\xz}{#1}
		\pgfmathsetmacro{\yz}{#2}
		\pgfmathsetmacro{\xo}{#1+#3}
		\pgfmathsetmacro{\yo}{#2+#4}
		\pgfoothis.set vertices(\xz,\yz,\xo,\yz,\xo,\yo,\xz,\yo);
	}
	
	\method set vertex nodes(#1,#2,#3,#4) {
		\pgfcoordinate{vertex 0}{\pgfpointanchor{#1}{center}}
		\pgfcoordinate{vertex 1}{\pgfpointanchor{#2}{center}}
		\pgfcoordinate{vertex 2}{\pgfpointanchor{#3}{center}}
		\pgfcoordinate{vertex 3}{\pgfpointanchor{#4}{center}}
	}
	
	\method set degrees(#1,#2) {
		\pgfooset{degree 0}{#1}
		\pgfooset{degree 1}{#2}
	}
	
	\method plot boundary() {
		\draw[eThickness, black] (vertex 0) -- (vertex 1) -- (vertex 2) -- (vertex 3) -- cycle;
	}

	\method plot domain point(#1,#2,#3,#4,#5,#6,#7) {
		\node [draw, black, fill=#7, shape=circle, minimum size=3pt, anchor=center,scale=#1] (domain point #6) at (barycentric cs:b0=#2,b1=#3,b2=#4,b3=#5) {};
	}

	\method set domain extent() {
		\coordinate (b0) at (barycentric cs:vertex 0=1,vertex 1=0.05,vertex 2=0.05,vertex 3=0.05) {};
		\coordinate (b1) at (barycentric cs:vertex 0=0.05,vertex 1=1,vertex 2=0.05,vertex 3=0.05) {};
		\coordinate (b2) at (barycentric cs:vertex 0=0.05,vertex 1=0.05,vertex 2=1,vertex 3=0.05) {};
		\coordinate (b3) at (barycentric cs:vertex 0=0.05,vertex 1=0.05,vertex 2=0.05,vertex 3=1) {};
	}
	
	\method plot domain points(#1,#2) {
		\pgfmathtruncatemacro{\pzero}{\pgfoovalueof{degree 0}}
		\pgfmathtruncatemacro{\pone}{\pgfoovalueof{degree 1}}
		\pgfoothis.mark domain pts(0,\pzero,0,\pone,#1,#2)
	}
	
	\method plot extended boundary() {
		\draw[eThickness, black] (vertex 0) -- (vertex 1) -- (vertex 2) -- (vertex 3) -- cycle;
		\foreach \x/\y in {0/1,1/2,2/3,3/0}{
			\draw[eThickness, dashed, black] (vertex \x) -- ($(vertex \y)!1.1!(vertex \x)$);
			\draw[eThickness, dashed, black] (vertex \y) -- ($(vertex \x)!1.1!(vertex \y)$);
		}
	}
	
	\method label vertices(#1,#2,#3,#4,#5) {
		\node [scale=#5] at (vertex 0) {#1};
		\node [scale=#5] at (vertex 1) {#2};
		\node [scale=#5] at (vertex 2) {#3};
		\node [scale=#5] at (vertex 3) {#4};
	}
	
	\method label edges(#1,#2,#3,#4,#5) {
		\node [scale=#5] at ($(vertex 0)!0.5!(vertex 1)$) {#1};
		\node [scale=#5] at ($(vertex 1)!0.5!(vertex 2)$) {#2};
		\node [scale=#5] at ($(vertex 2)!0.5!(vertex 3)$) {#3};
		\node [scale=#5] at ($(vertex 3)!0.5!(vertex 0)$) {#4};
	}
	
	\method label face(#1,#2) {
		\node [anchor=center,scale=#2] at (barycentric cs:vertex 0=0.25,vertex 1=0.25,vertex 2=0.25,vertex 3=0.25) {#1};
	}

	\method mark domain pts(#1,#2,#3,#4,#5,#6) {
		\pgfoothis.set domain extent();
		\pgfmathtruncatemacro{\pzero}{\pgfoovalueof{degree 0}}
		\pgfmathtruncatemacro{\pone}{\pgfoovalueof{degree 1}}
		\foreach \i in {#1,...,#2}{
			\foreach \j in {#3,...,#4}{
				\pgfmathsetmacro{\a}{\i/\pzero}
				\pgfmathsetmacro{\b}{\j/\pone}
				\pgfmathsetmacro{\aone}{(1-\a)*(1-\b)}
				\pgfmathsetmacro{\atwo}{(\a)*(1-\b)}
				\pgfmathsetmacro{\athree}{(\a)*(\b)}
				\pgfmathsetmacro{\afour}{(1-\a)*(\b)}
				\pgfoothis.plot domain point(#5,\aone,\atwo,\athree,\afour,\i\j, #6)
			}
		}
	}

	\method plot domain pt grid(#1,#2) {
		\pgfoothis.plot domain points(#1,#2);
		\pgfmathtruncatemacro{\pzero}{\pgfoovalueof{degree 0}}
		\pgfmathtruncatemacro{\pone}{\pgfoovalueof{degree 1}}
		\foreach \i in {0,...,\pzero}{
			\draw[dashed, eThickness] (domain point \i0) -- (domain point \i\pone);
		}
		\foreach \j in {0,...,\pone}{
			\draw[dashed, eThickness] (domain point 0\j) -- (domain point \pzero\j);
		}
		\pgfoothis.plot domain points(#1,#2);
	}

	\method mark edge pts(#1,#2,#3,#4) {
		{
			\ifthenelse{#1 = 0}{
				\pgfmathtruncatemacro{\iz}{0}
				\pgfmathtruncatemacro{\io}{\pgfoovalueof{degree 0}}
				\pgfmathtruncatemacro{\jz}{0}
				\pgfmathtruncatemacro{\jo}{#2}
				\pgfoothis.mark domain pts(\iz,\io,\jz,\jo,#3,#4)
			}{};}
	
		{\ifthenelse{#1=1}{
			\pgfmathtruncatemacro{\iz}{\pgfoovalueof{degree 0} - #2}
			\pgfmathtruncatemacro{\io}{\pgfoovalueof{degree 0}}
			\pgfmathtruncatemacro{\jz}{0}
			\pgfmathtruncatemacro{\jo}{\pgfoovalueof{degree 1}}
			\pgfoothis.mark domain pts(\iz,\io,\jz,\jo,#3,#4)
		}{};}

		{\ifthenelse{#1=2}{
			\pgfmathtruncatemacro{\iz}{0}
			\pgfmathtruncatemacro{\io}{\pgfoovalueof{degree 0}}
			\pgfmathtruncatemacro{\jz}{\pgfoovalueof{degree 1}-#2}
			\pgfmathtruncatemacro{\jo}{\pgfoovalueof{degree 1}}
			\pgfoothis.mark domain pts(\iz,\io,\jz,\jo,#3,#4)
		}{};}

		{\ifthenelse{#1=3}{
			\pgfmathtruncatemacro{\iz}{0}
			\pgfmathtruncatemacro{\io}{#2}
			\pgfmathtruncatemacro{\jz}{0}
			\pgfmathtruncatemacro{\jo}{\pgfoovalueof{degree 1}}
			\pgfoothis.mark domain pts(\iz,\io,\jz,\jo,#3,#4)
		}{};}
	}

	\method mark vertex pts(#1,#2,#3,#4,#5) {
		{
			\ifthenelse{#1 = 0}{
				\pgfmathtruncatemacro{\iz}{0}
				\pgfmathtruncatemacro{\io}{#2}
				\pgfmathtruncatemacro{\jz}{0}
				\pgfmathtruncatemacro{\jo}{#3}
				\pgfoothis.mark domain pts(\iz,\io,\jz,\jo,#4,#5)
			}{};}
		
		{\ifthenelse{#1=1}{
				\pgfmathtruncatemacro{\iz}{\pgfoovalueof{degree 0} - #2}
				\pgfmathtruncatemacro{\io}{\pgfoovalueof{degree 0}}
				\pgfmathtruncatemacro{\jz}{0}
				\pgfmathtruncatemacro{\jo}{#3}
				\pgfoothis.mark domain pts(\iz,\io,\jz,\jo,#4,#5)
			}{};}
		
		{\ifthenelse{#1=2}{
				\pgfmathtruncatemacro{\iz}{\pgfoovalueof{degree 0} - #2}
				\pgfmathtruncatemacro{\io}{\pgfoovalueof{degree 0}}
				\pgfmathtruncatemacro{\jz}{\pgfoovalueof{degree 1}-#3}
				\pgfmathtruncatemacro{\jo}{\pgfoovalueof{degree 1}}
				\pgfoothis.mark domain pts(\iz,\io,\jz,\jo,#4,#5)
			}{};}
		
		{\ifthenelse{#1=3}{
				\pgfmathtruncatemacro{\iz}{0}
				\pgfmathtruncatemacro{\io}{#2}
				\pgfmathtruncatemacro{\jz}{\pgfoovalueof{degree 1}-#3}
				\pgfmathtruncatemacro{\jo}{\pgfoovalueof{degree 1}}
				\pgfoothis.mark domain pts(\iz,\io,\jz,\jo,#4,#5)
			}{};}
	}
}
\pgfoonew \bezier=new bezier quad()

%% file: sections/introduction.tex
\section{Introduction}
\label{sec:introduction}

The construction of finite element spaces that satisfy a de Rham diagram is essential to
develop stable and accurate numerical schemes in several fields, such as electromagnetism,
fluid dynamics or magnetohydrodynamics. One of the main properties of this kind of finite
elements is that they preserve, at the discrete level, some of the fundamental properties of the continuous solution, such as the incompressibility condition in fluid dynamics.
The framework of finite element exterior calculus (FEEC) \cite{Arnold:2006,Arnold:2010} provides a systematic way of grouping, analyzing and developing such finite element spaces.
An analogous construction in the context of isogeometric analysis was given for the first time in \cite{Buffa:2010,Buffa:2011}, using tensor-product splines with high regularity for the discretization, and NURBS functions for the geometry. Since then, discrete spaces of spline differential forms have been used in several applications, such as fluid dynamics \cite{Evans:2013}, fluid-structure interaction \cite{Kamensky:2017}, electromagnetism \cite{Doelz:2024} or magnetohydrodynamics and plasma physics \cite{Holderied:2021,CamposPinto:2022}.

One of the main limitations of the cited construction is the tensor-product structure of the
spline spaces, for two main reasons: First, to maintain the tensor-product structure, the
refinement of any element has to be extended to the whole domain, which dramatically
increases the number of degrees of freedom. Second, the domain has to be defined as the
image of a square in $\mathbb{R}^2$, or a cube in $\mathbb{R}^3$, through a regular parameterization. While the definition can be generalized to multi-patch domains, formed by the disjoint union of several images of the square or cube, the high regularity property of splines is lost at the interfaces between patches, and the number of patches to construct simple geometries can be relatively high, see \cite[Sect.~2.3]{Cottrell:2009}  for examples of multi-patch constructions.

To deal with the first issue, different kinds of splines that break the tensor-product structure have been studied in the literature, such as T-splines, LR-splines, hierarchical splines or splines on triangulations. Among them, hierarchical B-splines \cite{Vuong:2011} and truncated hierarchical B-splines \cite{Giannelli:2012,Giannelli:2014}  are probably the most used in isogeometric analysis, thanks to their simple multi-level structure, which is based on two entities: the hierarchical mesh, formed by quadrilateral elements of different levels; and the hierarchical basis, which also combines tensor-product B-splines from different levels, selected from the relation of their supports to the hierarchical mesh. The extension to hierarchical splines that satisfy a de Rham sequence was introduced in \cite{Evans:2020}, with an unexpected drawback: not every hierarchical mesh gives a valid de Rham sequence. To solve this issue, a sufficient condition in the two-dimensional case was given in \cite{Evans:2020}, while a second sufficient condition was obtained in \cite{Shepherd:2024} in general dimensions.

Regarding the second issue, some geometric flexibility can be gained by the use of polar splines \cite{Toshniwal:2017}, a set of splines constructed on meshes with polar singularities.
These polar splines are easily defined (and computed) as linear combinations of
tensor-product B-splines, and one of their main interesting properties is that they allow
one to obtain $C^k$ continuity at the pole, which allows for their use in a direct formulation of high-order partial differential equations.
The special case of $C^1$ smoothness is particularly easy to study and implement, and was also discussed in the context of geometric design in \cite{Speleers:2021}.
The same $C^1$ smooth construction has been used in \cite{Toshniwal:2021} as the starting point of a polar spline de Rham diagram.
There, polar spline differential forms were introduced on disk-like (planar) and sphere-like (surface) geometries by defining suitable subspaces of B-spline differential forms via the notion of \Bezier extractions.
In particular, this polar spline de Rham diagram can incorporate tensor-product multi-degree B-splines \cite{Toshniwal:2020}, i.e., tensor-product spaces built from univariate spaces that allow local polynomial degree adaptivity.
It was proven in \cite{Toshniwal:2021} that this polar spline complex has the correct cohomological structure, and numerical evidence for optimal approximation of Hodge--Laplace problems, as well as for the inf-sup stability of the discretizations, was provided.
These spaces have been later used in follow-up works focused on  magnetohydrodynamics such as \cite{Holderied:2022,Possanner:2023}, incorporated in structure-preserving finite element libraries such as \cite{Holderied:2022b,Cabanas:2026}, while \cite{Gucclu:2025} provide an alternate implementation-oriented perspective of the construction that avoids building the spaces explicitly.

In this paper, we show that polar splines and hierarchical splines can be combined to
obtain locally refined splines on polar domains that form a discrete de Rham
complex.
The construction relies on the abstract framework of \cite{Giannelli:2014}, using the polar splines from \cite{Toshniwal:2021} at each level, and the selection mechanism of hierarchical splines to decide which functions are active. We prove that polar splines from \cite{Toshniwal:2021} are locally linearly independent, for each space of the discrete de Rham sequence. This property allows one to directly apply some of the theoretical results for hierarchical splines in \cite{Giannelli:2014}, and in particular to prove linear independence of active polar splines. The second important property regards the conditions on the hierarchical mesh to obtain a valid de Rham sequence. For the analysis we introduce the new concept of \emph{reduced complex at the pole}. Based on this concept we show that, to obtain a valid de Rham complex, there are no constraints on the hierarchical mesh in the neighborhood of the pole, and therefore the only constraints are the same as in the tensor-product case, which only have to be checked away from the pole.
Putting these ingredients together, our main result establishes sufficient conditions under which the hierarchical polar spline complex is exact.
Combined with the refinement algorithms proposed in \cite{Cabanas:2025,Dijkstra:2026} for hierarchical B-spline de Rham complexes, this provides the first construction of adaptively-refinable, structure-preserving spline differential forms on polar domains.

The outline of the paper is as follows. In \cref{sec:preliminaries-bsplines} we present the de Rham complex for tensor-product splines, with an extension to consider periodicity in the angular direction. Then, in \cref{sec:preliminaries-polar} we recall  the polar spline differential forms from \cite{Toshniwal:2021}, we introduce the new concept of reduced complex at the pole, and we prove that the basis functions of polar splines are locally linearly independent. \cref{sec:construction} regards the definition of the de Rham complex of hierarchical polar splines, their computation through suitable extraction operators, detailed in \cref{app:extraction}, and the proof of some of their  properties. In \cref{sec:exactness-polar} we prove the most important theoretical result of the paper, which is the exactness of the hierarchical polar spline complex, and the proof requires a similar result for an intermediate space with simple modifications at the pole. Finally, in \cref{sec:numerics} we present some numerical results on simple configurations to confirm our theoretical findings, and to show the capabilities of the constructed spaces.

%% file: sections/preliminaries-bsplines.tex
\section{B-spline complexes}\label{sec:preliminaries-bsplines}

In this section, all required notation and definitions will be introduced for
B-spline differential forms in two dimensions.
We will simplify the notation and exposition here by working only with classical polynomial B-splines but we note that the hierarchical construction and many results can be extended to more general settings (e.g., to multi-degree or Tchebycheffian splines \cite{Toshniwal:2017,Toshniwal:2020,Hiemstra:2020}) in a straightforward manner.
We start by introducing univariate B-splines in \cref{sec:preliminaries-univariate-bsplines} and then use them to construct cochain complexes with bivariate tensor-product B-splines in \cref{sec:preliminaries-bivariate-bsplines}.

\input{sections/preliminaries-univariate-bsplines}

\input{sections/preliminaries-bivariate-bsplines}

%% file: sections/preliminaries-univariate-bsplines.tex
\subsection{Univariate B-spline complexes}\label{sec:preliminaries-univariate-bsplines}

Let $\kntsU$ and $\kntsV$ be \emph{open knot vectors} for degrees $\degU, \degV \in\NN$ and of lengths $\ndofU+\degU+1,\ndofV+\degV+1 \in \NN$:
\begin{equation} \label{eq:knotVector}
\begin{split}
  \kntsU &:= [\underbrace{\kntU_1,\dots,\kntU_{\degU+1}}_{= 0},\kntU_{\degU+2},\dots,\kntU_{\ndofU},\underbrace{\kntU_{\ndofU+1},\dots,\kntU_{\ndofU+\degU+1}}_{= 1}],\quad 
  \kntU_{i+1} \geq \kntU_{i}\;,\\
  \kntsV &:= [\underbrace{\kntV_1,\dots,\kntV_{\degV+1}}_{= 0},\kntV_{\degV+2},\dots,\kntV_{\ndofV},\underbrace{\kntV_{\ndofV+1},\dots,\kntV_{\ndofV+\degV+1}}_{= 1}],\quad 
  \kntV_{i+1} \geq \kntV_{i}\;.
\end{split}
\end{equation}
The number of times a knot value is repeated in the knot vectors is called its \emph{multiplicity}, and we will denote the multiplicity of a knot $\gamma$ by $\mult{\gamma} \in \NN$.
We assume that $1 \leq \mult{\kntU_i} \leq \degU-1$, $i = \degU+2, \dots, \ndofU$, and $1 \leq \mult{\kntV_i} \leq \degV-1$, $i = \degV+2, \dots, \ndofV$.
Given such knot vectors, we will always assume that we can retrieve the corresponding degree by considering the multiplicity of the first (equivalently, the last) knot.

Using the above knot vectors, we can define B-splines $\bsp_{i,\degU}$, $i = 1,\dots,\ndofU$, and $\bsp_{i,\degV}$, $i = 1,\dots,\ndofV$, using the well-known Cox--deBoor recursion \cite{deBoor:1978}.
They form a basis for the spaces $\bspS^\cU$ and $\bspS^\cV$, respectively, of (at least) $C^1$-smooth piecewise-polynomial functions of degrees $\degU$ and $\degV$ on the interval $[0,1]$:
\begin{equation}
  \small
\begin{aligned}
  \bspS[\kntsU] &:= \big\{f\in C^1([0,1]): &&f|_{[\kntU_i, \kntU_{i+1}]} \in \PP_{\degU}\;,\;i=1,\dots,\ndofU+\degU\;,\\
  & &&\dfrac{d^rf}{\dcU^r}(\kntU_i^{-}) = \dfrac{d^rf}{\dcU^r}(\kntU_i^{+})\;,\;i=\degU+2,\dots,\ndofU\;\;,\;\;r = 0, \dots, \degU - \mult{\kntU_i}\big\}\;,\\
  \bspS[\kntsV] &:= \big\{f\in C^1([0,1]): &&f|_{[\kntV_i, \kntV_{i+1}]} \in \PP_{\degV}\;,\;i=1,\dots,\ndofV+\degV\;,\\
  & &&\dfrac{d^rf}{\dcV^r}(\kntV_i^{-}) = \dfrac{d^rf}{\dcV^r}(\kntV_i^{+})\;,\;i=\degV+2,\dots,\ndofV\;\;,\;\;r = 0, \dots, \degV - \mult{\kntV_i}\big\}\;.
\end{aligned}
\end{equation}
We will denote with $\bspB[\kntsU]$ and $\bspB[\kntsV]$ the column vectors which contains the corresponding B-spline bases.
With abuse of notation, we will denote sets containing these basis functions in the same way, and the meaning will be clear from the context.

Let $\sim$ be the equivalence relation on $[0,1]$ that identifies $\cU = 0$ and $\cU = 1$.
Given $1 \leq \pSmU \leq \degU-1$, we define the following periodic B-spline space on $[0, 1] / \sim$ with smoothness $\pSmU$ at $\cU=1$ as:
\begin{equation}
  \begin{split}
      \bspS[\pkntsU] &:= \left\{f \in \bspS[\kntsU]~:~\dfrac{d^rf}{\dcU^r}(1^-) = \dfrac{d^rf}{\dcU^r}(0^+)\;,r = 0, \dots, \pSmU\right\}\;.
  \end{split}
\end{equation}
The dimension of this space is $\pndofU = \ndofU - \pSmU - 1$, and the corresponding periodic B-spline basis functions will be denoted by $\bspB[\pkntsU] = [\bsp^\per_{i,\degU}~:~i = 1, \dots, \pndofU]$.

Finally, we define the B-spline spaces $\bspS[d\kntsU]$, $\bspS[d\kntsV]$ and $\bspS[d\pkntsU]$ which contain the derivatives of the splines in $\bspS[\kntsU]$, $\bspS[\kntsV]$ and $\bspS[\pkntsU]$, respectively.
To do so, we first introduce the notation $d\kntsU$ and $d\kntsV$, which refers to the knot vectors defined as:
\begin{equation}
  \begin{split}
    d\kntsU &:= [\kntU_2, \dots, \kntU_{\ndofU+\degU}]\;,\\
    d\kntsV &:= [\kntV_2, \dots, \kntV_{\ndofV+\degV}]\;.
  \end{split}
\end{equation}
These knot vectors define B-splines which span the following derivative spaces:
\begin{equation}
  \begin{split}
        \bspS[d\kntsU] := \left\{\frac{df}{d\cU}:f \in \bspS[\kntsU]\right\} = \partial_{\cU}\bspS[\kntsU]\;,\qquad
        \bspS[d\kntsV] := \left\{\frac{df}{d\cV}:f \in \bspS[\kntsV]\right\} = \partial_{\cV}\bspS[\kntsV]\;.
  \end{split}
\end{equation}
For the periodic case, we define the derivative space as:
\begin{equation}
  \bspS[d\pkntsU] := \left\{f \in \bspS[d\kntsU]~:~\dfrac{d^rf}{\dcU^r}(1^-) = \dfrac{d^rf}{\dcU^r}(0^+)\;,r = 0, \dots, \pSmU-1\right\}\;.
\end{equation}
Note that, unlike the non-periodic case, $\bspS[d\pkntsU] \supsetneq d\bspS[\pkntsU]$.
As before, we will collect the basis functions of the derivative spaces in $\bspB[d\kntsU]$, $\bspB[d\kntsV]$ and $\bspB[d\pkntsU]$.

We end this section on univariate splines with some properties of the periodic and non-periodic B-spline basis functions.
\begin{theorem}
  The following properties hold for the B-spline basis functions defined above:
  \begin{itemize}
    \item The B-spline basis functions in $\bspB[\mbf{\gamma}]$ are locally linearly
        independent, non-negative, and form a partition of unity, $\mbf{\gamma} \in \{\kntsU, \kntsV, \pkntsU, d\kntsU, d\kntsV, d\pkntsU\}$.
    \item The non-periodic B-spline basis functions in $\bspB[\mbf{\gamma}]$ satisfy the following end-point smoothness conditions:
    \begin{equation*}
      1 \leq i \leq p^\gamma+1\;:\;
      \begin{dcases}
        \dfrac{d^{i-1}\bsp_{j,p^\gamma}}{d\gamma^{i-1}}(0) \neq 0\;, & 1 \leq j \leq i\;,\\
        \dfrac{d^{i-1}\bsp_{j,p^\gamma}}{d\gamma^{i-1}}(0) = 0\;, & j > i\;,
      \end{dcases}
    \end{equation*}
    for $\mbf{\gamma} \in \{\kntsU, \kntsV, d\kntsU, d\kntsV\}$. Similar conditions hold at $\cU = 1$ and $\cV = 1$.
  \end{itemize}
\end{theorem}

%% file: sections/preliminaries-bivariate-bsplines.tex
\subsection{Bivariate periodic B-spline de Rham complexes}
\label{sec:preliminaries-bivariate-bsplines}
Using the above definitions, we can now introduce a B-spline discretization of a 2D de Rham complexes on the periodic domain $\qDomain = \tpDomain / \sim$, where $\sim$ is the equivalence relation identifying $\cU = 0$ and $\cU = 1$.

First, given $\kntsU, \kntsV$, we will define $\bspB[\pkntsU,\kntsV]$ as the column vector which contains the corresponding $\cU$-periodic tensor-product B-spline basis functions on $\qDomain$:
\begin{equation}
  \begin{split}
    \bspB[\pkntsU,\kntsV] &:= [\bsp^\per_{i,\degU}\bsp_{j,\degV} : i = 1, \dots, \pndofU\;\;,\;\;j = 1, \dots, \ndofV]\;.
  \end{split}
\end{equation}
In particular, we assume that we construct vectors such as the last one using the linear indexing $(i,j) \mapsto i + (j-1)\pndofU$.
We define the B-spline space $\bspS[\pkntsU,\kntsV]$ as the span of the basis functions in $\bspB[\pkntsU,\kntsV]$.
Other tensor-product B-spline spaces and the vectors containing their basis functions are defined in a similar way.

Using such tensor-product spaces, and with B-spline $0$-, $1$- and $2$-form spaces defined as:
\begin{equation}\label{eq:periodic-bsplineForms}
  \begin{split}
    \tpF{0,\per}{\pkntsU,\kntsV} &:= \bspS[\pkntsU,\kntsV]\;,\\
    \tpF{1,\per}{\pkntsU,\kntsV} &:= \bspS[d\pkntsU,\kntsV]\;\dcU + \bspS[\pkntsU,d\kntsV]\;\dcV\;,\\
    \tpF{2,\per}{\pkntsU,\kntsV} &:= \bspS[d\pkntsU,d\kntsV]\;\dcU\wedge \dcV\;,
  \end{split}
\end{equation}
we construct the B-spline de Rham complex on $\qDomain$ as follows:
\begin{equation}
  \begin{tikzcd}
    \complex{\bsp}^{,\per}~:~ \tpF{0,\per}{\pkntsU,\kntsV} \arrow{r}{d} & \tpF{1,\per}{\pkntsU,\kntsV} \arrow{r}{d} & \tpF{2,\per}{\pkntsU,\kntsV}\;.
  \end{tikzcd}
\end{equation}
The corresponding $k$-form basis functions are collected in the column vectors $\tpB{k,\per}{\pkntsU,\kntsV}$, $k = 0,1,2$, with $\pndofU\ndofV$ basis functions in $\tpB{0,\per}{\kntsU,\kntsV}$, $\pndofU\ndofV + \pndofU(\ndofV-1)$ basis functions in $\tpB{1,\per}{\kntsU,\kntsV}$ and $\pndofU(\ndofV-1)$ basis functions in $\tpB{2,\per}{\kntsU,\kntsV}$.
The following theorem characterizes the cohomology of the B-spline complex $\complex{\bsp}^{,\per}$; its proof can be found in \cite{Toshniwal:2021}, for example.
\begin{theorem}\label{thm:bspline-cohomology-periodic}
  $H^0(\complex{\bsp}^{,\per}) \isomorphic \RR$, $H^1(\complex{\bsp}^{,\per}) \isomorphic \RR$, and $H^2(\complex{\bsp}^{,\per}) = 0$.
\end{theorem}

%% file: sections/preliminaries-polarsplines.tex
\section{Spline complexes on a polar domain}\label{sec:preliminaries-polar}

In this section, we utilize the periodic B-spline complex from Section \ref{sec:preliminaries-bivariate-bsplines} to construct two spline complexes on a polar domain $\polarDomain \subset \RR^2$.
We start by defining the polar geometric map, $\mbf{F} : \qDomain \rightarrow \polarDomain$.
Next, we introduce an auxiliary spline complex on $\polarDomain$, which we call the \emph{intermediate B-spline complex on a polar domain}, and denote it $\complex{\overline{\bsp}}^{,\tint}$.
The spaces in this complex will not have sufficient regularity for discretizing the de Rham complex, but a subcomplex will be introduced following \cite{Toshniwal:2021} that does have the required regularity.
We call this subcomplex the \emph{polar spline complex} and denote it $\complex{\overline{\polar}}$.
Note that instead of recapping the presentation as in \cite{Toshniwal:2021}, we present an alternate but equivalent approach that is more amenable to the hierarchical construction in \cref{sec:construction}.

\subsection{The polar geometric map}
\label{subsec:polar-map}

\begin{figure}
    \centering
	\input{./tikz/scripts/polar_map_horiz}%

    \caption{
        A polar mapping which collapses the bottom edge of $\qDomain$ (displayed as a thick red line) to a single point, the pole, and yields a disk-like polar domain, $\polarDomain$.
        The thick black lines represent the sides of $\qDomain$ which have been identified to yield a periodic domain in the $\cU$-direction.
        See \eqref{eq:polar_map} for an explicit form of the geometric mapping that we use.
    }\label{fig:polar_map}
\end{figure}
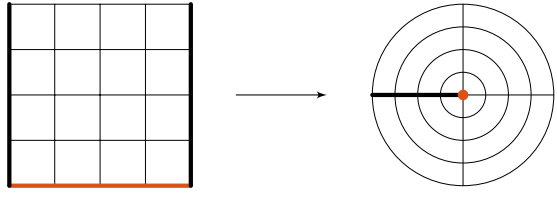

Let $\mbf{F} : \qDomain \rightarrow \polarDomain$ be a sufficiently smooth polar geometric map which takes the bottom edge, $\cV = 0$, and maps it to a single point, called the \emph{polar point} or, simply, the \emph{pole}; see Figure \ref{fig:polar_map}.
The Jacobian of such a mapping degenerates at the pole but we assume that it is invertible away from the pole.
In particular, we use the following spline mapping as a standard choice for $\mbf{F}$; this could be customized for specific applications.
With control points $\mbf{F}_{ij} \in \RR^2$, $i = 1, \dots, \pndofU$, $j = 1, \dots, \ndofV$, such that $\mbf{F}_{11} = \mbf{F}_{21} = \cdots = \mbf{F}_{\pndofU 1}$, we choose the polar mapping $\mbf{F}$ as
\begin{equation}\label{eq:polar_map}
    \begin{split}
        \polarDomain \ni (\cS, \cT) = \mbf{F}(\cU, \cV) &:= \left[F_\cU(\cU, \cV), F_\cV(\cU, \cV)\right]
        := \sum_{i=1}^{\pndofU}\sum_{j=1}^{\ndofV} \mbf{F}_{ij} \bsp^{\per}_{i,\degU}(\cU) \bsp_{j,\degV}(\cV)\;.
    \end{split}
\end{equation}
In particular, a standard choice for the control points \cite{Toshniwal:2017} could be $\mbf{F}_{ij} := \left( \rho_j\cos(\theta_i),\rho_j\sin(\theta_i) \right) \in \RR^2$, where
\begin{equation}\label{eq:control-points}
    \begin{split}
        \theta_i := 2\pi + \dfrac{(1-2i)\pi}{\pndofU}\;,\qquad
        \rho_j := \dfrac{j-1}{\ndofV-1}\;.
    \end{split}
\end{equation}

\begin{remark}
    This specific choice of the mapping means that $(0,0) \in \polarDomain$, the image of $\cV = 0$, is the polar point.
    While we only consider a single polar point in this work, the more general setting of double polar points from \cite{Toshniwal:2021} can be handled analogously.
\end{remark}

\subsection{Jets, local closedness/exactness, and reduced complexes at the pole}

To characterize the properties of the spline complexes in the following sections, we introduce three new definitions -- jets at the pole, local closedness/exactness at the pole, and reduced complex at the pole.
Thereafter, we end this subsection with a technical result that will be useful in proving the exactness of the polar spline complex in Theorem \ref{thm:polar-splines} below, as well as the exactness of the hierarchical polar spline complex in \cref{sec:exactness-polar}.

\begin{definition}[$r$-jets at the pole]
    Consider a function $f : \tpDomain \rightarrow \RR$ which is $r$-times continuously differentiable w.r.t. $\cV$ at $\cV = 0$.
    Then, we define its $r$-jet, $r \in \NN$, at the pole as
    \begin{equation*}
        \jet{r}{f} := \left(
            f\big|_{\cV = 0}\;,\; \dfrac{\partial f}{\partial \cV}\big|_{\cV = 0}\;,\; \dots\;,\;\dfrac{\partial^r f}{\partial \cV^r}\big|_{\cV = 0}
        \right)\;.
    \end{equation*}
    If all components of the $r$-jet are identically zero, we will say $\jet{r}{f} = 0$.
\end{definition}
\begin{definition}[Locally closed and exact $1$-forms]\label{def:locally-closed-exact}
    Consider a complex $\complex{V}$ of forms defined on $\tpDomain$.
    Consider a sufficiently regular $1$-form $f = f_\cU\;\dcU + f_\cV\;\dcV \in \vSpace{V}^1$.
    We say that $f$ is locally closed at the pole if
    \begin{equation*}
        \jet{0}{\dfrac{\partial f_\cV}{\partial \cU}} = \jet{0}{\dfrac{\partial f_\cU}{\partial \cV}}\;.
    \end{equation*}
    We say that $f$ is locally exact at the pole if it is locally closed and if there exists a $0$-form $g \in \vSpace{V}^0$ such that
    \begin{equation*}
        \begin{split}
            \jet{1}{f_\cU} &= \jet{1}{\dfrac{\partial g}{\partial \cU}}\;,\\
            \jet{0}{f_\cV} &= \jet{0}{\dfrac{\partial g}{\partial \cV}}\;.
        \end{split}
    \end{equation*}
\end{definition}
\begin{definition}[Reduced complex at the pole]\label{def:reduced-complex}
    We say that a complex $\complex{V}$ of forms (with sufficient regularity) on $\tpDomain$ is reduced at the pole if the spaces in it can be decomposed as:
    \begin{equation}
        \begin{split}
            \vSpace{V}^0 &= \RR \oplus g(\cV)\vSpace{V}_\cU^0 \oplus \check{\vSpace{V}}^0\;,\\
            \vSpace{V}^1 &= \left(g(\cV)\vSpace{V}_\cU^1 \oplus \check{\vSpace{V}}^{11}\right)\dcU
                            + \left(\vSpace{V}_\cU^0 1_\cV \oplus \check{\vSpace{V}}^{12}\right)\dcV\;,\\
            \vSpace{V}^2 &= \left(\vSpace{V}_\cU^1 1_\cV \oplus \check{\vSpace{V}}^2\right)\dcU \wedge \dcV\;,
        \end{split}
    \end{equation}
    where:
    \begin{itemize}
        \item $g$ is a $C^1$-smooth function of $\cV$ such that $g(0) = 0$ and $g'(0) \neq 0$;
        \item $\vSpace{V}_\cU^0$ and $\vSpace{V}_\cU^1$ are univariate spaces in $\cU$ such that $\frac{d\vSpace{V}_\cU^0}{\dcU} \subset \vSpace{V}_\cU^1$;
        \item $1_\cV$ is the space of functions constant in $\cV$;
        \item $f \in \check{\vSpace{V}}^0$ or $f \in \check{\vSpace{V}}^{11}$ implies that $\jet{1}{f} = 0$;
        \item $f \in \check{\vSpace{V}}^{12}$ or $f \in \check{\vSpace{V}}^2$ implies that $\jet{0}{f} = 0$.
    \end{itemize}
\end{definition}

\begin{lemma}\label{lem:reduced-complex}
    Consider a complex $\complex{V}$ which is reduced at the pole.
    Then, the following claims hold.
    \begin{itemize}
        \item If $f \in \vSpace{V}^1$ is locally closed at the pole, then it is locally exact at the pole.
        \item Consider a subcomplex $\complex{W} \subset \complex{V}$ and assume that $H^2(\complex{V}) = 0$. Then, $H^2(\complex{W}) = 0$ if the following conditions hold:
        \begin{itemize}
            \item $\vSpace{V}^1 \ni f = f_\cU\;\dcU + f_\cV\;\dcV$: $\jet{1}{f_\cU} = 0 = \jet{0}{f_\cV} \Longrightarrow f \in \vSpace{W}^1$;
            \item $\vSpace{W}^2 = \left\{f = f_{\cU\cV}\;\dcU\wedge\dcV \in \vSpace{V}^2 : \jet{0}{f_{\cU\cV}} = 0\right\}$.
        \end{itemize}
        In particular, the above implies that every $1$-form in $\vSpace{W}^1$ is locally exact at the pole.
    \end{itemize}
\end{lemma}
\begin{proof}
    Let $f \in \vSpace{V}^1$ be of the form:
    \begin{equation*}
        f = \left(g(\cV)f_{1}(\cU) + \check{f}_{2}(\cU,\cV)\right)\dcU
            + \left(f_{3}(\cU) + \check{f}_{4}(\cU,\cV)\right)\dcV\;,
    \end{equation*}
    with $f_1 \in \vSpace{V}_\cU^1$, $\check{f}_{2} \in \check{\vSpace{V}}^{11}$, $f_3 \in \vSpace{V}_\cU^0$, and $\check{f}_{4} \in \check{\vSpace{V}}^2$.
    Then, local closedness of $f$ implies:
    \begin{equation*}
        \frac{\partial f_\cU}{\partial\cV}\big|_{\cV=0} = \frac{\partial f_\cV}{\partial\cU}\big|_{\cV=0}
        \Longleftrightarrow 
        g'(0)f_1(\cU) = f_3'(\cU)\;.
    \end{equation*}
    Then, consider the $0$-form $F := \frac{g(\cV)}{g'(0)}f_3(\cU) \in \vSpace{V}^0$.
    It is easy to verify that:
    \begin{equation*}
        \begin{split}
            \jet{1}{\frac{\partial F}{\partial \cU}} &= \left( 0, f_3'(\cU) \right) = \left( 0, g'(0)f_1(\cU) \right) = \jet{1}{f_\cU}\;,\\
            \jet{0}{\frac{\partial F}{\partial \cV}} &= \left( f_3(\cU) \right) = \jet{0}{f_\cV}\;.
        \end{split}
    \end{equation*}
    Thus, $f$ is locally exact at the pole.

    Next, consider an arbitrary two-form $f = f_{\cU\cV} \dcU \wedge \dcV \in \vSpace{W}^2 \subset \vSpace{V}^2$.
    Then, since $H^2(\complex{V}) = 0$, we know that $f = dg$ for some $1$-form $g \in \vSpace{V}^1$.
    Since $dg \in \vSpace{W}^2$, we have $\jet{0}{dg} = 0$ and, thus, $g$ is locally closed at the pole.
    Therefore, it must be locally exact at the pole by the first part of this proof.
    Then, considering $g$ in the quotient space $\vSpace{V}^1 / d\vSpace{V}^0$, 
    we can find a representative $\overline{g} = \overline{g}_\cU\;\dcU + \overline{g}_\cV\;\dcV \in \vSpace{V}^1$ such that
    \begin{equation*}
        f = dg = d\overline{g}\;,\qquad
        \jet{1}{\overline{g}_\cU} = 0 = \jet{0}{\overline{g}_\cV} = 0\;.
    \end{equation*}
    Then, this must mean that $\overline{g} \in \vSpace{W}^1$, thus completing the proof.
\end{proof}

\subsection{B-spline complex on a polar domain}
We now define our first spline complex on $\polarDomain$.
First, we introduce some notation that will be used throughout the rest of this paper.
Given an a $q$-dimensional space $\vSpace{X}$, a basis $\basis{X} = [X_1, \dots, X_q]$ for it, and a full-rank matrix $\mat{D}$ of size $r \times q$, $r \leq q$, we will define a subspace $\vSpace{Y} \subset \vSpace{X}$ as the space spanned by the basis functions in $\basis{Y} = [Y_1, \dots, Y_r]$, where
\begin{equation}
    \basis{Y} := \mat{D}\basis{X}\;.
\end{equation}
That is, the $i$-th basis function in $\basis{Y}$ is defined as $Y_i := \sum_{j=1}^{q} \mat{D}_{ij}X_j$.
With abuse of notation, we will equivalently say that $\vSpace{Y} := \mat{D}\vSpace{X}$.

Let us now define appropriate subspaces of the periodic B-spline complex from Section \ref{sec:preliminaries-bivariate-bsplines} on $\tpDomain$.
In particular, we define the spaces as:
\begin{equation}
    \bspS^{k,\tint} := \mat{E}^{k,\tint}\tpF{k,\per}{\pkntsU,\kntsV}\;,\qquad k = 0, 1, 2\;,
\end{equation}
where $\mat{E}^{k,\tint}$ are the extraction operators defined in Appendix \ref{app:intermediate-extraction}.
With this construction, 0-form B-splines that do not vanish at $\cV=0$ are summed into a single function, 1-forms with non-vanishing trace at $\cV=0$ are removed, while the space of 2-forms is unchanged.
The basis associated to the space $\bspS^{k,\tint}$ defined above will be denoted as:
\begin{equation}
    \bspB^{k,\tint} = \left[ \bsp^{k,\tint}_i~:~i = 1, \dots, n^{k,\tint} \right] := \mat{E}^{k,\tint}\tpB{k,\per}{\kntsU,\kntsV}\;,\qquad k = 0, 1, 2\;,
\end{equation}
where $n^{k,\tint} = \dim(\bspS^{k,\tint})$ have been defined in Appendix \ref{app:intermediate-extraction}.
From the local linear independence of B-splines, it follows that the basis functions in $\bspB^{k,\tint}$ are also locally linearly independent.

\begin{lemma}\label{lem:local-linear-independence-intermediate}
    The basis functions in $\bspB^{k,\tint}$, $k = 0, 1, 2$, are locally linearly independent in $\tpDomain$.
\end{lemma}
\begin{proof}
    The basis functions in $\bspB^{k,\tint}$, $k = 1, 2$, are simply a subset of the B-spline basis functions in $\tpB{k,\per}{\kntsU,\kntsV}$, $k = 1, 2$ (see \cref{app:intermediate-extraction}) and are thus locally linearly independent.
    By the same reasoning, for $k = 0$, the only elements on which we might lose local linear independence are those that lie in the support of the first basis function,
    \begin{equation*}
        \bsp^{0,\tint}_{1}(\cU,\cV) = \sum_{i=1}^{\pndofU} \bsp^{\per}_{i,\degU}(\cU) \bsp_{1,\degV}(\cV)\;,
    \end{equation*}
    In other words, we only need to consider the elements that contain the bottom boundary, $\cV = 0$.
    Consider an arbitrary spline $f = \sum_{i} f_i^\tint \bsp^{0,\tint}_{i}$ and re-express it in terms of B-splines as:
    \begin{equation*}
        f(\cU,\cV) = f_1^\tint\sum_{i=1}^{\pndofU}\bsp^{\per}_{i,\degU}(\cU) \bsp_{1,\degV}(\cV) + \sum_{i=1}^{\pndofU}\sum_{j=2}^{\ndofV} f_{ij} \bsp^{\per}_{i,\degU}(\cU) \bsp_{j,\degV}(\cV)\;,
    \end{equation*}
    where $f_{ij} := f^\tint_{i + (j-2)\pndofU + 1 }$.
    Then, by local linear independence of B-splines, if $f$ is identically zero on any element that contains the bottom boundary, then $f_1^\tint = 0$ and $f_{ij} = 0$ if $\bsp^{\per}_{i,\degU}\bsp_{j,\degV}$ is non-zero on that element.
    This completes the proof.
\end{proof}

These spaces also form the following complex on $\tpDomain$:
\begin{equation}
    \begin{tikzcd}
        \complex{\bsp}^{,\tint}~:~ \bspS^{0,\tint} \arrow{r}{d} & \bspS^{1,\tint} \arrow{r}{d} & \bspS^{2,\tint}\;,
    \end{tikzcd}
\end{equation}
and it can be easily shown that it is exact and reduced at the pole; we note this result in the following theorem.
\begin{theorem}\label{thm:intermediate-complex}
    The complex $\complex{\bsp}^{,\tint}$ is reduced at the pole.
    Moreover, we have:
    $H^0(\complex{{\bsp}}^{,\tint}) \isomorphic \RR$, $H^1(\complex{{\bsp}}^{,\tint}) = 0$, and $H^2(\complex{{\bsp}}^{,\tint}) = 0$.
\end{theorem}
\begin{proof}
    We omit the proof of the complex's exactness since it follows from its similarity to the periodic B-spline complex (see \Cref{thm:bspline-cohomology-periodic}).
    Instead, since Definition \ref{def:reduced-complex} is a bit technical, we focus here on proving that the complex is reduced at the pole.
    We do so constructively, thus exemplifying the notation introduced in Definition \ref{def:reduced-complex}.
    
    First, choose $g(\cV) = \bsp_{2,\degV}(\cV)$, and define the following univariate spaces:
    \begin{equation*}
        \begin{split}
            \vSpace{V}^0_\cU &= \bspS[\pkntsU]\;,\qquad
            \vSpace{V}^1_\cU = \bspS[d\pkntsU]\;.
        \end{split}
    \end{equation*}
    The bivariate spaces are accordingly defined as:
    \begin{equation*}
        \begin{split}
            \check{\vSpace{V}}^0 &:= \spaceSpan{B~:~B \in \bspB^{0,\tint}
            ~~\text{and}~~
            \jet{1}{B} = 0}\;,\\
            \check{\vSpace{V}}^{11} &:= \spaceSpan{B~:~B\;\dcU \in \bspB^{1,\tint}
            ~~\text{and}~~
            \jet{1}{B} = 0}\;,\\
            \check{\vSpace{V}}^{12} &:= \spaceSpan{B~:~B\;\dcV \in \bspB^{1,\tint}
            ~~\text{and}~~
            \jet{0}{B} = 0}\;,\\
            \check{\vSpace{V}}^2 &:= \spaceSpan{B~:~B\;\dcU\wedge\dcV \in \bspB^{2,\tint}
            ~~\text{and}~~
            \jet{0}{B} = 0}\;.
        \end{split}
    \end{equation*}
    It can be checked that these spaces satisfy the properties in Definition \ref{def:reduced-complex} and, moreover, that the spaces $\bspS^{k,\tint}$, $k = 0, 1, 2$, can be decomposed as in Definition \ref{def:reduced-complex} using the above spaces.
    For instance, consider $f \in \bspS^{0,\tint}$, $f = \sum_{i} f_i^\tint \bsp^{0,\tint}_i$
    Then, we can equivalently express it as follows:
    \begin{equation*}
        \begin{split}
            f &= \sum_{i} f_i^\tint \bsp^{0,\tint}_i\;\\
        &= f_1^\tint\sum_{i=1}^{\pndofU}\sum_{j = 1}^{\ndofV}\bsp^{\per}_{i,\degU}\bsp_{j,\degV}
        + \left(\sum_{i=1}^{\pndofU}(f_{i2}-f_1^\tint)\bsp_{i,\degU}^{\per}\right)\bsp_{2,\degV}
        + \sum_{i=1}^{\pndofU}\sum_{j=3}^{\ndofV} (f_{ij}-f_1^\tint) \bsp^{\per}_{i,\degU} \bsp_{j,\degV}\;,
        \end{split}
    \end{equation*}
    where $f_{ij} := f^\tint_{i + (j-2)\pndofU + 1 }$.
    By partition of unity, we see that the first term is a constant function, the second term is a member of $g(\cV)\vSpace{V}_\cU^0$, and the third term is in $\check{\vSpace{V}}^0$.
    The spaces $\bspS^{1,\tint}$ and $\bspS^{2,\tint}$ can be similarly treated, thus completing the proof.
\end{proof}

\begin{remark}
    Note that, from this section onwards, we suppress the dependence of the spline spaces on $\{\pkntsU, \kntsV\}$ to emphasize their non-tensor-product structure and to simplify the notation.
\end{remark}

Using this complex, we could define an intermediate B-spline complex on $\polarDomain$ via pullbacks through the polar mapping $\mbf{F}$,
\begin{equation}
    \begin{tikzcd}
        \complex{\overline{\bsp}}^{,\tint}~:~ \overline{\bspS}^{0,\tint} \arrow{r}{d} & \overline{\bspS}^{1,\tint} \arrow{r}{d} & \overline{\bspS}^{2,\tint}\;,
    \end{tikzcd}
\end{equation}
where we define the spaces as
\begin{equation}\label{eq:intermediate-polar-spaces-physical}
    \overline{\bspS}^{k,\tint} := \left\{ f ~:~ \mbf{F}^* f  \in \bspS^{k,\tint} \right\}\;,\qquad k = 0, 1, 2\;,
\end{equation}
and $\mbf{F}^*$ is the pullback induced by the mapping $\mbf{F}$.
Note, however, that this intermediate complex $\complex{\overline{\bsp}}^{,\tint}$ is not a subcomplex of the $L^2$ de Rham complex on $\polarDomain$ because the spaces do not have sufficient regularity at the pole.
For instance, $2$-forms in $\overline{\bspS}^{2,\tint}$ will in general not be $L^2$ integrable on $\polarDomain$.
This is rectified in the following section.

\subsection{Polar spline complex}
We define polar spline form spaces $\polarS^k \subset \bspS^{k,\tint}$, $k = 0, 1, 2$, as appropriate subspaces that will ensure their required regularity on $\polarDomain$.
We do so by constructing matrices $\mat{E}^{k}$, $k = 0, 1, 2$, as shown in Appendix \ref{app:extraction}, and then using them to define $\polarF{k}$ as below,
\begin{equation}
    \begin{split}
        \polarF{k} &:= \mat{E}^{k}\bspS^{k,\tint}\;,\qquad
        \polarB^k := \mat{E}^{k}\bspB^{k,\tint}\;,\qquad
        k = 0, 1, 2\;.
    \end{split}
\end{equation}
It can be readily shown that these spaces form the following complex,
\begin{equation}
    \begin{tikzcd}
        {\complex{\polar}}~: & \polarF{0} \arrow{r}{d} & \polarF{1} \arrow{r}{d} & \polarF{2}
    \end{tikzcd}\;.
\end{equation}
We end this section with a theorem that summarizes all the properties of $\complex{\polar}$ and its spaces.

\begin{theorem}\label{thm:polar-splines}
    ~
    \begin{enumerate}[label=(\Alph*)]
        \item\label{item:cardinality} Denoting the basis functions in $\polarB^k$ as $\left[\polar^k_{i} : i = 1, \dots, n^k\right]$, we have:
        \begin{itemize}
            \item $n^0 = n^{0,\tint} - \pndofU + 2$,
            \item $n^1 = n^{1,\tint} - 2\pndofU + 2$,
            \item $n^2 = n^{2,\tint} - \pndofU$.
        \end{itemize}
        \item\label{item:nnz-basis-bottom} At the bottom boundary of $\tpDomain$, i.e., at $\cV = 0$, the following hold.
        \begin{itemize}
            \item Only $\polar^0_i$, $i = 1, 2, 3$, have non-zero values and first derivatives.
            \item Only $\polar^1_i = \polar^1_{i,\cU}\dcU + \polar^1_{i,\cV}\dcV$, $i = 1, 2$, are such that at least one of $\polar^1_{i,\cU}$ or $\polar^1_{i,\cV}$ has non-zero values.
            \item All $\polar^2_i = \polar^2_{i,\cU\cV}\dcU\wedge\dcV$ are such that $\polar^2_{i,\cU\cV}$ has zero values.
        \end{itemize}
        \item\label{item:linear-ind-bottom} At the bottom boundary of $\tpDomain$, i.e., at $\cV = 0$, the following hold.
        \begin{itemize}
            \item If $f = \sum_{i=1}^3 c_i \polar^0_i$, $i = 1, 2, 3$, has zero values and first derivatives, then all $c_i$ are zero.
            \item If $f = f_\cU\dcU + f_\cV\dcV = c_1 \polar^1_1 + c_2 \polar^1_2$ is such that both $f_\cU$ and $f_\cV$ have zero values, then $c_1 = c_2 = 0$.
        \end{itemize}
        \item\label{item:limited-influence} Except for the basis functions identified in \cref{item:nnz-basis-bottom}, every other basis function in $\polarB^k$ is identical to some tensor-product B-spline $k$-form in $\bspB^{k,\tint}$, $k = 0, 1, 2$. Moreover:
        \begin{itemize}
            \item $\bspS^{0,\tint} \ni f\;,\;\jet{1}{f} = 0 \Longrightarrow f \in \polarF{0}$,
            \item $\bspS^{1,\tint} \ni f = f_\cU\;\dcU + f_\cV\;\dcV$: $\jet{1}{f_\cU} = 0 = \jet{0}{f_\cV} \Longrightarrow f \in \polarF{1}$,
            \item $\polarF{2} = \left\{\bspS^{2,\tint} \ni f = f_{\cU\cV}\;\dcU\wedge\dcV : \jet{0}{f_{\cU\cV}} = 0\right\}$.
        \end{itemize}
        \item\label{item:local-linear-independence} The basis functions in $\polarB^k$, $k = 0, 1, 2$, are locally linearly independent.
        \item\label{item:local-exactness} Every $1$-form in $\polarF{1}$ is locally exact at the pole.
        \item\label{item:exactness}
        $H^0(\complex{\polar}) \isomorphic \RR$, $H^1(\complex{\polar}) = 0$, and $H^2(\complex{\polar}) = 0$.
    \end{enumerate}
\end{theorem}
\begin{proof}
    Except for \cref{item:local-linear-independence} and \cref{item:exactness}, we omit the details of the proofs and refer the reader to \cite{Toshniwal:2021}.
    However, we do note that these follow from the construction of the matrices $\mat{E}^{k}$ (\cref{app:extraction}) and standard properties of B-splines such as local support, linear independence, and, in particular, their smoothness: for $0 \leq r \leq \degV$, the $r$-jet of the tensor-product B-spline $\bsp_{i,\degU}^\per\bsp_{j,\degV}$ is non-zero if and only if $1 \leq j \leq r+1$.
    In particular, \cref{item:local-exactness} follows from \cref{lem:reduced-complex}.

    \noindent
    \textbf{Proof of \cref{item:local-linear-independence}.}\\
		Since the basis functions in $\bspB^{k,\tint}$ are locally linearly independent, and they only differ from $\polarB^k$ on the elements close to the pole, we can focus on those elements. Let us consider any element in the support of the polar functions $\polar^0_1$, $\polar^0_2$, or $\polar^0_3$, and let us denote by $I$ and $I^\tint$ the respective sets of basis functions in $\polarB^k$ and $\bspB^{k,\tint}$ not vanishing on that element. From \cref{app:extraction} we have that
	    \begin{equation*}
			\begin{split}
				\begin{bmatrix}
					\polar^k_i
				\end{bmatrix}_{i \in {I}} = \mat{C}^k
				\begin{bmatrix}
					\bsp^{0,\tint}_i
				\end{bmatrix}_{i \in {I}^\tint}\;.
			\end{split}
		\end{equation*}
		We will prove that $\mat{C}^k$ has full rank, which implies  that the polar functions are linearly independent on the element.
		 
		For $k = 0$, by construction of $\bspB^{0,\tint}$, there exist indices $i_1, i_2, i_3 \in I^\tint$ such that the functions $\bsp^{0,\tint}_{i_1}$, $\bsp^{0,\tint}_{i_2}$, and $\bsp^{0,\tint}_{i_3}$ are respectively identical to some B-splines $\bsp^{\per}_{j,\degU}\bsp_{2,\degV}$, $\bsp^{\per}_{j+1,\degU}\bsp_{2,\degV}$ and $\bsp^{\per}_{j+2,\degU}\bsp_{2,\degV}$, for some index $j \in \{1, \ldots, \pndofU \}$. Then, up to permutations, we can express $\mat{C}^0$ in the following block-diagonal form:
		\begin{equation*}
			\mat{C}^0 = 
			\begin{bmatrix}
				\widehat{\mat{C}}^0 & \\
				& \mat{I}
			\end{bmatrix}\;,
		\end{equation*}
	   with $\widehat{\mat{C}}^0$ a submatrix of $\widehat{\mat{E}}^0$ (see \cref{app:extraction}) with three rows that contains (at least) non-vanishing columns corresponding to $\bsp^{0,\tint}_{i_1}$, $\bsp^{0,\tint}_{i_2}$, and $\bsp^{0,\tint}_{i_3}$.
	   Since any such submatrix has full rank (because of our choice of control points $\mbf{F}_{ij}$, see \eqref{eq:control-points}), we have that $\mat{C}^0$ is full-rank, thus completing the proof for $k = 0$.
	   The proofs for $k = 1$ and $k = 2$ follow from the same reasoning.

    \noindent
    \textbf{Proof of \cref{item:exactness}.}\\
    We now present an alternative proof of \cref{item:exactness} that is different from \cite{Toshniwal:2021}, since the same argumentation will be useful in later sections.
    Consider the short exact sequence of complexes
    \begin{equation*}
        \begin{tikzcd}
            0 \arrow{r} & \complex{\polar} \arrow{r}{} & \complex{\bsp}^{,\tint} \arrow{r}{} & \complex{\bsp}^{,\tint} / \complex{\polar}  \arrow{r} & 0\;,
        \end{tikzcd}
    \end{equation*}
    where the first map is the inclusion and the second map is the canonical projection.
    This induces the following long exact sequence in cohomology \cite{Hatcher:2002}:
    \begin{equation*}
        \begin{tikzcd}
        H^0(\complex{\polar}) \arrow[d] & H^1(\complex{\polar}) \arrow[d] & H^2(\complex{\polar}) \arrow[d] \\
        H^0(\complex{\bsp}^{,\tint}) \arrow[d] \ar[draw=none]{r}[name=X, anchor=center]{}  & H^1(\complex{\bsp}^{,\tint}) \arrow[d] \ar[draw=none]{r}[name=Y, anchor=center]{}  & H^2(\complex{\bsp}^{,\tint}) \arrow[d] \\
        H^0(\complex{\bsp}^{,\tint} / \complex{\polar}) \ar[rounded corners,
        to path={ -- ([yshift=-2ex]\tikztostart.south)
                  -| (X.center) \tikztonodes
                  |- ([yshift=2ex]\tikztotarget.north)
                  -- (\tikztotarget)}]{ruu}[at end]{} & H^1(\complex{\bsp}^{,\tint} / \complex{\polar}) \ar[rounded corners,
                  to path={ -- ([yshift=-2ex]\tikztostart.south)
                            -| (Y.center) \tikztonodes
                            |- ([yshift=2ex]\tikztotarget.north)
                            -- (\tikztotarget)}]{ruu}[at end]{} & H^2(\complex{\bsp}^{,\tint} / \complex{\polar})
        \end{tikzcd}
    \end{equation*}
    From \cref{thm:intermediate-complex}, we have that $H^1(\complex{\bsp}^{,\tint} / \complex{\polar}) \isomorphic H^2(\complex{\polar})$ and $H^2(\complex{\bsp}^{,\tint} / \complex{\polar}) = 0$.

    First, we note that it is clear that $H^0(\complex{\polar}) \isomorphic \RR$, because only constants are in the kernel of $d : \polarF{0} \rightarrow \polarF{1}$.
    Next, since $\complex{\bsp}^{,\tint}$ is reduced at the pole (\cref{thm:intermediate-complex}), and since $\complex{\polar} \subset \complex{\bsp}^{,\tint}$ satisfies the conditions of the second part of \cref{lem:reduced-complex} (by \cref{item:limited-influence,item:local-exactness} of this theorem), we have that $H^2(\complex{\polar}) = 0$.

    Finally, since we already know the dimensions of $H^0(\complex{\polar})$ and $H^2(\complex{\polar})$, as well as the dimensions of $\polarF{k}$, $k = 0, 1, 2$, we can use the Euler characteristic of the complex $\complex{\polar}$ to compute the dimension of $H^1(\complex{\polar})$ as follows:
    \begin{equation*}
        \begin{split}
            \sum_{k=0}^2 (-1)^k \dim(H^k(\complex{\polar})) &= \sum_{k=0}^2 (-1)^k \dim(\polarF{k})\;,\\
            \Longrightarrow \dim(H^1(\complex{\polar})) &= 1 - n^0 + n^1 - n^2 = 1 - n^{0,\tint} + n^{1,\tint} - n^{2,\tint} = 0\;.
        \end{split}
    \end{equation*}
\end{proof}

The polar spline de Rham complex on $\polarDomain$ is now defined as,
\begin{equation}
    \begin{tikzcd}
        \complex{\overline{\polar}}~:~ \overline{\polarS}^{0} \arrow{r}{d} & \overline{\polarS}^{1} \arrow{r}{d} & \overline{\polarS}^{2}\;,
    \end{tikzcd}
\end{equation}
where we define the spaces as
\begin{equation}
    \overline{\polarS}^{k} := \left\{ f ~:~ \mbf{F}^* f  \in \polarS^{k} \right\}\;,\qquad k = 0, 1, 2\;,
\end{equation}
with $\mbf{F}^*$ as in \cref{eq:intermediate-polar-spaces-physical}.
As shown in \cite{Toshniwal:2021}, this complex is an exact subcomplex of the $L^2$ de Rham complex on $\polarDomain$.

\begin{remark}
    Note that the polar spline spaces $\polarF{k}$, $k = 0, 1, 2$, can be equivalently defined as
    \begin{equation}
        \polarF{k} = \mat{E}^k\mat{E}^{k,\tint}\tpF{k,\per}{\pkntsU,\kntsV}\;,\qquad k = 0, 1, 2\;.
    \end{equation}
    The total extraction operators $\mat{E}^k\mat{E}^{k,\tint}$ are then identical to those defined in \cite{Toshniwal:2021}.
\end{remark}

%% file: tikz/scripts/polar_map_horiz.tex
\begin{tikzpicture}[scale=0.4,line cap=round, line join=round]
	\def\L{6}
	\def\S{-12}

	\begin{scope}
		\draw [step=\L/4, black, ultra thin] (0,0) grid (\L, \L);
		\draw [myRed, ultra thick] (0,0) -- (\L, 0);
		\draw [black, ultra thick] (0,0) -- (0, \L);
		\draw [black, ultra thick] (\L,0) -- (\L, \L);
		
		\node (a) at (\L, \L/2){};
		
		\draw [black, ultra thin] (\L/2-\S, \L/2) circle (\L/2);
		\draw [black, ultra thin] (\L/2-\S, \L/2) circle (\L/4);
		\draw [black, ultra thin] (\L/2-\S, \L/2) circle (\L/8);
		\draw [black, ultra thin] (\L/2-\S, \L/2) circle (3*\L/8);
		\draw [black, ultra thin] (\L/2-\S, 0) -- (\L/2-\S, \L);
		\draw [black, ultra thin] (\L-\S, \L/2) -- (\L/2-\S, \L/2);
		\draw [black, ultra thick] (\L/2-\S, \L/2) -- (-\S, \L/2);
		\fill[myRed] (\L/2-\S, \L/2) circle (5pt);
		
		\node (b) at (-\S, \L/2){};
		
		\path [draw, -latex'] ($(a)!0.25!(b)$) -- ($(a)!0.75!(b)$); 
	\end{scope}
\end{tikzpicture}

%% file: sections/construction.tex
\section{Hierarchically-refinable spline differential forms}
\label{sec:construction}

In this section, we describe the construction of adaptively-refinable B-spline and polar spline differential forms.
We start by defining the hierarchical construction for nested sequences of spline differential forms in Section \ref{sec:construction-hierarchical}.
In Section \ref{sec:construction-linear-ind}, we show that the generating sets of the hierarchical spaces also form a basis, both for B-splines and polar splines.

\input{./sections/construction-hierarchical}
\input{./sections/construction-linear-ind}

%% file: sections/construction-hierarchical.tex
\subsection{Hierarchical spline de Rham complexes}
\label{sec:construction-hierarchical}

\subsubsection{A general construction of hierarchical spline spaces}
\label{sec:construction-hierarchical-general}

Consider a \emph{domain hierarchy}, i.e., a sequence of nested, closed domains,
\begin{equation}
    \tpDomain =: \tpDomain_0\supseteq \tpDomain_1 \supseteq \cdots \supseteq \tpDomain_{L}\;.
\end{equation}
Consider a sequence of $L+1$ nested $k$-form spline spaces, all defined on $\tpDomain$,
\begin{equation}
    \vSpace{X}_0^k \subset \vSpace{X}_1^k \subset \cdots \subset \vSpace{X}_L^k\;,
\end{equation}
and the corresponding vector of $k$-form spline basis functions, $\basis{X}^k_\ell := [X^k_{i,\ell} : i = 1, \dots, n^k_\ell]$.
Then, a hierarchical $k$-form spline space, denoted $\hSpace{X}^k$ is defined as the span of a collection of generating functions, which are picked according to the chosen domain hierarchy.

\begin{definition}[Hierarchical $k$-form spline space]
    The hierarchical $k$-form spline space $\hSpace{X}^k$ is defined to be the span of generating functions $\hBasis{X}^k$,
    \begin{equation}
        \hBasis{X}^k := \left[
            X_{i,\ell}^k~:~(i,\ell) \in \activeS{k,X}{}
        \right]\;,
    \end{equation}
    where $\activeS{}{}$ is called the \emph{active index-set for functions} and is defined recursively as below.
    \begin{enumerate}
        \item \emph{Initialization}: $\activeS{k,X}{0} := \{(i,0)~:~i = 1, \dots, n^k_0\}$.
        \item \emph{Recursive case}: For $\ell = 1,\dots,L$, $\activeS{k,X}{\ell} := \activeS{k,X}{\ell,0} \cup \activeS{k,X}{\ell,1}$, where
            \begin{equation}
                \begin{split}
                    \activeS{k,X}{\ell,0} &:= \{(i,\ell') \in \activeS{k,X}{\ell-1}~:~ \supp(X^k_{i,\ell'}) \not\subseteq \tpDomain_{\ell}\}\;,\\
                    \activeS{k,X}{\ell,1} &:= \{(i,\ell)~:~ \supp(X^k_{i,\ell}) \subseteq \tpDomain_{\ell}\}\;.
                \end{split}
            \end{equation}
        \item \emph{Finalization}: $\activeS{k,X}{} := \activeS{k,X}{L}$.
    \end{enumerate}
\end{definition}
In the following we will see applications of this definition to different choices of the sequence of spaces $\vSpace{X}^k_\ell$.
For instance, we will apply it to the sequence of intermediate B-spline spaces $\bspS^{k,\tint}_\ell$ and to the sequence of polar spline spaces $\polarS^k_\ell$ to construct the hierarchical B-spline spaces $\hSpace{\bsp}^{k,\tint}$ and the hierarchical polar spline spaces $\hSpace{\polar}^k$, respectively.

\subsubsection{Hierarchical B-spline and polar-spline complexes}
\label{sec:construction-hierarchical-specific}

In the following, we will specialize the above hierarchical construction to the context of B-splines and polar-spline $k$-forms.
Consider knot vectors $\kntsU_0, \dots, \kntsU_L$ and $\kntsV_0, \dots, \kntsV_L$, and their corresponding $\cU$-periodic B-spline spaces for $\ell = 0, \dots, L$ (recall \eqref{eq:periodic-bsplineForms}):
\begin{equation}
  \begin{split}
    \tpF{0,\per}{\pkntsU_\ell,\kntsV_\ell}_\ell &:= \bspS[\pkntsU_\ell,\kntsV_\ell]\;,\\
    \tpF{1,\per}{\pkntsU_\ell,\kntsV_\ell}_\ell &:= \bspS[d\pkntsU_\ell,\kntsV_\ell]\;\dcU + \bspS[\pkntsU_\ell,d\kntsV_\ell]\;\dcV\;,\\
    \tpF{2,\per}{\pkntsU_\ell,\kntsV_\ell}_\ell &:= \bspS[d\pkntsU_\ell,d\kntsV_\ell]\;\dcU\wedge \dcV\;,
  \end{split}
\end{equation}
Let us denote the corresponding B-spline basis functions as $\bspB^{k,\per}_\ell$, $k = 0, 1, 2$, $\ell = 0, \dots, L$.
We also introduce the corresponding intermediate B-spline spaces and basis functions:
\begin{equation}
    \bspS^{k,\tint}_\ell := \mat{E}^{k,\tint}_\ell\bspS^{k,\per}_\ell\;,\qquad
    \bspB^{k,\tint}_\ell := \mat{E}^{k,\tint}_\ell\bspB^{k,\per}_\ell\;,
\end{equation}
as well as the polar spline spaces and polar spline basis functions:
\begin{equation}\label{eq:polar-spline-levels}
    \polarS^k_\ell := \mat{E}^k_\ell\bspS^{k,\tint}_\ell\;,\qquad
    \polarB^k_\ell := \mat{E}^k_\ell\bspB^{k,\tint}_\ell\;,\qquad
    \ell = 0, \dots, L~~,~~k = 0, 1, 2\;,
\end{equation}
where $\mat{E}^{k,\tint}_\ell$ and $\mat{E}^k_\ell$ are the matrices defined in \cref{app:intermediate-extraction,app:extraction-refinement}, respectively.

\begin{assumption}[Nested knot vectors]\label{ass:knot-vectors-nested}
    The knot vectors are chosen such that all B-spline spaces form a nested sequence,
    \begin{equation}
        \bspS^{k,\per}_0 \subset \bspS^{k,\per}_1 \subset \cdots \subset \bspS^{k,\per}_L\;,\qquad k = 0, 1, 2\;.
    \end{equation}
\end{assumption}

\begin{proposition}
    Given Assumption 1, the intermediate B-spline spaces form the following nested sequence for $k = 0, 1, 2$:
    \begin{equation}
        \bspS^{k,\tint}_0 \subset \bspS^{k,\tint}_1 \subset \cdots \subset \bspS^{k,\tint}_L\;.
    \end{equation}
\end{proposition}
\begin{proof}
    The proof follows from the construction of the matrices $\mat{E}^{k,\tint}_\ell$, $k = 0, 1, 2$, as outlined in \cref{app:intermediate-extraction}.
    We discuss the case for $k = 0$ in detail, the other cases follow similarly.

	Consider an arbitrary level-$\ell$ spline $f = \sum_{i=1}^{n^{0,\tint}_\ell} f_{i,\ell}^\tint \bsp^{0,\tint}_{i,\ell} \in \bspS^{0,\tint}_\ell$.
    Then, using the definition of $\mat{E}^{0,\tint}_\ell$,
    \begin{equation*}
        f(\cU,\cV) = \sum\limits_{i=1}^{\pndofU_\ell}\sum_{j=1}^{\ndofV_\ell} f_{ij,\ell} \bsp^{\per}_{i,\degU,\ell}(\cU) \bsp_{j,\degV,\ell}(\cV)\;,
    \end{equation*}
    where
    \begin{equation*}
        f_{ij,\ell} = \begin{dcases}
            f^\tint_{1,\ell}\;, & i = 1, \dots, \pndofU_\ell\;,\;j = 1\;,\\
            f^\tint_{i + (j-2)\pndofU_\ell + 1,\ell}\;, & i = 1, \dots, \pndofU_\ell\;,\; j = 2, \dots, \ndofV_\ell\;.
        \end{dcases}
    \end{equation*}
    Then, let us define the coefficients $f_{ij,\ell+1}$ using the refinement matrix $\mat{K}^{\per}_\ell$ from \cref{app:extraction-refinement}:
    \begin{equation}\label{eq:two-scale-intermediate-1}
        f_{ij,\ell+1} :=
        \sum_{r=1}^{\pndofU_{\ell}}\sum_{s=1}^{\ndofV_{\ell}}K^{\cU}_{ir,\ell} f_{rs,\ell}K^{\cV}_{js,\ell} , \text{ for } i = 1, \dots, \pndofU_{\ell}, j = 1, \dots, \ndofV_\ell\;,
    \end{equation}
    such that the spline $f$ can also be expressed as:
    \begin{equation*}
        f(\cU,\cV) = \sum\limits_{i=1}^{\pndofU_{\ell+1}}\sum_{j=1}^{\ndofV_{\ell+1}} f_{ij,\ell+1} \bsp^{\per}_{i,\degU,\ell+1}(\cU) \bsp_{j,\degV,\ell+1}(\cV)\;.
    \end{equation*}
    Then, by the properties of B-spline knot insertion matrices \cite{deBoor:1978}, we can verify that:
    \begin{equation*}
        f_{1,\ell}^\tint = f_{i1,\ell+1}\;,\qquad i = 1, \dots, \pndofU_{\ell+1}\;.
    \end{equation*}
    Therefore, the coefficients $f_{i,\ell+1}^\tint$ defined as
    \begin{equation}\label{eq:two-scale-intermediate-2}
        f_{i,\ell+1}^\tint := \begin{dcases}
            f_{1,\ell}^\tint\;, & i = 1\;,\\
            f_{jk,\ell+1}\;, & i = j + (k-2)\pndofU_{\ell+1} + 1\;, j = 1, \dots, \pndofU_{\ell+1}\;, k = 2, \dots, \ndofV_{\ell+1}\;,\\
        \end{dcases}
    \end{equation}
    describe the same spline:
    \begin{equation*}
        f = \sum\limits_{i=1}^{n^{0,\tint}_{\ell+1}} f_{i,\ell+1} \bsp^{0,\tint}_{i,\ell+1}\;,
    \end{equation*}
    thus completing the proof.
\end{proof}

\begin{proposition}
    Given Assumption 1, the polar spline spaces form the following nested sequence for $k = 0, 1, 2$:
    \begin{equation}
        \polarS^k_0 \subset \polarS^k_1 \subset \cdots \subset \polarS^k_L\;.
    \end{equation}
\end{proposition}
\begin{proof}
    The proof follows from the construction of the matrices $\mat{E}^{k}_\ell$, $k = 0, 1, 2$, as outlined in \cref{app:extraction}.
    We discuss the cases for $k = 0$ in detail, the other cases follow similarly.

	Let $\mat{K}^{\tint}_\ell$ be the refinement matrix relating the spline coefficients of intermediate B-spline spaces at levels $\ell$ and $\ell+1$; this matrix is defined by \cref{eq:two-scale-intermediate-1,eq:two-scale-intermediate-2}.
    In particular, for a given spline $f = \sum_{i=1}^{n^{0,\tint}_\ell} f_{i,\ell}^\tint \bsp^{0,\tint}_{i,\ell} \in \bspS^{0,\tint}_\ell$, let $\widehat{\mat{K}}^{\tint}_\ell$ be its size $(\pndofU_{\ell+1}+1) \times (\pndofU_\ell+1)$ submatrix such that:
    \begin{equation*}
        \left[
            f^\tint_{i,\ell} : i = 1, \dots, \pndofU_{\ell}+1
        \right]
        \xmapsto{\widehat{\mat{K}}^{\tint}_\ell}
        \left[
            f^\tint_{i,\ell+1} : i = 1, \dots, \pndofU_{\ell+1}+1
        \right]\;,
    \end{equation*}
    Next, recall that the level-$\ell$ matrix $\mat{E}^0_{\ell}$ can be expressed as in \eqref{eq:polar-ext-0}:
    \begin{equation*}
        \mat{E}^0_{\ell} = 
    \begin{bmatrix}
    \widehat{\mat{E}}^0_{\ell} &\\
    & \mat{I}_{n^0_{\ell}-3}
    \end{bmatrix}\;.
    \end{equation*}
    In particular, the matrix $\widehat{\mat{E}}^0_{\ell}$ is computed using the control points:
    \begin{equation*}
        \left[
            \mbf{F}_{11,\ell}\;\;\mbf{F}_{12,\ell}\;\;\mbf{F}_{22,\ell}\;\;\cdots\;\;\mbf{F}_{\pndofU_{\ell}2,\ell}
        \right]\;,
    \end{equation*}
    while the matrix $\widehat{\mat{E}}^0_{\ell+1}$ is computed using the control points:
    \begin{equation*}
        \left[
            \mbf{F}_{11,\ell+1}\;\;\mbf{F}_{12,\ell+1}\;\;\mbf{F}_{22,\ell+1}\;\;\cdots\;\;\mbf{F}_{\pndofU_{\ell+1}2,\ell+1}
        \right]\;,
    \end{equation*}
    with the submatrix $\widehat{\mat{K}}^{\tint}_\ell$ identified above mapping the former set of control points to the latter.
    Consequently, the matrices $\widehat{\mat{E}}^0_{\ell}$ and $\widehat{\mat{E}}^0_{\ell+1}$ can be related as follows:
    \begin{equation}\label{eq:two-scale-polar-0}
        \begin{split}
            \widehat{\mat{E}}^0_{\ell+1} &= 
            \underbrace{\mat{L}\mat{R}_{\ell+1}\mat{R}^{-1}_\ell \mat{L}^{-1}}_{=: \mat{T}^{-T}_{\ell}} \widehat{\mat{E}}^0_{\ell} \widehat{\mat{K}}^T_{\ell}\;,
        \end{split}
    \end{equation}
    where $\mat{L}$, $\mat{R}_\ell$ and $\mat{R}_{\ell+1}$ are as defined in \cref{eq:polar-ext-0-core}.
    
    Next, consider an arbitrary spline $f = \sum_{i=1}^{n^0_\ell} f^\tpolar_{i,\ell} \polar^0_{i,\ell} \in \polarS^0_\ell$.
    Let us define coefficients $f^\tpolar_{i,\ell+1}$ as follows,
    \begin{equation}\label{eq:two-scale-polar-1}
        f^\tpolar_{i,\ell+1} =
        \begin{dcases}
            \sum_{j=1}^3 T_{ij,\ell}f^\tpolar_{j,\ell}\;,& i = 1, \dots, 3\;,\\
            \sum_{j=1}^{n^{0,\tint}_\ell} K^{\tint}_{ij,\ell} \sum_{k=1}^{n^{0}_\ell} E^0_{kj,\ell}f^\tpolar_{k,\ell}\;,&\text{otherwise}\;.
        \end{dcases}
    \end{equation}
    where $T_{ij,\ell}$ and $K^{\tint}_{ij,\ell}$ are the $ij$-th entries of $\mat{T}_{\ell}$ and $\mat{K}^{\tint}_{\ell}$, respectively.
    Then, with the coefficients $f^\tpolar_{i,\ell}$ and $f^\tpolar_{i,\ell+1}$ arranged in vectors $\mbf{f}^\tpolar_\ell$ and $\mbf{f}^\tpolar_{\ell+1}$, respectively, we can verify that the splines on the left- and right-hand sides are identical, thus completing the proof,
    \begin{equation*}
        \begin{split}
        \left(\mbf{f}^\tpolar_\ell\right)^T
        \begin{bmatrix}
            \widehat{\mat{E}}^0_{\ell} &\\
            & \mat{I}_{n^0_{\ell}-3}
        \end{bmatrix}
        \bspB^{0,\tint}_\ell
        &=
        \left(\mbf{f}^\tpolar_\ell\right)^T
        \begin{bmatrix}
            \widehat{\mat{E}}^0_{\ell} &\\
            & \mat{I}_{n^0_{\ell}-3}
        \end{bmatrix}
        \left(\mat{K}^{\tint}_{\ell}\right)^T\bspB^{0,\tint}_{\ell+1}\;,\\
        &=
        \left(\mbf{f}^\tpolar_{\ell+1}\right)^T
        \begin{bmatrix}
            \widehat{\mat{E}}^0_{\ell+1} &\\
            & \mat{I}_{n^0_{\ell+1}-3}
        \end{bmatrix}
        \bspB^{0,\tint}_{\ell+1}\;.
        \end{split}
    \end{equation*}
\end{proof}

\begin{remark}
    The above proof also provides an explicit description of the two-scale relations relating polar spline degrees of freedom at levels $\ell$ and $\ell+1$, see \cref{eq:two-scale-polar-0,eq:two-scale-polar-1}.
    Similar relations can be derived for the $1$- and $2$-form cases.
\end{remark}

Following the nestedness results, we can readily apply the hierarchical construction to both B-splines and polar splines.
For the B-spline case, we will denote the result of applying the hierarchical construction to the sequence of spaces $\bspS^{k,\tint}_{\ell}$ as $\hSpace{\bsp}^{k,\tint}$.
Similarly, for the polar spline case, we will denote the result of applying the hierarchical construction to the sequence of spaces $\polarS^k_\ell$ as $\hSpace{\polar}^k$.
We will denote the corresponding active index-sets as $\activeS{k,\bsp^\tint}{}$ and $\activeS{k,\polar}{}$, respectively.
The hierarchical B-spline complex will be denoted as $\complex{\hSpace{\bsp}}^{,\tint}$,
\begin{equation}
    \begin{tikzcd}
        \complex{\hSpace{\bsp}}^{,\tint}~:~ \hSpace{\bsp}^{0,\tint} \arrow{r}{d} & \hSpace{\bsp}^{1,\tint} \arrow{r}{d} & \hSpace{\bsp}^{2,\tint}\;,
    \end{tikzcd}
\end{equation}
and the hierarchical polar spline complex will be denoted as $\hpolarComplex$,
\begin{equation}
    \begin{tikzcd}
        \hpolarComplex~:~ \hSpace{\polar}^0 \arrow{r}{d} & \hSpace{\polar}^1 \arrow{r}{d} & \hSpace{\polar}^2\;.
    \end{tikzcd}
\end{equation}
Let \(n^k\) denote the number of functions in $\hBasis{\polar}^k$, $k = 0, 1, 2$, and denote them as $\hBasis{\polar}^k = [\hbsp{\polar}^k_i : i = 1, \dots, n^k]$.

Finally, and as in the case of the other spaces introduced thus far, the hierarchical polar spline de Rham complex on $\polarDomain$ can now be defined as,
\begin{equation}
    \begin{tikzcd}
        \complex{\overline{\hSpace{\polar}}}~:~ \overline{\hSpace{\polar}}^{0} \arrow{r}{d} & \overline{\hSpace{\polar}}^{1} \arrow{r}{d} & \overline{\hSpace{\polar}}^{2}\;,
    \end{tikzcd}
\end{equation}
where we define the spaces on the physical domain $\polarDomain$ as
\begin{equation}
    \overline{\hSpace{\polar}}^{k} := \left\{ f ~:~ \mbf{F}^* f  \in \hSpace{\polar}^{k} \right\}\;,\qquad k = 0, 1, 2\;.
\end{equation}
These spaces are later used in \cref{sec:numerics} for numerical experiments.

%% file: sections/construction-linear-ind.tex
\subsection{Linear independence of the generating functions}
\label{sec:construction-linear-ind}

We will now show that, for any $k = 0, 1, 2$, the generating functions of the hierarchical B-spline and polar spline spaces, whose indices are respectively contained in the sets $\activeS{k,\bsp}{}$ and $\activeS{k,\polar}{}$, are linearly independent.
We will do so under an additional assumption on the domain hierarchy.

\begin{assumption}[Domain hierarchy from $0$-form refinement]\label{ass:refinement-domains}
    For $\ell=1, \dots, L$, we assume that $\tpDomain_{\ell}$ is given as the union of
	supports of polar spline $0$-forms.
    That is, for any $\ell \in \{1, \dots, L\}$, there exists $I_{\ell-1} \subseteq \{1, \dots, n^0_{\ell-1}\}$ such that
    \begin{equation}
        \tpDomain_{\ell} = \bigcup_{i \in I_{\ell-1}}\suppwp{P^0_{i,\ell-1}}\;.
    \end{equation}
\end{assumption}

The above assumption immediately implies linear independence of the hierarchical B-splines in $\hSpace{\bsp}^{k,\tint}$.
\begin{theorem}
    \label{thm:bspline-linear-ind}
    For all $k = 0, 1, 2$, Assumptions 1-2 imply that the spline functions in $\hBasis{\bsp}^{k,\tint}$ are linearly independent and, thus, form a basis for $\hSpace{\bsp}^{k,\tint}$.
\end{theorem}
\begin{proof}
    The linear independence of the $k$-form functions in $\hBasis{\bsp}^{k,\tint}$ follows from the local linear independence of the functions in single-level spaces $\bspS^{k,\tint}_\ell$; this follows from \cref{lem:local-linear-independence-intermediate} and the results in \cite{Giannelli:2012}.
\end{proof}

In fact, Assumptions 1-2 also imply linear independence of the polar spline functions in $\hSpace{\polar}^k$.
We show this in Theorem \ref{thm:polar-linear-ind} below, but first present two intermediate results.
Both results follow simply from the local support of the splines considered, and we only present the proof of the second one.

\begin{lemma}\label{lem:active-intermediate-basis}
    There is a unique level $0 \leq \ell \leq L$ such that $\bsp^{0,\tint}_{1,\ell}$ is active.
    Moreover, except the function identified above, every other active function in $\hBasis{\bsp}^{k,\tint}$ is identical to some level-$\ell'$ tensor-product B-spline $k$-form in $\bspB^{k,\per}_{\ell'}$,
     $k = 0, 1, 2$.
\end{lemma}
\begin{proof}
    The claims follow from the definitions of the spaces $\bspS^{k,\tint}_\ell$, and from the fact that $\suppwp{\bsp^{0,\tint}_{1,\ell}} = [0,1] \times [0, \cV_{\degV_\ell+2, \ell}]$, i.e., it is formed by all elements touching the bottom boundary of $\tpDomain$.
\end{proof}

\begin{remark}\label{rem:unique-level}
    The unique level $\ell$ mentioned in the statement of Lemma \ref{lem:active-intermediate-basis} is the level of the coarsest element touching the bottom boundary of $\tpDomain$.
\end{remark}

\begin{lemma}\label{lem:active-polar-basis}
    Given Assumptions 1-2, let $\ell$ be the unique level mentioned in the statement of Lemma \ref{lem:active-intermediate-basis}.
    Then, the following claims hold at the bottom boundary of $\tpDomain$, $\cV = 0$.
    \begin{itemize}
        \item There are only $3$ functions in $\hBasis{\polar}^0$, all from level $\ell$, which have non-zero values and first derivatives.
        \item There are only $2$ functions in $\hBasis{\polar}^1$, all from level $\ell$, which have non-zero values.
        \item All functions in $\hBasis{\polar}^2$ have zero values.
    \end{itemize}
Moreover, except the level-$\ell$ functions identified above, every other function in $\hBasis{\polar}^k$ is identical to some tensor-product B-spline $k$-form in $\hBasis{\bsp}^{k,\tint}$, $k = 0, 1, 2$.
\end{lemma}
\begin{proof}
    We prove the claim for $k = 0$, the other cases can be shown analogously.
    For any level $\ell$, consider the first three functions in $\polarB^0_\ell$.
    For all three, the closure of their individual support is equal to $[0,1] \times [0, \cV_{\degV_\ell+3, \ell}]$.
    This, in combination with Assumptions 1-2, implies the following: 
    \begin{equation*}
        [0,1] \times [0, \cV_{\degV_{\ell+1}+3, \ell+1}] \subset \tpDomain_{\ell+1} \Longleftrightarrow [0,1] \times [0, \cV_{\degV_\ell+3, \ell}] \subset \tpDomain_{\ell+1}\;.
    \end{equation*}
	The leftwards implication is trivial from Assumption~1, while the rightwards follows from Assumption~2.
    This implies that $\{(1,\ell+1),(2,\ell+1),(3,\ell+1)\} \subset \activeS{0,\polar}{}$ only if $(i,\ell) \notin \activeS{0,\polar}{}$ for any $i \in \{1, 2, 3\}$.
    The claim follows.
\end{proof}

\begin{corollary}\label{cor:similarities-hb-hp}
    Given Assumptions 1-2, the following claims hold:
    \begin{align}
        \hSpace{\bsp}^{0,\tint} \ni f \;,\; \jet{1}{f} = 0 &\Longrightarrow f \in \hSpace{\polar}^0\;,\\
        \hSpace{\bsp}^{1,\tint} \ni f = f_\cU\dcU + f_\cV\dcV \;,\; \jet{1}{f_\cU} = 0 = \jet{0}{f_\cV} &\Longrightarrow f \in \hSpace{\polar}^1\;,\\
        \left\{\hSpace{\bsp}^{2,\tint} \ni f = f_{\cU\cV}\dcU\wedge\dcV \;,\; \jet{0}{f_{\cU\cV}} = 0\right\} &= \hSpace{\polar}^2\;.
    \end{align}
\end{corollary}

\begin{theorem}\label{thm:polar-linear-ind}
    For all $k = 0, 1, 2$, Assumptions 1-2 imply that the spline functions in $\hBasis{\polar}^k$ are linearly independent and, thus, form a basis for $\hSpace{\polar}^k$.
\end{theorem}
\begin{proof}
    The linear independence of the $k$-form functions in $\hBasis{\polar}^k$ follows from the local linear independence of the functions in single-level spaces $\polarS^{k}_\ell$, that was proved in \cref{item:local-linear-independence} of \cref{thm:polar-splines}, and the results in \cite{Giannelli:2014}.
\end{proof}

%% file: sections/exactness_polar.tex
\section{Exactness of the adaptively-refinable polar-spline complex}
\label{sec:exactness-polar}
In this section, we examine the cohomology groups of the constructed hierarchical polar-spline complex and show that it is exact.
We do so by first studying the intermediate complex $\hintComplex$ and then using the obtained results to characterize the cohomology of $\hpolarComplex$.

\input{./sections/exactness_hb}

\subsection{Exactness of the complex $\hpolarComplex$}
The zero-th cohomology of $\hpolarComplex$ can be characterized easily and we do so in the following result.
\begin{proposition}\label{prop:h0h1}
    $\dim H^0(\hpolarComplex) = 1$.
\end{proposition}
\begin{proof}
    If $f \in \hSpace{\polar}^0$ is such that $df = 0$, then $f$ must be constant on all of $\tpDomain$, and so the zero-th cohomology is one-dimensional.
\end{proof}

For characterizing $H^1(\hpolarComplex)$ and $H^2(\hpolarComplex)$, our main tool will be the following short exact sequence of complexes: $\hpolarComplex \rightarrow  \hintComplex \rightarrow \hintComplex / \hpolarComplex$.
This implies the following long exact sequence connecting the cohomologies of the different complexes:
\begin{equation}\label{eq:longExactSequence}
    \begin{tikzcd}
        H^0(\hpolarComplex) \arrow[d] & H^1(\hpolarComplex) \arrow[d] & H^2(\hpolarComplex) \arrow[d] \\
        H^0(\hintComplex) \arrow[d] \ar[draw=none]{r}[name=X, anchor=center]{}  & H^1(\hintComplex) \arrow[d] \ar[draw=none]{r}[name=Y, anchor=center]{}  & H^2(\hintComplex) \arrow[d] \\
        H^0(\hintComplex / \hpolarComplex) \ar[rounded corners,
        to path={ -- ([yshift=-2ex]\tikztostart.south)
                  -| (X.center) \tikztonodes
                  |- ([yshift=2ex]\tikztotarget.north)
                  -- (\tikztotarget)}]{ruu}[at end]{} & H^1(\hintComplex / \hpolarComplex) \ar[rounded corners,
                  to path={ -- ([yshift=-2ex]\tikztostart.south)
                            -| (Y.center) \tikztonodes
                            |- ([yshift=2ex]\tikztotarget.north)
                            -- (\tikztotarget)}]{ruu}[at end]{} & H^2(\hintComplex / \hpolarComplex)
    \end{tikzcd}
\end{equation}
We will approach the desired characterization in the same way as in the proof of \cref{item:exactness} in \cref{thm:polar-splines}.
That is, we first address $H^2(\hpolarComplex)$ and then $H^1(\hpolarComplex)$.

In particular, since $\hintComplex$ is reduced at the pole, we immediately get the desired result for the second cohomology of the polar-spline complex.
\begin{proposition}\label{prop:h2}
    Given Assumptions 1-3, $H^2(\hpolarComplex) = 0$.
\end{proposition}
\begin{proof}
    \cref{lem:intermediate-reduced-complex} shows that $\hintComplex$ is reduced at the pole.
    Moreover, by \cref{cor:similarities-hb-hp}, $\hpolarComplex$ is a subcomplex of $\hintComplex$ that satisfies the conditions laid out in \cref{lem:reduced-complex}.
    Thus, the result follows immediately from \cref{lem:reduced-complex} and the fact that $H^2(\hintComplex) = 0$ from \cref{cor:exactness-intermediate-complex}.
\end{proof}

Finally, using a counting argument near the pole, we can now show that $H^1(\hpolarComplex) = 0$.
\begin{proposition}
    Given Assumptions 1-3, $\dimwp{H^1(\hpolarComplex)} = 0$.
\end{proposition}
\begin{proof}
    As in the proof of \cref{item:exactness} in \cref{thm:polar-splines}, we have that:
    \begin{equation*}
        \begin{split}
            \dimwp{H^1(\hpolarComplex)} &= \dimwp{H^0(\hintComplex/\hpolarComplex)}\;\\
            &= \dimwp{\hSpace{\bsp}^{0,\tint}/\hSpace{\polar}^0} - \dimwp{\hSpace{\bsp}^{1,\tint}/\hSpace{\polar}^1} + \dimwp{\hSpace{\bsp}^{2,\tint}/\hSpace{\polar}^2}\;.
        \end{split}
    \end{equation*}
    Then, using \cref{cor:similarities-hb-hp,lem:intermediate-reduced-complex}, we see that the intermediate hierarchical spaces and the polar spline spaces are identical away from the pole.
    Therefore, the dimensions of the quotient spaces $\hSpace{\bsp}^{k,\tint}/\hSpace{\polar}^k$ only depend on the differences between the spaces near the poles.
    For instance, by ignoring all splines $f$ for which $\jet{1}{f} = 0$, we can deduce that the dimension of $\hSpace{\bsp}^{0,\tint}/\hSpace{\polar}^0$ is simply $(1 + \dimwp{\hSpace{V}_\cU^0}) - 3$.
    By similar arguments, we can also deduce that the dimensions of $\hSpace{\bsp}^{1,\tint}/\hSpace{\polar}^1$ and $\hSpace{\bsp}^{2,\tint}/\hSpace{\polar}^2$ are $\dimwp{\hSpace{V}_\cU^1} + \dimwp{\hSpace{V}_\cU^0} - 2$ and $\dimwp{\hSpace{V}_\cU^1}$, respectively.
    Therefore,
    \begin{equation*}
        \dimwp{H^1(\hpolarComplex)} =
        \left( 1 + \dimwp{\hSpace{V}_\cU^0} - 3 \right)
        - \left( \dimwp{\hSpace{V}_\cU^1} + \dimwp{\hSpace{V}_\cU^0} - 2 \right)
        + \left( \dimwp{\hSpace{V}_\cU^1} \right) = 0\;.
    \end{equation*}
\end{proof}

We collect the above results in the following main result of this paper.
\begin{theorem}\label{thm:exactness}
    Given Assumptions 1-3, the hierarchical polar-spline complex $\hpolarComplex$ is exact, i.e.,
    \begin{equation*}
        H^0(\hpolarComplex) \isomorphic \RR, \quad H^1(\hpolarComplex) = H^2(\hpolarComplex) = 0\;.
    \end{equation*}
\end{theorem}

%% file: sections/exactness_hb.tex
\subsection{Exactness of the complex $\complex{\hSpace{\bsp}}^{,\tint}$}
\label{sec:exactness-hb}

In this section, we present some intermediate results on the complex $\hintComplex$, and place an additional assumption that is required for guaranteeing the right cohomological structure of the complex.
For this intermediate complex, consider the unique level $\ell$ that is referenced in \cref{lem:active-intermediate-basis}, and consider the following sets of active function indices:
\begin{equation}\label{eq:reduced-active-sets}
    \small
    \begin{split}
        \ractiveS{0,\bsp^\tint}{} &:= \left\{
            (i,\ell') \in \activeS{0,\bsp^\tint}{} ~:~ \jet{1}{\bsp^{0,\tint}_{i,\ell'}} \neq 0
        \right\}\backslash \{ (1,\ell) \}\;,\\
        \ractiveS{1,\bsp^\tint}{} &:= \left\{
            (i,\ell') \in \activeS{1,\bsp^\tint}{} ~:~
            \bsp^{1,\tint}_{i,\ell'} = \bsp^{1,\tint}_{i,\ell',\cU}\;\dcU
            \text{ and }
            \jet{1}{\bsp^{1,\tint}_{i,\ell',\cU}} \neq 0
        \right\} \\
        &\qquad\bigcup
        \left\{
            (i,\ell') \in \activeS{1,\bsp^\tint}{} ~:~
            \bsp^{1,\tint}_{i,\ell'} = \bsp^{1,\tint}_{i,\ell',\cV}\;\dcV
            \text{ and }
            \jet{0}{\bsp^{1,\tint}_{i,\ell',\cV}} \neq 0
        \right\}\;,\\
        \ractiveS{2,\bsp^\tint}{} &:= \left\{
            (i,\ell') \in \activeS{2,\bsp^\tint}{} ~:~
            \bsp^{2,\tint}_{i,\ell'} = \bsp^{2,\tint}_{i,\ell',\cU\cV}\;\dcU\wedge\dcV
            \text{ and }
            \jet{0}{\bsp^{2,\tint}_{i,\ell',\cU\cV}} \neq 0
        \right\}\;.
    \end{split}
\end{equation}

\begin{corollary}\label{cor:reduced-intermediate-basis}
    Given Assumptions 1-2 and the unique level $\ell$ in \cref{lem:active-intermediate-basis}, the sets $\ractiveS{0,\bsp^\tint}{}$, $\ractiveS{1,\bsp^\tint}{}$, and $\ractiveS{2,\bsp^\tint}{}$ are such that:
    \begin{equation*}
        (i, \ell') \in \ractiveS{k,\bsp^\tint}{} \Longrightarrow \ell' \geq \ell\;.
    \end{equation*}
\end{corollary}
\begin{proof}
    The proof follows from \cref{lem:active-intermediate-basis} and \cref{rem:unique-level}, since $\ell$ is the level of the coarsest element touching the pole.
\end{proof}

Consider a sub-complex of $\hintComplex$, denoted $\hintComplexZ$, defined by imposing zero boundary conditions at the polar edge -- this only modifies the space of 0-forms in the complex by eliminating the basis function corresponding to the index $(1,\ell) \in \activeS{0,\bsp^\tint}{}$.
In fact, the complex $\hintComplexZ$ coincides with the hierarchical B-spline complex with zero boundary conditions imposed at the polar edge and periodicity across the vertical edges.
Then, as was done in \cite{Evans:2020,Shepherd:2024}, it can be shown that this complex is exact if the hierarchical mesh satisfies certain conditions.
We utilize the conditions from \cite{Shepherd:2024}, and state them here in the following simplified form in our two-dimensional setting.
\begin{definition}[Minimal intersection \cite{Shepherd:2024}]\label{def:minimal-intersection}
    For any level $\ell$, consider two inactive B-splines vanishing at the pole, such that:
    \begin{equation*}
        \begin{split}
            &\bsp^{0}_{i,\ell}, \bsp^{0}_{j,\ell} \in \{\bsp \in \bspB^{0,\per}_{\ell}~:~\jet{0}{\bsp} = 0\} = \bspB^{0,\per}_{\ell} \cap \bspB^{0,\tint}_{\ell}\;,\\
            &\suppwp{\bsp^{0}_{i,\ell}}, \suppwp{\bsp^{0}_{j,\ell}} \subset \tpDomain_{\ell+1}\;.
        \end{split}
    \end{equation*}
    Then, we say that $\bsp^{0}_{i,\ell}$ and $\bsp^{0}_{j,\ell}$ share a minimal intersection if there exists $\bsp^{2}_{k,\ell+1} \in \bspB^{2,\per}_{\ell+1} = \bspB^{2,\tint}_{\ell+1}$ satisfying the following conditions:
    \begin{equation*}
            \begin{drcases*}
                \suppwp{\bsp^{0}_{i,\ell}} \cap \suppwp{\bsp^{0}_{j,\ell}} =: S_{\cU,\ell} \times S_{\cV,\ell}\\
                \suppwp{\bsp^{2}_{k,\ell+1}} =: S_{\cU,\ell+1} \times S_{\cV,\ell+1}
            \end{drcases*}
            \qquad\text{ such that }\qquad
            \begin{dcases*}
                S_{\cU,\ell+1} \subseteq S_{\cU,\ell}\;,\\
                \text{or}\\
                S_{\cV,\ell+1} \subseteq S_{\cV,\ell}\;.
            \end{dcases*}
    \end{equation*}
\end{definition}

\begin{assumption}[Shortest chains \cite{Shepherd:2024}]\label{ass:shortest-chains}
	For any level $\ell$, consider two inactive B-splines $\bsp^{0}_{i_1,\ell}, \bsp^{0}_{i_r,\ell}$ that share a minimal intersection.
    Then we assume that there exists a shortest chain of inactive B-splines $\bsp^{0}_{i_j,\ell}$, $j=1,\ldots,r$, between them.
    That is:
    \begin{itemize}
        \item $\bsp^{0}_{i_j,\ell} = \bsp^\per_{s_j,\degU,\ell}\bsp_{t_j,\degV,\ell} \in \bspB^{0,\per}_{\ell} \cap \bspB^{0,\tint}_{\ell}$ for all $j=1,\ldots,r$;
        \item $|(s_1,t_1) - (s_{r},t_{r})| = r-1$, and
        $|(s_j,t_j) - (s_{j+1},t_{j+1})| = 1$ for all $j=1,\ldots,r-1$;
        
        \item $\suppwp{\bsp^{0}_{i_j,\ell}} \subset \tpDomain_{\ell+1}$ for all $j=1,\ldots,r$.
    \end{itemize}
\end{assumption}
\begin{remark}
    Note that, because of the periodicity of the tensor-product B-splines in $\bspB^{0,\per}_{\ell} \cap \bspB^{0,\tint}_{\ell}$, the indices $s_1$ and $t_1$ in Assumption 3 are understood in a periodic sense.
    That is, $\pndofU_\ell + k$ is identified with $k$ for all $k \in \NN$.
\end{remark}

Then, by repeating the local arguments from \cite{Shepherd:2024}, it can be shown that the complex $\hintComplexZ$ is exact; we state this result in the following proposition without detailing the proof.
\begin{proposition}
    Given Assumptions 1-3, the complex $\hintComplexZ$ is exact:
    \begin{equation*}
        H^0(\hintComplexZ) = H^1(\hintComplexZ) = H^2(\hintComplexZ) = 0\;.
    \end{equation*}
\end{proposition}
\begin{proof}
    The proof uses the same arguments as in \cite{Shepherd:2024}.
    In particular, one only needs to show that the cohomology of the B-spline complex that is deactivated at level $\ell$ is isomorphic to the cohomology of the B-spline complex that is activated at level $\ell+1$.
    The structure of the boundary conditions in the current setting does not affect this argument and (up to book-keeping) the proof is identical to the one in \cite{Shepherd:2024}.
\end{proof}

Then, utilizing the following short exact sequence:
\begin{equation*}
    \begin{tikzcd}
        0 \arrow{r} & \hintComplexZ \arrow{r}{} & \hintComplex \arrow{r}{} & \hintComplex / \hintComplexZ  \arrow{r} & 0\;,
    \end{tikzcd}
\end{equation*}
and the long sequence as in Theorem~\ref{thm:polar-splines}, we immediately get the following result.
\begin{corollary}\label{cor:exactness-intermediate-complex}
    Given Assumptions 1-3, the complex $\hintComplex$ has the following cohomological structure:
    \begin{equation*}
        H^0(\hintComplex) \isomorphic \RR\;,
        H^1(\hintComplex) = H^2(\hintComplex) =  0\;.
    \end{equation*}
\end{corollary}

As a final result for the complex $\hintComplex$, we show that it is reduced at the pole.
This will simplify the proof of the exactness of $\hpolarComplex$ in the next section by allowing us to use the same arguments as in the polar spline case discussed in Section \ref{sec:preliminaries-polar}.
\begin{lemma}\label{lem:intermediate-reduced-complex}
    The complex $\hintComplex$ is reduced at the pole.
\end{lemma}
\begin{proof}	
    Let us define the univariate B-spline spaces, $\hSpace{V}_\cU^0$ and $\hSpace{V}_\cU^1$, as follows:
    \begin{equation}\label{eq:reduced-spaces-at-pole}
        \small
        \begin{split}
            \hSpace{V}_\cU^0 &:= \spaceSpan{
                \left\{
                    \bsp_{i_1} : \bsp^{0,\tint}_{i,\ell'} = \bsp_{i_1,\degU,\ell}\bsp_{i_2,\degV,\ell}\text{ for some }(i,\ell') \in \ractiveS{0,\bsp^\tint}{}\;
                \right\}
            }\;,\\
            \hSpace{V}_\cU^1 &:= \spaceSpan{
                \left\{
                    \bsp_{i_1} : \bsp^{1,\tint}_{i,\ell'} = \bsp_{i_1,\degU,\ell}\bsp_{i_2,\degV,\ell}\;\dcU\text{ for some }(i,\ell') \in \ractiveS{1,\bsp^\tint}{}
                \right\}
            }\;,
        \end{split}
    \end{equation}
    where we recall that $\ractiveS{0,\bsp^\tint}{}$ and $\ractiveS{1,\bsp^\tint}{}$ are defined in \cref{eq:reduced-active-sets}.
    Moreover, and analogously to \cref{thm:intermediate-complex}, we define the spaces:
    \begin{equation*}
        \begin{split}
            \check{\vSpace{V}}^0 &:= \spaceSpan{B~:~B \in \hBasis{\bsp}^{0,\tint}
            ~~\text{and}~~
            \jet{1}{B} = 0}\;,\\
            \check{\vSpace{V}}^{11} &:= \spaceSpan{B~:~B\;\dcU \in \hBasis{\bsp}^{1,\tint}
            ~~\text{and}~~
            \jet{1}{B} = 0}\;,\\
            \check{\vSpace{V}}^{12} &:= \spaceSpan{B~:~B\;\dcV \in \hBasis{\bsp}^{1,\tint}
            ~~\text{and}~~
            \jet{0}{B} = 0}\;,\\
            \check{\vSpace{V}}^2 &:= \spaceSpan{B~:~B\;\dcU\wedge\dcV \in \hBasis{\bsp}^{2,\tint}
            ~~\text{and}~~
            \jet{0}{B} = 0}\;.
        \end{split}
    \end{equation*}
    Note that these spaces are different from those in \cref{thm:intermediate-complex}; we reuse the notation simply to establish a connection with \cref{def:reduced-complex}.
    Then, with $g(\cV) := \bsp_{2,\degV,\ell}(\cV)$ and level $\ell$ as in \cref{lem:active-intermediate-basis}, we have that the spaces $\hSpace{\bsp}^{0,\tint}$, $\hSpace{\bsp}^{1,\tint}$, and $\hSpace{\bsp}^{2,\tint}$ can be expressed in the following form:
    \begin{equation*}
        \begin{split}
            \hSpace{\bsp}^{0,\tint} &= \RR \oplus g(\cV)\hSpace{V}_\cU^0 \oplus \check{\vSpace{V}}^0\;,\\
            \hSpace{\bsp}^{1,\tint} &= \left(g(\cV)\hSpace{V}_\cU^1 \oplus \check{\vSpace{V}}^{11}\right)\dcU
                            + \left(\hSpace{V}_\cU^01_\cV \oplus \check{\vSpace{V}}^{12}\right)\dcV\;,\\
            \hSpace{\bsp}^{2,\tint} &= \left(\hSpace{V}_\cU^1 1_\cV \oplus \check{\vSpace{V}}^2\right)\dcU \wedge \dcV\;.
        \end{split}
    \end{equation*}
    
    For example, consider $\hSpace{\bsp}^{0,\tint}$ and let $B \in \hBasis{\bsp}^{0,\tint}$ such that $\jet{0}{B} = 0$ and $\jet{1}{B} \neq 0$.
    With \cref{cor:reduced-intermediate-basis} in mind, $B$ must be a B-spline from level $\ell' \geq \ell$.
    By nestedness of the level-wise spaces (particularly in the coordinate $\cV$), $B$ can be expressed as $B_1 + B_2$ where $B_1 \in g(\cV)\hSpace{V}_\cU^0$ with $\jet{1}{B_1} = \jet{1}{B}$, and $B_2 \in \check{\vSpace{V}}^0$ with $\jet{1}{B_2} = 0$.

    The reasoning is similar for the other spaces and, finally, it is easy to check that \cref{def:reduced-complex} applies.
\end{proof}

%% file: sections/numerics.tex
\section{Numerical results}
\tikzexternaldisable
\label{sec:numerics}

In the numerical tests that follow, there are some common configurations we point out here:

\begin{itemize}
    \item The domain used in all examples will be as defined in \eqref{eq:polar_map}. This
        implies that the polar domain \(\Omega\) will be an approximation of the disk.
        Although it is possible to exactly represent circular domains using generalized
        Tchebycheffian splines, it falls outside the scope of the present work. 
    \item All knot vectors are uniform, and the associated B-splines are maximally smooth.
    \item The polynomial degrees for the whole complex are indicated by the degrees
        \(p^{\theta}, p^{\rho}\) of the 0-form space --- in the polar and radial directions,
        respectively.
    \item Between each pair of consecutive levels, the knot vectors of each dimension are
        uniformly subdivided by a factor of 2.
    \item Whenever a visualization of a differential form is shown, it will be in terms of
        musical isomorphisms that map differential forms to proxy fields. For example, a
        1-form \(u=u_{x}dx+u_{y}dy\) is visualized as the vector field
        \(u^{\sharp}\coloneqq \left(g^{11}u_{x}+g^{12}u_{y}\right)\partial_{x} +
        \left(g^{21}u_{x}+g^{22}u_{y}\right)\partial_{y}\), where \(dx,dy\) denote the
        differentials in the polar domain, \(g\) the metric tensor associated with the polar
        map described in \Cref{subsec:polar-map}, and \(\partial_x, \partial_y\) the
        canonical basis vectors in \(\RR^2\). Recall that this map is degenerate at the
        pole, so we exclude this point from the visualizations.
    \item The quadrature rule used is Gauss-Legendre with \((p^{\theta}+1) \times
        (p^{\rho}+1)\) points. 
    \item All examples are derived via Hodge Laplacians, as described in \cite[Section
        3.2]{Arnold:2010}. The \(L^2\) inner product of differential
        forms \(\langle u, v\rangle\), as used in the weak formulations, is defined by
        \[
            \langle u, v \rangle_{\Omega} \coloneqq \int_{\Omega} u \wedge \star v.
        \]
        We also define the \(L^2\) inner product over the boundary \(\partial \Omega\) of
        \(\Omega\):
        \[
            \langle u, v \rangle_{\partial \Omega} \coloneqq \int_{\partial\Omega}
            \trace_{\Omega} u \wedge \trace_{\Omega} \star v.
        \]
\end{itemize}

All the tests were implemented using the \emph{Julia} package \texttt{Mantis.jl}
\cite{Cabanas:2026}.

\subsection{Harmonics}
\label{subsec:spurious-harmonics}

The first example demonstrates a non-structure-preserving discretization to motivate the
importance of preserving cohomology. 
The chosen local refinement pattern deliberately breaks the cohomology of the continuous
complex. This illustrates how arbitrary refinement strategies can lead to spurious results,
and how the structure-preserving refinements differ between the polar and tensor-product
settings.

The problem of interest is the mixed weak formulation of the 1-form Hodge Laplacian: Find
\((\sigma, u) \in \overline{\hSpace{\polar}}^0 \times \overline{\hSpace{\polar}}^1\) such
that 
\begin{equation}
    \label{eq:1-form-wf}
    \begin{cases}
        \langle \tau, \sigma\rangle_{\Omega} - \langle d\tau, u\rangle_{\Omega} = 0, &
    \forall \tau \in \overline{\hSpace{\polar}}^0, \\
    \langle v, d\sigma \rangle_{\Omega} + \langle dv,  du \rangle_{\Omega} =  \langle
        v, f \rangle_{\Omega}, & \forall v \in \overline{\hSpace{\polar}}^1,
    \end{cases}
\end{equation}
with weak imposition of a null trace for \(u\) as this is enough to exhibit spurious
harmonics.

For this case, we set \(p^\theta=p^\rho=3\), and the forcing term to be \(f=dx + dy\).
To demonstrate spurious harmonics, we first solve \eqref{eq:1-form-wf} on a locally
refined polar mesh which is not cohomology-preserving.
Next, the problem is solved again on a modified mesh that avoids the spurious results.
A discussion of both these cases is presented in \Cref{fig:harmonics-example}.

In addition to the visual indication in \Cref{fig:harmonics-example-problematic-solution}
that cohomology is not preserved, we can also assert this numerically.
The system in \eqref{eq:1-form-wf} can be computationally assembled into the following
system of block-matrices, each representing an integral in the weak formulation:
\begin{align*}
    \begin{bmatrix*}
        A & -B^T \\
        B & C
    \end{bmatrix*}
    \begin{bmatrix*}
        \sigma \\
        u
    \end{bmatrix*}
    =
    \begin{bmatrix}
        0 \\
        F
    \end{bmatrix}.
\end{align*}
For the 1-form Hodge Laplacian, the relevant cohomology space is
\(H^1(\complex{\overline{H\polar}})\).
From \Cref{thm:exactness}, we expect that a structure-preserving discretization satisfies
$H^1(\complex{\overline{H\polar}}) = 0$. However, due to our choice of refinement in \Cref{fig:harmonics-example-problematic-tp}, we observe
\[
    \dim H^1(\complex{\overline{H\polar}}) = \dim \overline{\hSpace{\polar}}^1 - \rank B -
    \rank C = 1,
\]
where the expression above results from the Hodge decomposition of
\(\overline{\hSpace{\polar}}^1\).

\begin{figure}[htbp]
    \centering
    \hfill
    \def\sideLength{0.3\textwidth}
    \begin{subfigure}{\sideLength}
        \centering
        \includegraphics[width=\textwidth]{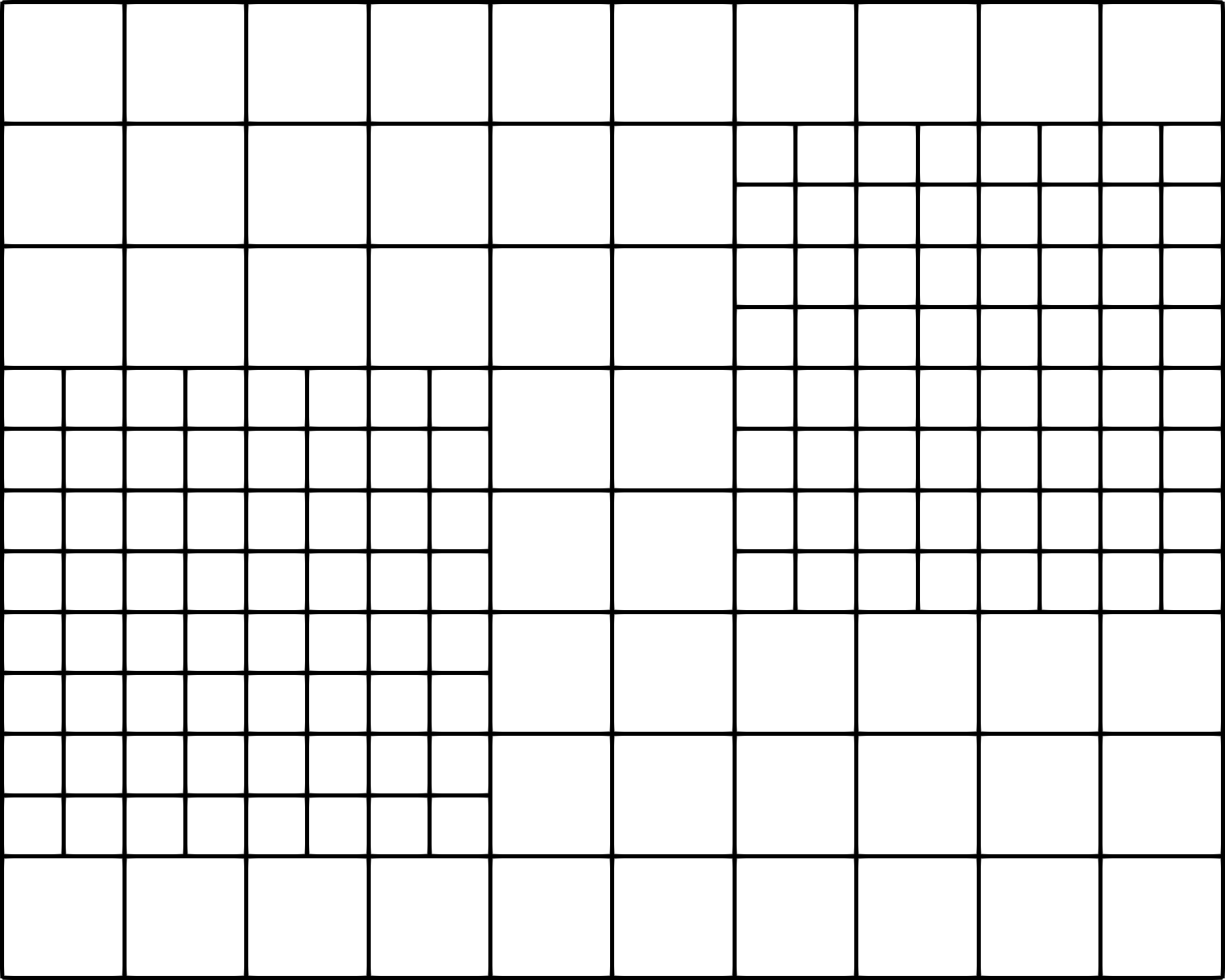}
        \caption{Tensor-product mesh.}
        \label{fig:harmonics-example-problematic-tp}
    \end{subfigure}
    \hfill
    \begin{subfigure}{\sideLength}
        \centering
        \includegraphics[width=\textwidth]{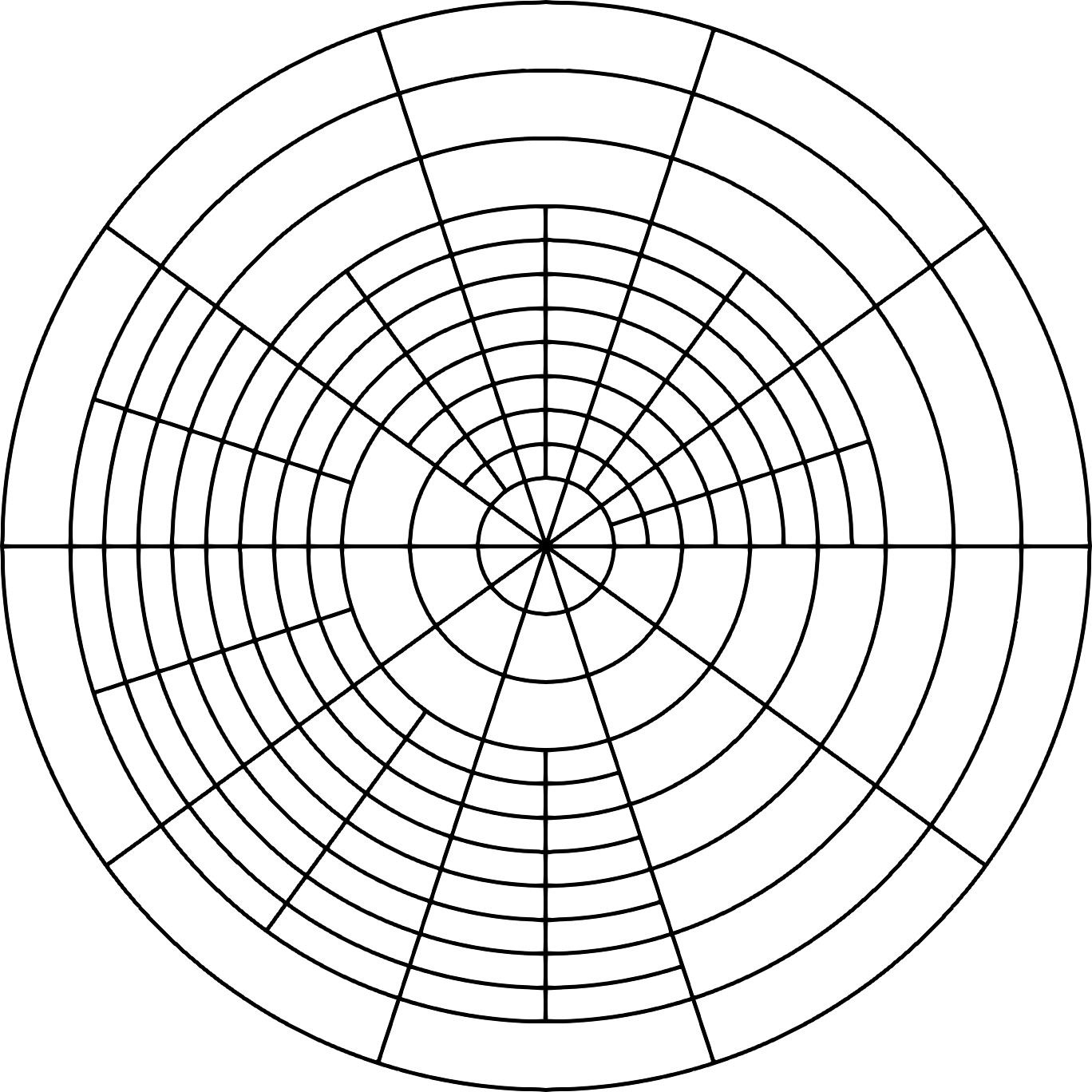}
        \caption{Polar mesh.}
    \end{subfigure}
    \hfill
    \begin{subfigure}{\sideLength}
        \centering
        \includegraphics[width=\textwidth]{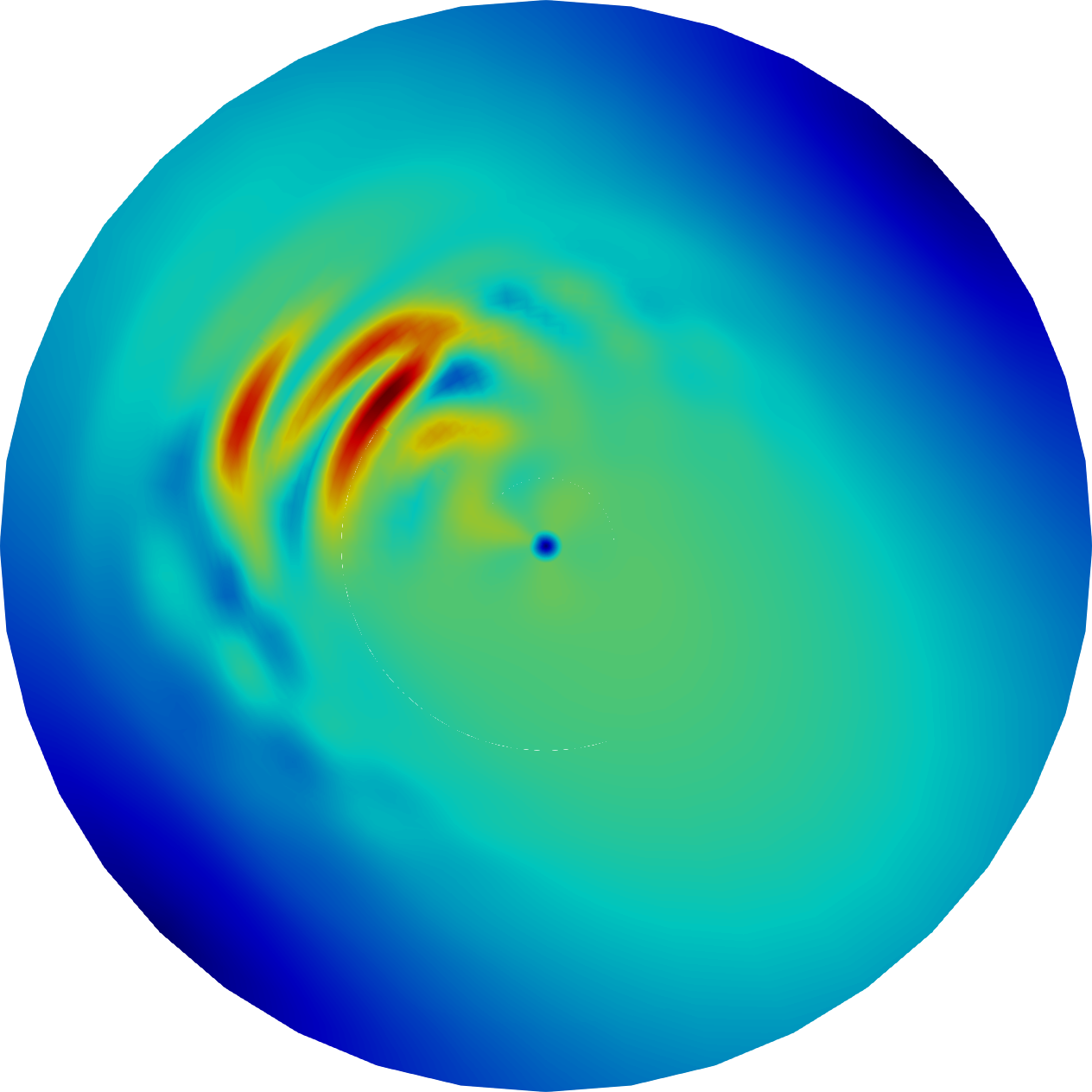}
        \caption{Computed solution.}
        \label{fig:harmonics-example-problematic-solution}
    \end{subfigure}
    \hfill \strut\\
    \hfill
    \begin{subfigure}{\sideLength}
        \centering
        \includegraphics[width=\textwidth]{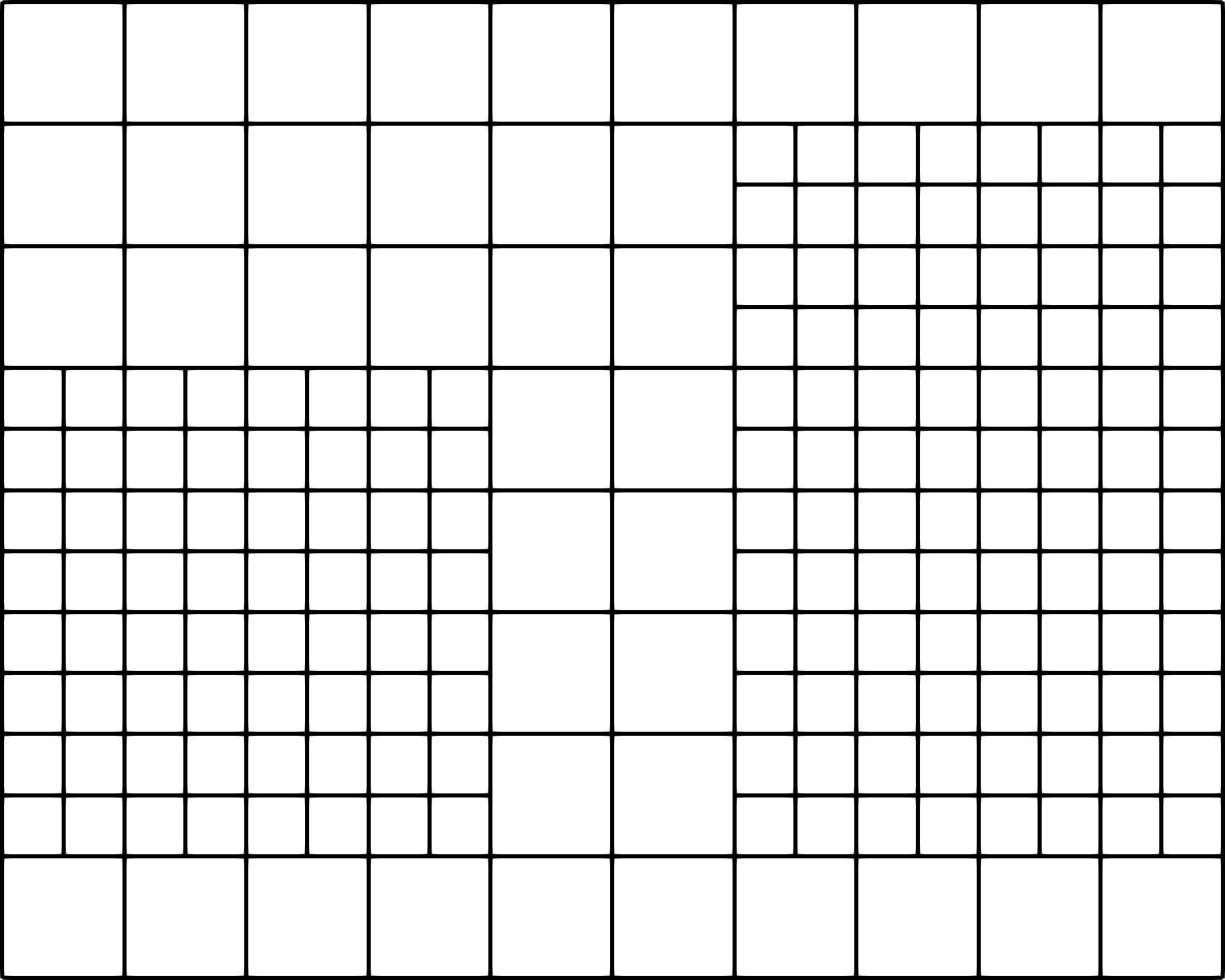}
        \caption{Tensor-product mesh.}
    \end{subfigure}
    \hfill
    \begin{subfigure}{\sideLength}
        \centering
        \includegraphics[width=\textwidth]{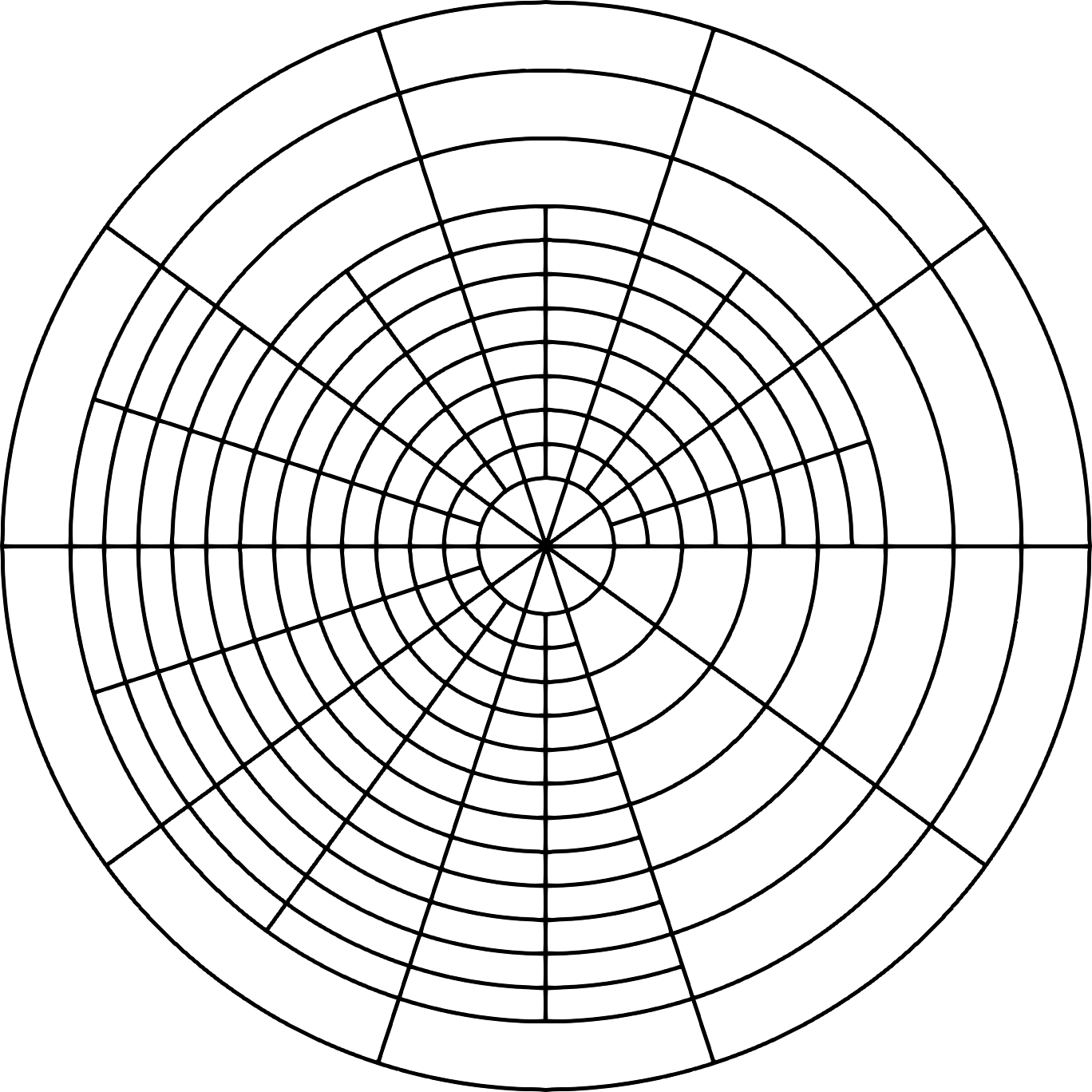}
        \caption{Polar mesh.}
    \end{subfigure}
    \hfill
    \begin{subfigure}{\sideLength}
        \centering
        \includegraphics[width=\textwidth]{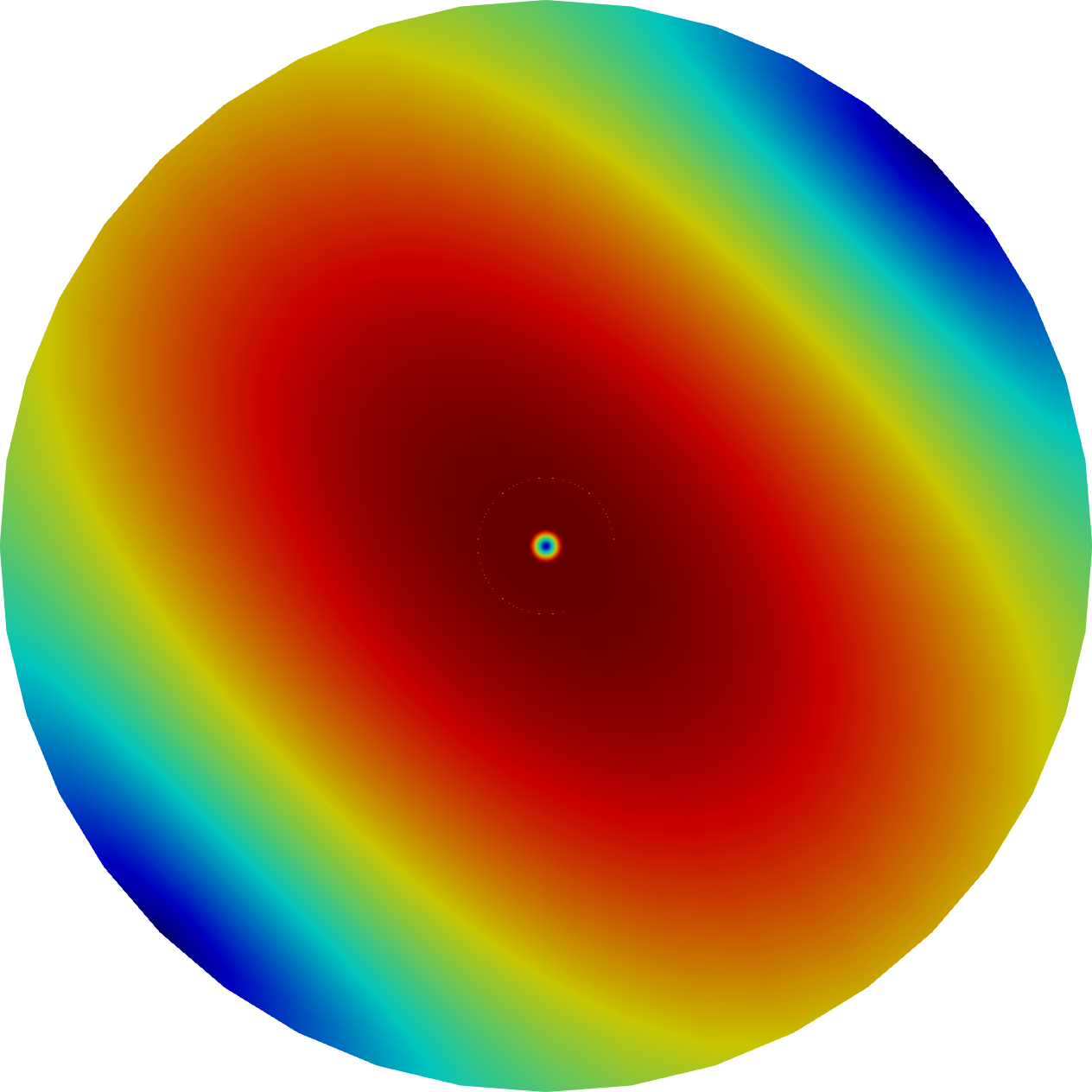}
        \caption{Computed solution.}
    \end{subfigure}
    \hfill \strut
    \caption{
        Comparison of two solutions of \eqref{eq:1-form-wf} on a cohomology-breaking and a
        cohomology-preserving mesh --- on the top and bottom rows, respectively.
        On the left, the tensor-product meshes corresponding to the polar meshes in the
        center. On the right, the magnitude of the computed \(u^{\sharp}\).
        Notice that the mesh in \Cref{fig:harmonics-example-problematic-tp} would be
        structure-preserving, in a tensor-product setting, if not for the introduction of
        periodicity.
        The loss of accuracy associated with the broken cohomology is clearly visible in the
        top-left quadrant of \Cref{fig:harmonics-example-problematic-solution}.
    }
    \label{fig:harmonics-example}
\end{figure}

\subsection{Maxwell eigenvalue}
\label{subsec:maxwell}

A class of PDEs where structure-preservation is of special importance are Maxwell's
equations.
These are often used as a prototypical example of how cohomology-breaking discretizations
can ruin numerical accuracy.
For this reason, in the current subsection we will solve a Maxwell eigenvalue-type problem.
The weak formulation \eqref{eq:maxwell-wf} is written with differential forms, but it can be
expressed using vector proxies as a curl-curl eigenvalue problem.

The solutions to this problem are well understood, and can be grouped in two different
types: transverse electric (TE) and transverse magnetic (TM) modes; the distinction between
the two modes can be summarised in the boundary conditions imposed and the direction of
propagation. 
Since the problem is formulated in two spatial dimensions, it corresponds to a dimensional
reduction of the three-dimensional problem; particularly, under the assumption of invariance
in out-of-plane direction.
As such, the solutions can be decoupled into either TE or TM modes, and no TEM modes are
computed. The latter require a three-dimensional formulation on a cylindrical domain.

The weak formulation of interest is given by the following equation:
Find \((u, \lambda) \in \vSpace{V} \times \RR\) such that 
\begin{equation}
    \label{eq:maxwell-wf}
    \langle dv, du \rangle_{\Omega} = \lambda^2 \langle v, u \rangle_{\Omega},\quad \forall
    v \in \vSpace{V},
\end{equation}
where \(\vSpace{V} \coloneqq  \overline{\hSpace{\polar}}^{1}_0\) for TE modes, and  \(\vSpace{V}
\coloneqq  \overline{\hSpace{\polar}}^{1}\) for TM modes.
Note that in the former we impose \(\trace_{\partial \Omega} u = \trace_{\partial \Omega} v=
0\) essentially, and in the latter \(\trace_{\partial \Omega} \star du = 0\) weakly. It is
well known that the eigenvalue solutions of \eqref{eq:maxwell-wf} correspond to the positive
zeros of Bessel functions of the first kind, for TM modes, and their derivatives, for TE
modes. 
We will denote the former by \(z_{n,i}\) and the latter by \(z'_{n,i}\) --- in relation to
the Bessel function \(J_n\).

The problem is discretized using \(p^\theta = p^\rho = 3\), and the initial mesh has
\(80\times20\) elements in the polar and radial directions. Subsequently, we construct a
truncated hierarchical polar spline space by performing local refinement to create a spiral
pattern; the final mesh can be seen in \Cref{fig:maxwell-example}. 
This particular choice of refinement is arbitrary, and is intended only to verify the
correct behavior of the approximation.
A summary of the computed results can be seen in \Cref{tab:maxwell-example}. 

The mesh shows that it is possible to refine only some of the elements adjacent to the pole.
Despite the local refinement, no spurious eigenmodes appear, and both TE and TM modes
are correctly approximated. This indicates that cohomology is preserved. Moreover, as
discussed in \Cref{sec:numerics}, the geometry itself is an approximation of the unit disk.
Therefore, besides the discretization error, there is a geometric approximation error that
should also be considered.
Once the geometric error dominates, further mesh refinement is insufficient to achieve
eigenvalue convergence.

\begin{figure}[htbp]
    \centering
    \begin{subfigure}{0.45\textwidth}
        \centering
        \includegraphics[width=\textwidth]{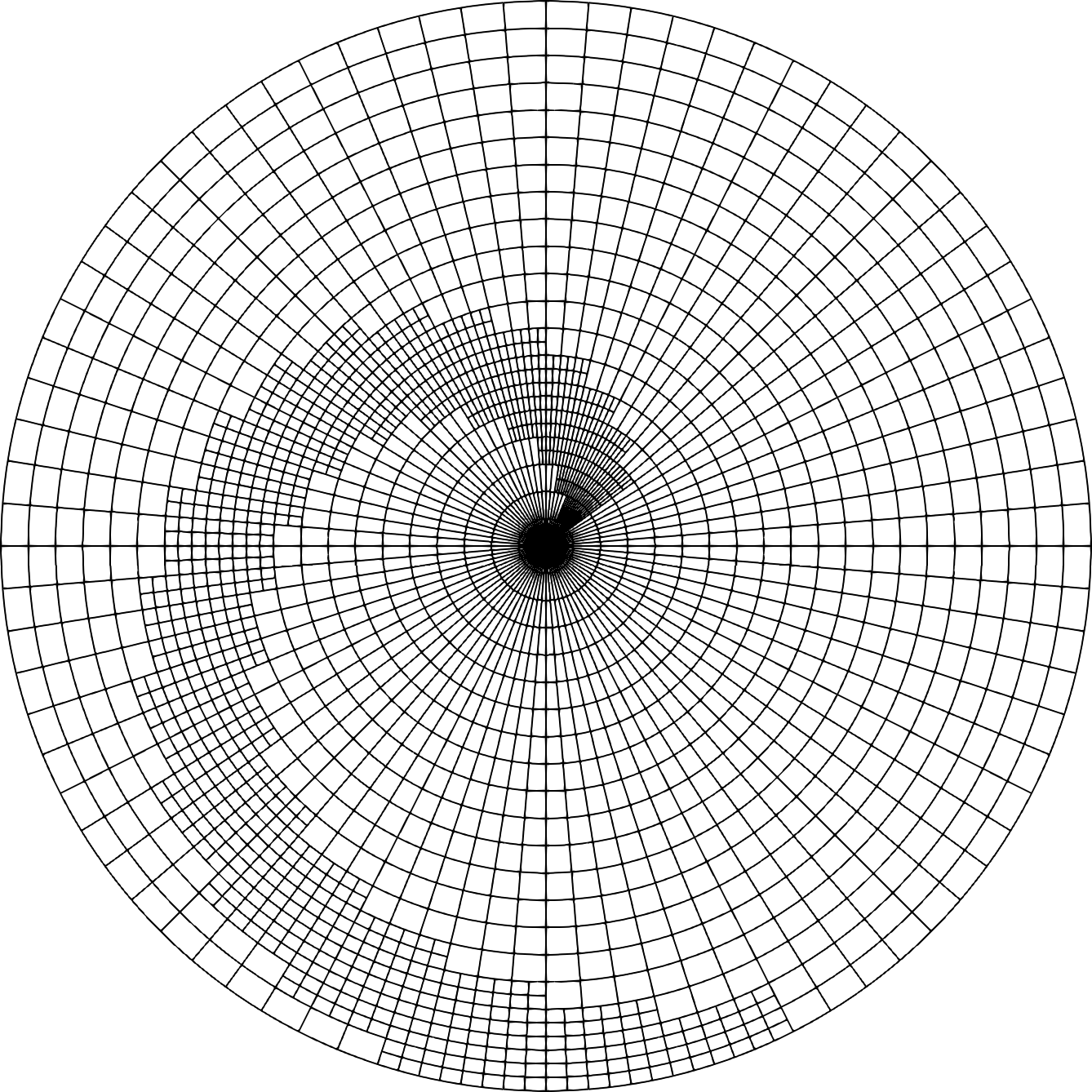}
        \caption{Full mesh}
        \label{fig:maxwell-full}
    \end{subfigure}
    \hfill
    \begin{subfigure}{0.45\textwidth}
        \centering
        \includegraphics[width=\textwidth]{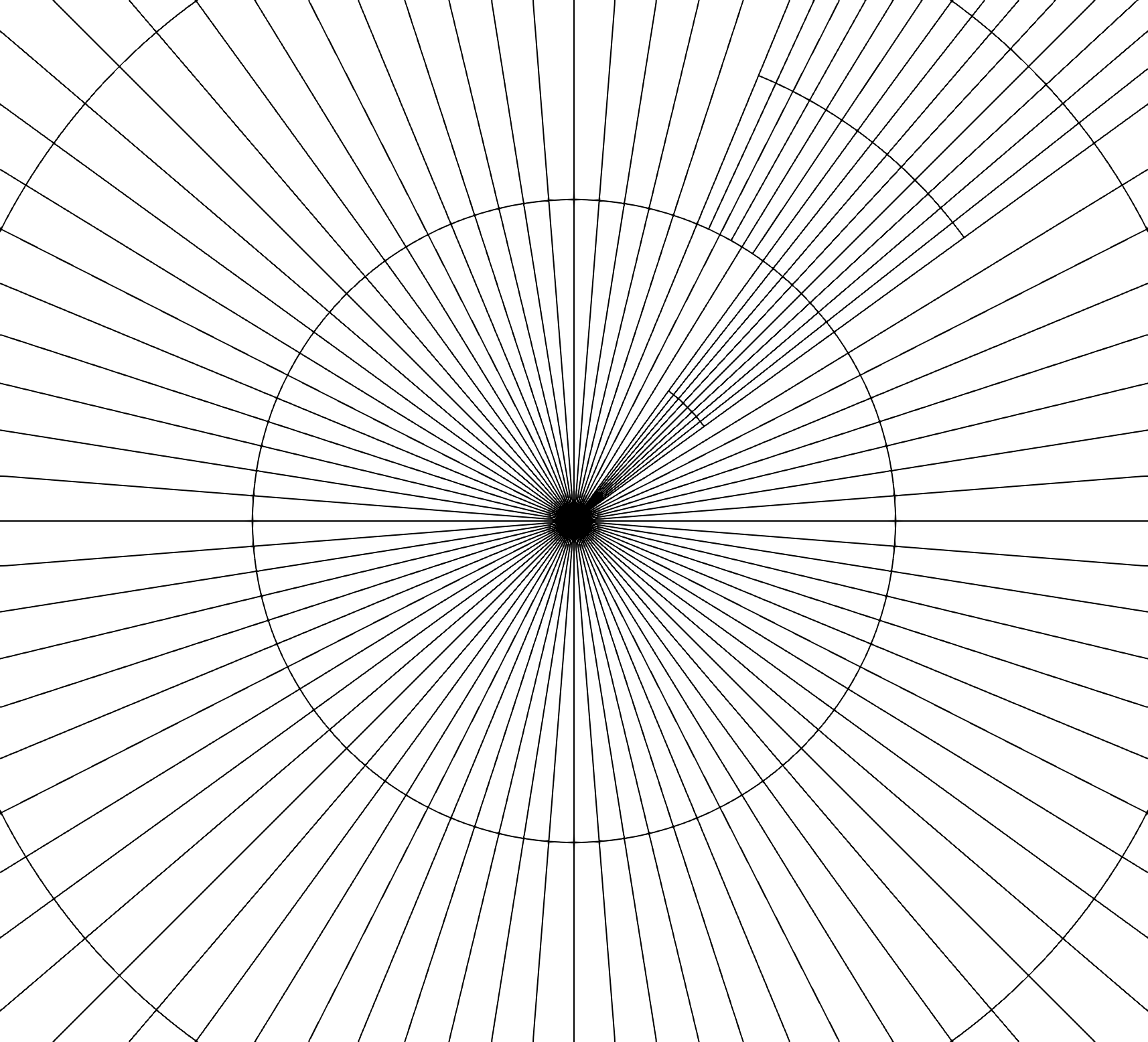}
        \caption{Zoomed behavior near the pole.}
        \label{fig:maxwell-zoom}
    \end{subfigure}
    \caption{
        The locally refined mesh used to compute the solution of \eqref{eq:maxwell-wf}.
        Note that only some of the elements adjacent to the pole are refined.
    }
    \label{fig:maxwell-example}
\end{figure}

\def\storeTE#1#2{%
    \expandafter\xdef\csname TE@#1\endcsname{#2}%
}
\def\storeTM#1#2{%
    \expandafter\xdef\csname TM@#1\endcsname{#2}%
}
\begin{filecontents*}{figures/eigvals-TE.csv}
eigvals
1.843077743398
1.843077743399
3.05737871923
3.057378719431
3.83564754064
4.205510570306
4.205510570561
5.323023180498
5.323023181035
5.336927097343
\end{filecontents*}
\csvreader{figures/eigvals-TE.csv}{eigvals=\val}{%
    \storeTE{\thecsvrow}{\val}%
}
\begin{filecontents*}{figures/eigvals-TM.csv}
eigvals
2.407299320024
3.835647518869
3.835647519478
5.140905182165
5.140905184644
5.525756653526
6.386724989385
6.386724989775
7.022803672999
7.022803700814
\end{filecontents*}
\csvreader{figures/eigvals-TM.csv}{eigvals=\val}{%
    \storeTM{\thecsvrow}{\val}%
}
\newcommand{\eigenrowTE}[3]{%
    $\lambda_{#1}$ & \csname TE@#1\endcsname & #2 & #3 \\
}
\newcommand{\eigenrowTM}[3]{%
    $\lambda_{#1}$ & \csname TM@#1\endcsname & #2 & #3 \\
}

\begin{table}[htbp]
	\centering
    \hfill
    \begin{subtable}[b]{0.48\textwidth}
        \centering
        \renewcommand{\arraystretch}{1.2}
        \begin{tabular}{
            c 
            S[table-format=1.5, round-mode=places, round-precision=5]
            S[table-format=1.5, round-mode=places, round-precision=5]
            c
        }
            \toprule
            \multicolumn{4}{c}{\textbf{TE}} \\
            \midrule
            {\(\lambda_{i}\)} & {Computed} & {Reference} & {\(z'_{n,i}\)}\\
            \midrule
            \eigenrowTE{1}{1.84118}{\(z'_{1,1}\)}
            \eigenrowTE{2}{1.84118}{\(z'_{1,1}\)}
            \eigenrowTE{3}{3.05424}{\(z'_{2,1}\)}
            \eigenrowTE{4}{3.05424}{\(z'_{2,1}\)}
            \eigenrowTE{5}{3.83171}{\(z'_{0,1}\)}
            \eigenrowTE{6}{4.20119}{\(z'_{3,1}\)}
            \eigenrowTE{7}{4.20119}{\(z'_{3,1}\)}
            \eigenrowTE{8}{5.31755}{\(z'_{4,1}\)}
            \eigenrowTE{9}{5.31755}{\(z'_{4,1}\)}
            \eigenrowTE{10}{5.33144}{\(z'_{1,2}\)}
            \bottomrule
        \end{tabular}
        \label{tab:maxwell-te}
    \end{subtable}%
    \begin{subtable}[b]{0.48\textwidth}
        \centering
        \renewcommand{\arraystretch}{1.2}
        \begin{tabular}{
            c
            S[table-format=1.5, round-mode=places, round-precision=5]
            S[table-format=1.5, round-mode=places, round-precision=5]
            c
        }
            \toprule
            \multicolumn{4}{c}{\textbf{TM}} \\
            \midrule
            {\(\lambda_{i}\)} & {Computed} & {Reference} & {\(z_{n,i}\)}\\
            \midrule
            \eigenrowTM{1}{2.40483}{\(z_{0,1}\)}
            \eigenrowTM{2}{3.83171}{\(z_{1,1}\)}
            \eigenrowTM{3}{3.83171}{\(z_{1,1}\)}
            \eigenrowTM{4}{5.13562}{\(z_{2,1}\)}
            \eigenrowTM{5}{5.13562}{\(z_{2,1}\)}
            \eigenrowTM{6}{5.52008}{\(z_{0,2}\)}
            \eigenrowTM{7}{6.38016}{\(z_{3,1}\)}
            \eigenrowTM{8}{6.38016}{\(z_{3,1}\)}
            \eigenrowTM{9}{7.01559}{\(z_{1,2}\)}
            \eigenrowTM{10}{7.01559}{\(z_{1,2}\)}
            \bottomrule
        \end{tabular}
        \label{tab:maxwell-tm}
    \end{subtable}
    \hfill\strut
    \caption{Eigenvalues computed from \eqref{eq:maxwell-wf} for TE and TM modes.}
    \label{tab:maxwell-example}
\end{table}

\subsection{n-Form Hodge Laplacian}
\label{subsec:n-form-hl}

For the final numerical test, we run a convergence study for two different polynomial
degrees, and compare them against global refinement. The weak formulation chosen is
associated with the volume-form Hodge Laplacian:
Find \((u, \phi) \in \overline{\hSpace{\polar}}^1 \times \overline{\hSpace{\polar}}^2\) such
that 
\begin{equation}
    \label{eq:n-form-wf}
    \begin{cases}
        \langle v, u \rangle_{\Omega} - \langle dv , \phi \rangle_{\Omega} =
        - \langle v, \phi_{\text{ex}} \rangle_{\partial \Omega}, & \forall v \in
    \overline{\hSpace{\polar}}^1,
    \\
        \langle \tau, du \rangle_{\Omega} =  \langle \tau, f \rangle_{\Omega}, & \forall \tau
    \in \overline{\hSpace{\polar}}^2,
    \end{cases}
\end{equation}
where \(f\) is a forcing term such that \(\phi\) converges to the exact solution:
\begin{equation*}
    \phi_{\text{ex}} \coloneqq \tanh\left(10((x-1)^2 + (y-1)^2 - 2)\right)dx\wedge dy\;,    
\end{equation*}
as is shown in \Cref{fig:n-form-analytical-solution}.
Note that the right-hand side of the first equation in \eqref{eq:n-form-wf} corresponds to a
natural boundary condition on \(\phi\).

The spaces \(\overline{\hSpace{\polar}}^1\) and \(\overline{\hSpace{\polar}}^2\) are first
initialized on a polar mesh with \(32 \times 16\) elements in the polar and radial
directions with \(p^\theta = p^\rho = 2\), and no refinement. Next, we execute 20 steps of
an adaptive refinement scheme, as illustrated in \Cref{fig:adaptive-loop}, with a marking
parameter \(\vartheta=0.20\). The parameter is used to define which basis functions are
marked for refinement as follows:
\begin{enumerate}
    \item The error at element \(\Omega_e\) is computed as \(E_e \coloneqq \|\phi -
        \phi_{\text{ex}}\|_{L^2(\Omega_e)}\) --- since the exact solution is known.
    \item The error of basis function \(P^0_{i, \ell}\) is computed as 
        \(
            (E(P^0_{i, \ell}))^2 \coloneqq 
            \sum\limits_{\Omega_e \subseteq \supp P^0_{i,
            \ell}}\left(E_e\int_{\Omega_e}
            \star P^0_{i, \ell}\right)^2.
        \)
    \item The maximum error is defined as 
        \(
            E_{\text{max}} \coloneqq \max_{i, \ell} E(P^0_{i, \ell}).
        \)
    \item Basis functions are marked for refinement using \(\vartheta\)
        \(
            R \coloneqq\left\{P^0_{i, \ell}: E(P^0_{i, \ell}) \geq
            (1-\vartheta)E_{\text{max}}\right\}.
        \)
    \item The supports of basis functions in \(R\) are refined.
\end{enumerate}

\Cref{fig:n-form-convergence} shows that locally-refined polar splines converge faster than
globally-refined polar splines in the pre-asymptotic regime.
This can be explained by the concentration of degrees of freedom near the steep gradient of
\(\phi_{\text{ex}}\), as opposed to the whole domain.
Since the chosen forcing term exhibits non-trivial behavior near the pole, every element
close to the pole is refined multiple times.

\begin{minipage}[b]{0.45\textwidth}
    \centering
    \begin{tikzpicture}
        \begin{axis}[
                enlargelimits=false,
                width=0.9\textwidth, height=0.9\textwidth,
                hide axis,
                colorbar, colorbar/width=0.2cm,
                point meta min=-1, point meta max=1,
                colormap/jet,
                colorbar style={
                    at={(1.05,0.5)}, anchor=west,
                    ytick={-1, 0, 1},
                    scaled y ticks=false,
                    yticklabels={-1, 0, 1},
                },
            ]
            \addplot graphics [
                xmin=0, xmax=1,
                ymin=0, ymax=1,
            ] {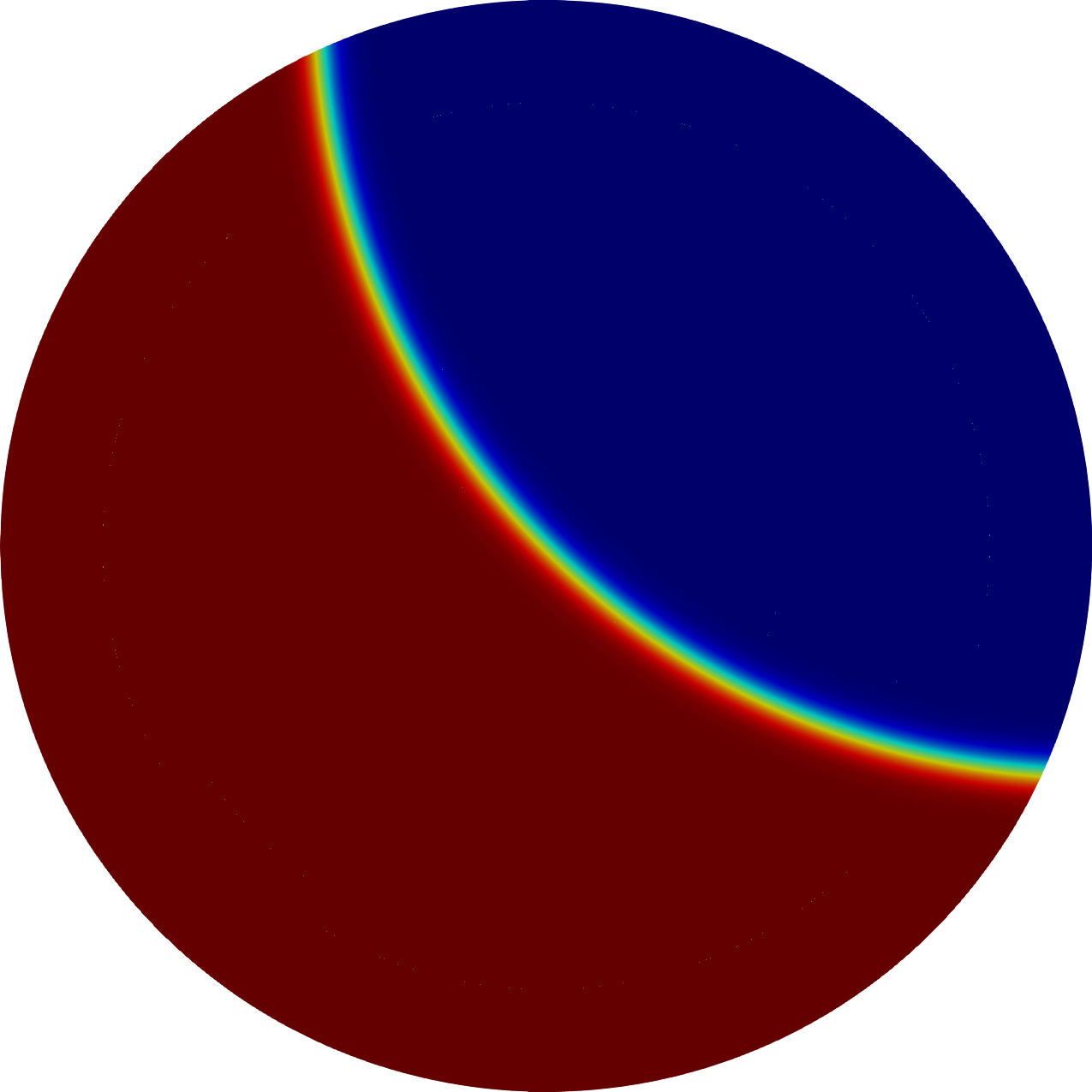};
        \end{axis}
    \end{tikzpicture}
    \captionof{figure}{Exact solution \(\phi_{\text{ex}}\).}
    \label{fig:n-form-analytical-solution}
\end{minipage}
\hfill
\begin{minipage}[b]{0.45\textwidth}
    \centering
    \begin{tikzpicture}[scale=1.0]
        \node (A) at (0,0) {Solve};
        \node (B) at (0,-2.0) {Estimate};
        \node (C) at (4.0,-2.0) {Mark};
        \node (D) at (4.0,0.0) {Refine};

        \draw[->,>=Stealth,line width=1pt] (A) to (B);
        \draw[->,>=Stealth,line width=1pt] (B) to (C);
        \draw[->,>=Stealth,line width=1pt] (C) to (D);
        \draw[->,>=Stealth,dashed,line width=1pt] (D) to (A);
    \end{tikzpicture}
    
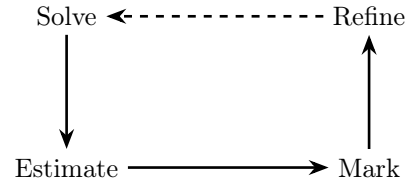
\captionof{figure}{Adaptive loop illustration. 
    Since the exact solution is known, we can compute the errors after solving, instead
    of estimating them.
    }
    \label{fig:adaptive-loop}
\end{minipage}

\newcommand{\logLogSlopeTriangle}[5]
{

    \pgfplotsextra
    {
        \pgfkeysgetvalue{/pgfplots/xmin}{\xmin}
        \pgfkeysgetvalue{/pgfplots/xmax}{\xmax}
        \pgfkeysgetvalue{/pgfplots/ymin}{\ymin}
        \pgfkeysgetvalue{/pgfplots/ymax}{\ymax}

        \pgfmathsetmacro{\xArel}{#1}
        \pgfmathsetmacro{\yArel}{#3}
        \pgfmathsetmacro{\xBrel}{#1-#2}
        \pgfmathsetmacro{\yBrel}{\yArel}
        \pgfmathsetmacro{\xCrel}{\xArel}

        \pgfmathsetmacro{\lnxB}{\xmin*(1-(#1-#2))+\xmax*(#1-#2)} 
        \pgfmathsetmacro{\lnxA}{\xmin*(1-#1)+\xmax*#1} 
        \pgfmathsetmacro{\lnyA}{\ymin*(1-#3)+\ymax*#3} 
        \pgfmathsetmacro{\lnyC}{\lnyA+#4*(\lnxA-\lnxB)}
        \pgfmathsetmacro{\yCrel}{\lnyC-\ymin)/(\ymax-\ymin)} 

        \coordinate (A) at (rel axis cs:\xArel,\yArel);
        \coordinate (B) at (rel axis cs:\xBrel,\yBrel);
        \coordinate (C) at (rel axis cs:\xCrel,\yCrel);

        \draw[#5]   (A)-- node[pos=0.5,anchor=north] {1}
                    (B)-- 
                    (C)-- node[pos=0.5,anchor=west] {#4}
                    cycle;
    }
}

\begin{figure}[h]
    \centering
    \begin{tikzpicture}
        \begin{loglogaxis}[
            width=\textwidth,
            height=0.4\textheight,
            xlabel={\(N^{-1/2}\)},
            ylabel={Error},
            grid=both,
            grid style={line width=0.2pt, draw=gray!30},
            major grid style={line width=0.4pt, draw=gray!60},
            legend pos=north west,
            legend cell align=left,
            mark size=2pt,
        ]

            \addplot[
                color=black,
                mark=square*,
                thick,
            ] table[
                x expr=1/sqrt(\thisrow{dofs}),
                y=errors,
                col sep=comma,
            ] {figures/hier-p2.csv};
            \addlegendentry{\(\overline{\hSpace{\polar}}^2,\ p^\theta=p^\rho=2\)}

            \addplot[
                color=black,
                mark=square,
                thick,
            ] table[
                x expr=1/sqrt(\thisrow{dofs}),
                y=errors,
                col sep=comma,
            ] {figures/tp-p2.csv};
            \addlegendentry{\(\overline{\vSpace{\polar}}^2,\ p^\theta=p^\rho=2\)}

            \addplot[
                color=black,
                mark=triangle*,
                thick,
            ] table[
                x expr=1/sqrt(\thisrow{dofs}),
                y=errors,
                col sep=comma,
            ] {figures/hier-p3.csv};
            \addlegendentry{\(\overline{\hSpace{\polar}}^2,\ p^\theta=p^\rho=3\)}

            \addplot[
                color=black,
                mark=triangle,
                thick,
            ] table[
                x expr=1/sqrt(\thisrow{dofs}),
                y=errors,
                col sep=comma,
            ] {figures/tp-p3.csv};
            \addlegendentry{\(\overline{\vSpace{\polar}}^2,\ p^\theta=p^\rho=3\)}

            \logLogSlopeTriangle{0.2}{0.1}{0.52}{2}{black};
            \logLogSlopeTriangle{0.18}{0.1}{0.41}{3}{black};
        \end{loglogaxis}
    \end{tikzpicture}
    \caption{
        Convergence rates for \eqref{eq:n-form-wf} using different polynomial degrees. \(N\)
        is defined as the number of degrees of freedom of either
        \(\overline{\hSpace{\polar}}^2\) or \(\overline{\vSpace{\polar}}^2\).
    } \label{fig:n-form-convergence}
\end{figure}
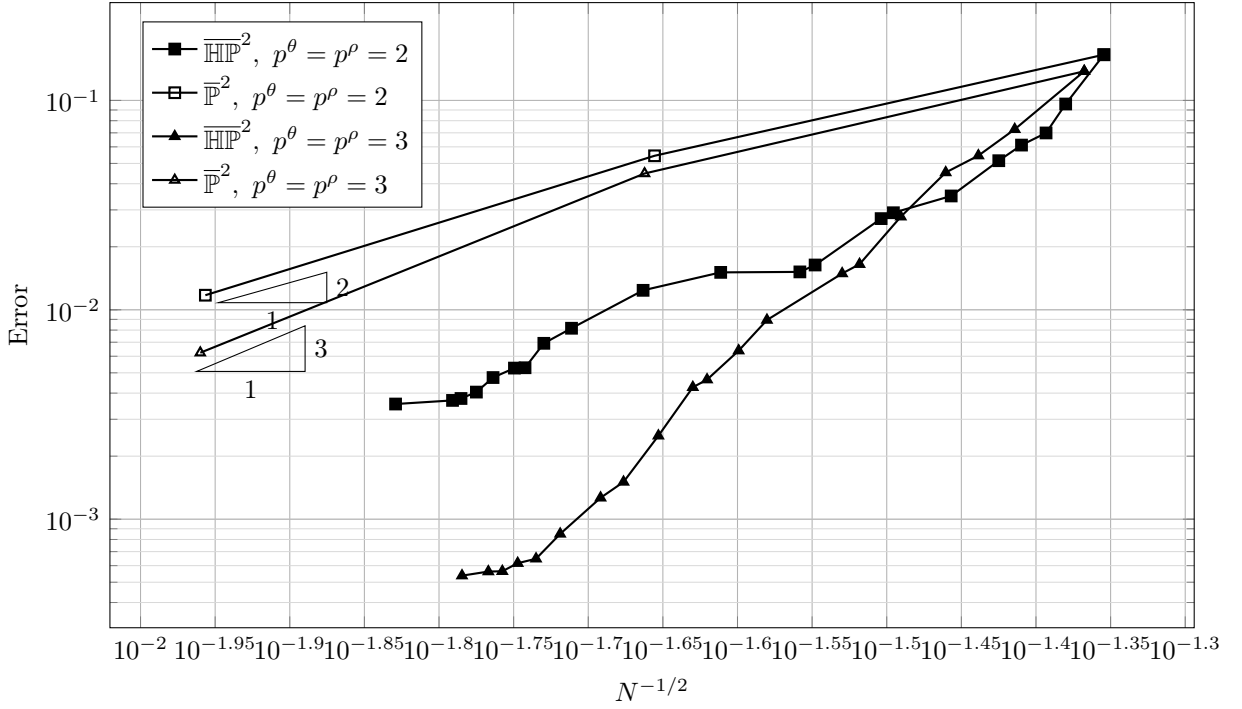

\begin{figure}[htbp]
    \def\sideLength{0.45\textwidth}
    \def\plotHeight{5.7cm}
    \centering
    \hfill
    \begin{subfigure}[b]{\sideLength}
        \includegraphics[height=\plotHeight]{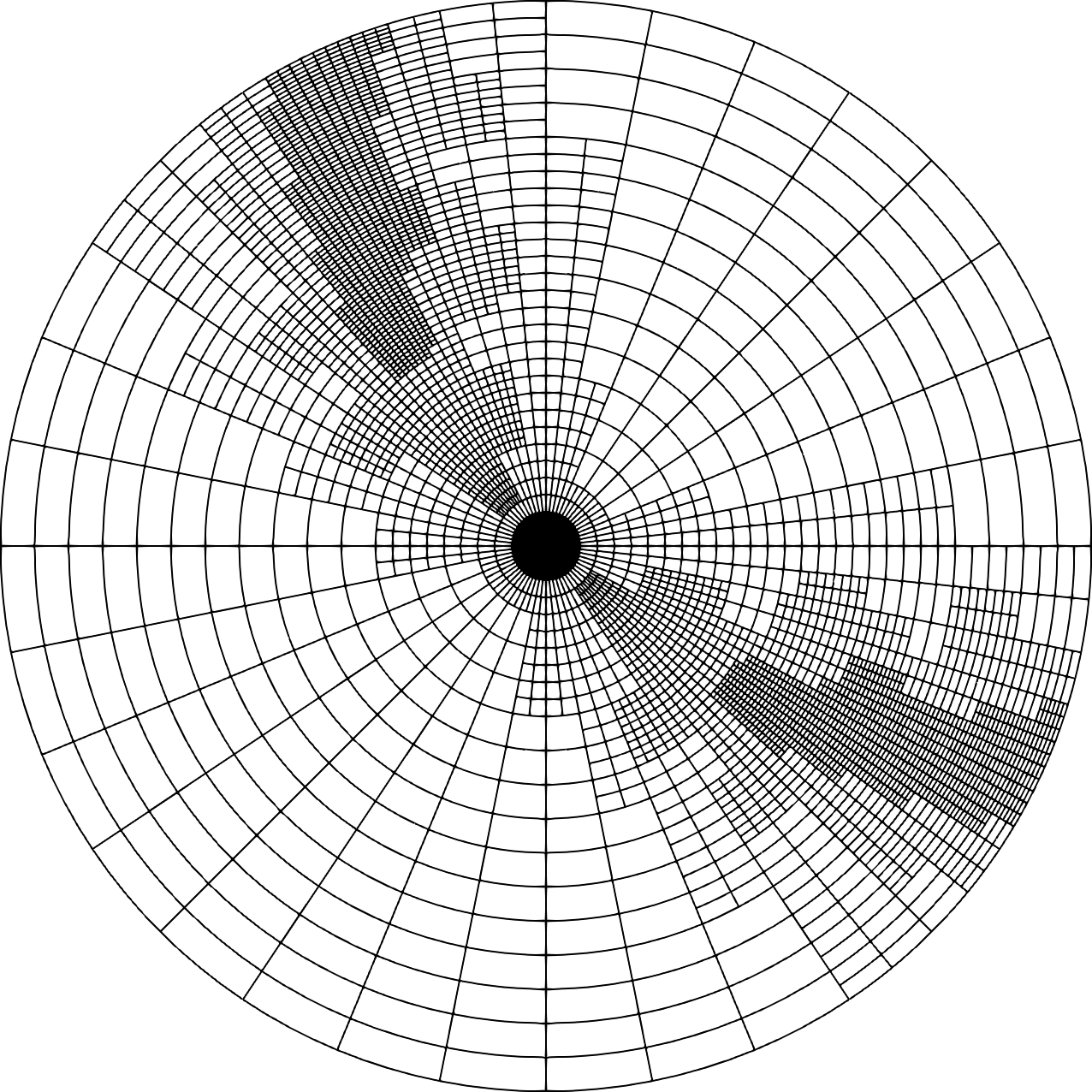}
        \caption{Full mesh.}
        \label{fig:n-form-mesh-full}
    \end{subfigure}
    \hfill
    \begin{subfigure}[b]{\sideLength}
        \centering
        \includegraphics[height=\plotHeight]{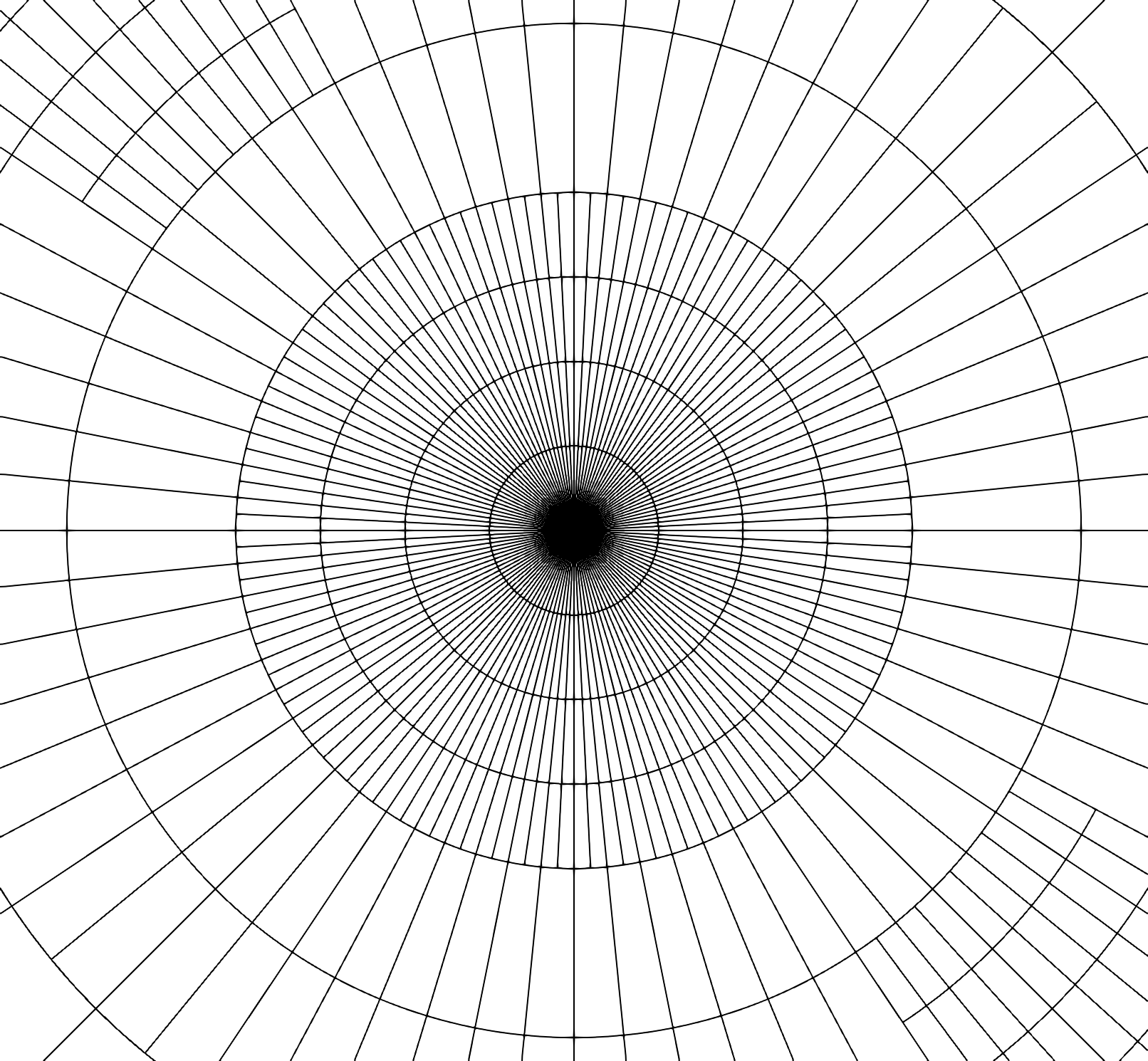}
        \caption{Close-up of the pole.}
        \label{fig:n-form-mesh-close}
    \end{subfigure}
    \hfill\strut
    \caption{
        Locally-refined polar mesh after the last refinement step, for the case \(p^\theta =
        p^\rho=3\).
    }
    \label{fig:n-form-mesh}
\end{figure}

%% file: sections/conclusions.tex
\section{Conclusions}
\label{sec:conclusions}

We have presented and analyzed adaptively-refinable, structure-preserving discretizations of the de Rham complex on polar domains.
The construction combines two ingredients: the polar spline differential forms of \cite{Toshniwal:2021}, which extract a smoother subspace from the tensor-product B-spline forms on a collapsed-edge parameterization of a disk-like domain, and the hierarchical spline framework of \cite{Vuong:2011,Giannelli:2014}, which selects, level by level, the active basis functions and thereby enables local refinement.
In this way, our results extend the polar spline de Rham complex of \cite{Toshniwal:2021} to the adaptively-refinable setting, and complement the hierarchical B-spline de Rham complexes of \cite{Evans:2020,Shepherd:2024,Cabanas:2025,Dijkstra:2026} by accommodating the polar singularity.

A key enabler of this extension is the close structural similarity between polar splines and tensor-product B-splines.
In particular, we prove that the polar spline $k$-form spaces, $k = 0, 1, 2$, are nested, and that their basis functions are locally linearly independent.
These are precisely the properties required by the abstract hierarchical spline theory of \cite{Giannelli:2014}, and they allow us to establish the linear independence of the hierarchical polar spline $k$-form basis functions.
Furthermore, to prove exactness of the corresponding complex, we exploited the local nature of the polar spline construction by introducing the notion of a reduced complex that isolates the behavior of the spaces in the neighborhood of the polar point.

In the context of structure-preserving discretizations on polar domains, the resulting construction has the following features.
\begin{itemize}
    \item It is cohomologically correct under sufficient conditions on the domain hierarchy that impose no constraints near the pole.
    The only conditions required are the known tensor-product conditions of \cite{Shepherd:2024}, which have to be verified only away from the pole; in particular, the elements surrounding the pole may be refined freely without breaking the cohomology of the complex.
    \item Together with the refinement algorithms for hierarchical B-spline de Rham complexes proposed in \cite{Cabanas:2025,Dijkstra:2026}, it provides the first adaptively-refinable, structure-preserving spline differential forms on polar domains.
    \item It is assembled entirely through extraction operators acting on periodic tensor-product B-splines, so it can be incorporated into existing isogeometric software with little effort; all the numerical experiments reported here were carried out with the structure-preserving library \emph{Mantis.jl} \cite{Cabanas:2026}.
\end{itemize}
The numerical experiments confirm the theory and illustrate the practical behavior of the construction.

Several extensions of the present work are natural.
The construction has been developed here for planar, disk-like domains with a single polar point; the cases of double polar points and of sphere-like surface geometries treated in \cite{Toshniwal:2021} can be addressed with analogous arguments.
Another generalization is to replace the univariate B-spline spaces with generalized spline spaces, such as multi-degree \cite{Toshniwal:2020} or Tchebycheffian splines \cite{Hiemstra:2020}, which retain the structural properties exploited in this paper.
This would require the generalization of the results from \cite{Shepherd:2024} to the case of hierarchical Tchebycheffian splines.
Finally, since polar spline de Rham complexes underpin structure-preserving solvers in magnetohydrodynamics and plasma physics \cite{Holderied:2022b,Possanner:2023}, the local refinement enabled here is a step towards adaptive, structure-preserving simulations on the tokamak- and stellarator-like geometries arising in those applications.

%% file: sections/app-intermediate-extraction.tex
\section{Extraction operators defining $\bspS^{k,\tint}$}\label{app:intermediate-extraction}
The pullbacks of splines in $\overline{\bspS}^{k,\tint}$, $k = 0, 1, 2$, are members of a subspace of $\tpF{k,\per}{\pkntsU,\kntsV}$ which can be defined using suitable extraction operators $\mat{E}^{k,\tint}$, $k = 0, 1, 2$.

Here we assume that we are working with the polar map as in \eqref{eq:polar_map}.
We will denote size $a \times a$ identity matrices by $\mat{I}_a$.
An explicit zero matrix of size $a \times b$ will be denoted by $\mat{0}_{a \times b}$.
Finally, with $n^{k,\tint}$, $k = 0, 1, 2$, defined as:
\begin{equation}
    n^{0,\tint} = \pndofU(\ndofV-1) + 1\;,\qquad
    n^{1,\tint} = 2\pndofU(\ndofV-1)\;,\qquad
    n^{2,\tint} = \pndofU(\ndofV-1)\;.
\end{equation}

\subsection{Definition of $\mat{E}^{0,\tint}$}
We define $\mat{E}^{0,\tint}$ as the following matrix of size $n^{0,\tint} \times \pndofU\ndofV$,
\begin{equation}\label{eq:int-ext-0}
    \mat{E}^{0,\tint} := 
    \begin{bmatrix}
    1 & \cdots & 1 &\\
    & & & \mat{I}_{n^{0,\tint}-1}
    \end{bmatrix}\;.
\end{equation}

\subsection{Definition of $\mat{E}^{1,\tint}$}
We define $\mat{E}^{1,\tint}$ as the following matrix of size $n^{1,\tint} \times \pndofU(2\ndofV - 1)$,
\begin{equation}\label{eq:int-ext-1}
    \mat{E}^{1,\tint} := 
    \begin{bmatrix}
    \mat{0}_{n^{1,\tint} \times \pndofU} & \mat{I}_{n^{1,\tint}}
    \end{bmatrix}\;.
\end{equation}

\subsection{Definition of $\mat{E}^{2,\tint}$}
We define $\mat{E}^{2,\tint}$ simply as the identity matrix of size $n^{2,\tint} \times n^{2,\tint}$,
\begin{equation}\label{eq:int-ext-2}
    \mat{E}^{2,\tint} := \mat{I}_{n^{2,\tint}}\;.
\end{equation}

%% file: sections/app-extraction.tex
\section{Extraction operators defining $\polarF{k}$}\label{app:extraction}
The definition of $\polarF{k}$ through extraction operators was performed in \cite{Toshniwal:2021}.
However, for the theoretical and implementational purposes of this paper, we represent these spaces as subspaces of the intermediate space $\bspS^{k,\tint}$.
The corresponding extraction matrices $\mat{E}^{k}$, $k = 0, 1, 2$, are defined below.

As in \cref{app:intermediate-extraction}, we will denote an identity matrix of size $a \times a$ by $\mat{I}_a$.
An explicit zero matrix of size $a \times b$ will be denoted by $\mat{0}_{a \times b}$.
Finally, with $n^k$, $k = 0, 1, 2$, as (cf. Theorem \ref{thm:polar-splines}\ref{item:cardinality}),
\begin{equation}
    n^0 = n^{0,\tint} - \pndofU + 2\;,\qquad
    n^1 = n^{1,\tint} - 2\pndofU + 2\;,\qquad
    n^2 = n^{2,\tint} - \pndofU\;.
\end{equation}

\subsection{Definition of $\mat{E}^0$}
We define $\mat{E}^0$ as the following block-diagonal matrix,
\begin{equation}\label{eq:polar-ext-0}
    \mat{E}^0 := 
    \begin{bmatrix}
    \widehat{\mat{E}}^0 &\\
    & \mat{I}_{n^0-3}
    \end{bmatrix}\;,
\end{equation}
where $\widehat{\mat{E}}^0$ is a $3 \times (\pndofU+1)$ matrix defined as
\begin{equation}\label{eq:polar-ext-0-core}
    \begin{split}
    \widehat{\mat{E}}^0 &:= 
    \underbrace{\begin{bmatrix*}[r]
        \frac{1}{3} & 0 & \frac{1}{3}\\
        -\frac{1}{6} & \frac{\sqrt{3}}{6} & \frac{1}{3}\\
        -\frac{1}{6} & -\frac{\sqrt{3}}{6} & \frac{1}{3}
    \end{bmatrix*}}_{:= \mat{L}}
    \underbrace{
    \begin{bmatrix*}[c]
        \frac{1}{R} & &\\
        & \frac{1}{R} & \\
        & & 1
    \end{bmatrix*}}_{:= \mat{R}}
    \begin{bmatrix*}[c]
        \mbf{F}_{11} & \mbf{F}_{12} & & \mbf{F}_{22} & \cdots & \mbf{F}_{\pndofU 2}\\
        1 & 1 & & 1 & \cdots & 1
    \end{bmatrix*}\;,
\end{split}
\end{equation}
and $R := \max\limits_{\substack{i = 1, \dots, \pndofU \\ j = 1, 2}} \|\mbf{F}_{ij}\|_2$.

\subsection{Definition of $\mat{E}^1$}
We define $\mat{E}^1$ as the following block-diagonal matrix,
\begin{equation}\label{eq:polar-ext-1}
    \mat{E}^1 := 
    \begin{bmatrix}
    \widehat{\mat{E}}^1_1 & & \widehat{\mat{E}}^1_2 & \\
    & \mat{I}_{\pndofU(\ndofV-2)} & & &\\
    &  & & & \mat{I}_{\pndofU(\ndofV-2)}
    \end{bmatrix}\;,
\end{equation}
where $\widehat{\mat{E}}^1_1$ is a $[2 \times \pndofU]$ matrix defined as
\begin{equation}\label{eq:polar-ext-1-core-1}
    \begin{split}
    \widehat{\mat{E}}^1_1 &:= 
    \begin{bmatrix*}[c]
        \bary^{2}_{22}-\bary^{2}_{12} & \bary^{2}_{32}-\bary^{2}_{22} &\cdots &  \bary^{2}_{12}-\bary^{2}_{\pndofU 2}\\
        \bary^{3}_{22}-\bary^{3}_{12} & \bary^{3}_{32}-\bary^{3}_{22} &\cdots & \bary^{3}_{12}-\bary^{3}_{\pndofU 2}
    \end{bmatrix*}\;,
\end{split}
\end{equation}
and $\widehat{\mat{E}}^1_2$ is a $[2 \times \pndofU]$ matrix defined as
\begin{equation}\label{eq:polar-ext-1-core-2}
    \begin{split}
        \widehat{\mat{E}}^1_2 &:= 
        \begin{bmatrix*}[c]
            \bary^{2}_{12}-\bary^{2}_{11} & \bary^{2}_{22}-\bary^{2}_{21} & \cdots & \bary^{2}_{\pndofU 2}-\bary^{2}_{\pndofU 1}\\
            \bary^{3}_{12}-\bary^{3}_{11} & \bary^{3}_{22}-\bary^{3}_{21} & \cdots & \bary^{3}_{\pndofU 2}-\bary^{3}_{\pndofU 1}
        \end{bmatrix*}\;.
\end{split}
\end{equation}

\subsection{Definition of $\mat{E}^2$}
We define $\mat{E}^2$ as the following matrix,
\begin{equation}\label{eq:polar-ext-2}
    \mat{E}^2 := 
    \begin{bmatrix}
        \mat{0}_{n^2 \times \pndofU} & \mat{I}_{n^2}
    \end{bmatrix}\;.
\end{equation}

%% file: sections/app-extraction-refinement.tex
\section{Definition of level-$\ell$ $k$-form polar basis functions}\label{app:extraction-refinement}

In order to build the matrices $\mat{E}^k_\ell$, $k = 0, 1, 2$, $\ell = 0, \dots, L$, we reuse the definitions from Appendix \ref{app:extraction}.
However, we do so for a modified choice of the control points $\mbf{F}_{ij}$ for each level.

In particular, assume that the control points $\mbf{F}_{ij}$, as in \eqref{eq:polar_map}-\eqref{eq:control-points}, are given at level $0$.
Moreover, let $\mat{K}_{\ell}^\cU$ and $\mat{K}_{\ell}^\cV$ denote the refinement matrices such that
\begin{equation}
    \begin{split}
        \bspB[\pkntsU_\ell] &= \mat{K}_{\ell}^{\cU,T} \bspB[\pkntsU_{\ell+1}]\;,\\
        \bspB[\kntsV_\ell] &= \mat{K}_{\ell}^{\cV,T} \bspB[\kntsV_{\ell+1}]\;,\\
    \end{split}
\end{equation}
Then, we define the control points for the polar map at level $\ell > 0$ recursively as,
\begin{equation}\label{eq:control-point-refinement}
    \begin{split}
        \left[\mbf{F}_{ij,\ell} : i = 1, \dots, \pndofU_{\ell}, j = 1, \dots, \ndofV_\ell\right] 
    &:=\\
    &\hspace*{-3cm}\left(
        \underbrace{\mat{K}_{\ell-1}^{\cU} \otimes \mat{K}_{\ell-1}^{\cV}}_{ =: \mat{K}^{\per}_{\ell-1}}
    \right)
    \left[\mbf{F}_{ij,\ell-1} : i = 1, \dots, \pndofU_{\ell-1}, j = 1, \dots, \ndofV_{\ell-1}\right] \;.
    \end{split}
\end{equation}
Then, repeating the construction from \cref{app:intermediate-extraction,app:extraction} with the updated control points, we can define the matrices $\mat{E}^{k,\tint}_{\ell}$ and $\mat{E}^k_{\ell}$, $k = 0, 1, 2$, and use them in \eqref{eq:polar-spline-levels}.

%% file: bibliography.bib
@article{Dijkstra:2026,
  title={Macro-element refinement schemes for {THB}-splines: Applications to {B}{\'e}zier projection and structure-preserving discretizations},
  author={Dijkstra, Kevin and Giannelli, Carlotta and Toshniwal, Deepesh},
  journal={Computer Methods in Applied Mechanics and Engineering},
  volume={452},
  pages={118707},
  year={2026},
  publisher={Elsevier}
}

@inproceedings{Possanner:2023,
  title={High-order structure-preserving algorithms for plasma hybrid models},
  author={Possanner, Stefan and Holderied, Florian and Li, Yingzhe and Na, Byung Kyu and Bell, Dominik and Hadjout, Said and G{\"u}{\c{c}}l{\"u}, Yaman},
  booktitle={International Conference on Geometric Science of Information},
  pages={263--271},
  year={2023},
  organization={Springer}
}

@article{Gucclu:2025,
  title={A broken-{FEEC} framework for structure-preserving discretizations of polar domains with tensor-product splines},
  author={G{\"u}{\c{c}}l{\"u}, Yaman and Patrizi, Francesco and Pinto, Martin Campos},
  journal={arXiv preprint arXiv:2505.15996},
  year={2025}
}

@ARTICLE{Arnold:2006,
	author = {Arnold, Douglas N. and Falk, Richard S. and Winther, Ragnar},
	title = {Finite element exterior calculus, homological techniques, and applications},
	journal = {Acta Numer.},
	year = {2006},
	volume = {15},
	pages = {1--155},
	fjournal = {Acta Numerica},
	issn = {0962-4929},
	mrclass = {58A15 (47N40 65N30 74G15 74S05)},
	mrnumber = {MR2269741 (2007j:58002)},
	mrreviewer = {Thomas Garrity}
}

@ARTICLE{Arnold:2010,
	author = {Arnold, Douglas N. and Falk, Richard S. and Winther, Ragnar},
	title = {Finite element exterior calculus: from {H}odge theory to numerical
	stability},
	journal = {Bull. Amer. Math. Soc. (N.S.)},
	year = {2010},
	volume = {47},
	pages = {281--354},
	number = {2},
	coden = {BAMOAD},
	fjournal = {American Mathematical Society. Bulletin. New Series},
	issn = {0273-0979},
	mrclass = {58A14 (65N30)},
	mrnumber = {2594630 (2011f:58005)},
	mrreviewer = {Satyendra K. Tomar}
}

@book{deBoor:1978,
  title={A practical guide to splines},
  author={De Boor, Carl},
  volume={27},
  year={1978},
  publisher={springer New York}
}

@ARTICLE{Buffa:2010,
	author = {A. Buffa and G. Sangalli and R. V{\'a}zquez},
	title = {Isogeometric analysis in electromagnetics: {B}-splines approximation},
	journal = {Comput. Methods Appl. Mech. Engrg.},
	year = {2010},
	volume = {199},
	pages = {1143 - 1152},
	number = {17-20},
	fjournal = {Computer Methods in Applied Mechanics and Engineering},
	issn = {0045-7825}
}

@article{Buffa:2011,
  title={Isogeometric discrete differential forms in three dimensions},
  author={Buffa, Annalisa and Rivas, Judith and Sangalli, Giancarlo and V{\'a}zquez, Rafael},
  journal={SIAM Journal on Numerical Analysis},
  volume={49},
  number={2},
  pages={818--844},
  year={2011},
  publisher={SIAM}
}

@article{Cabanas:2025,
  title={Construction of exact refinements for the two-dimensional {HB}/{THB}-spline de {R}ham complex},
  author={Cabanas, Diogo C and Shepherd, Kendrick M and Toshniwal, Deepesh and V{\'a}zquez, Rafael},
  journal={SIAM Journal on Scientific Computing (accepted); arXiv preprint arXiv:2502.19542},
  year={2026}
}

@article {CamposPinto:2022,
	AUTHOR = {Campos Pinto, Martin and Kormann, Katharina and
	Sonnendr\"ucker, Eric},
	TITLE = {Variational framework for structure-preserving electromagnetic
	particle-in-cell methods},
	JOURNAL = {J. Sci. Comput.},
	FJOURNAL = {Journal of Scientific Computing},
	VOLUME = {91},
	YEAR = {2022},
	NUMBER = {2},
	PAGES = {Paper No. 46, 39},
	ISSN = {0885-7474,1573-7691},
	MRCLASS = {78A25 (35Q61 35Q70 65M60 65M70 70S05)},
	MRNUMBER = {4402737},
	DOI = {10.1007/s10915-022-01781-3},
	URL = {https://doi.org/10.1007/s10915-022-01781-3},
}

@BOOK{Cottrell:2009,
	title = {Isogeometric {A}nalysis: toward integration of {CAD} and {FEA}},
	publisher = {John Wiley \& Sons},
	year = {2009},
	author = {Cottrell, J. A. and Hughes, T. J. R. and Bazilevs, Y.}
}

@article {Doelz:2024,
	AUTHOR = {D\"olz, J\"urgen and Ebert, David and Sch\"ops, Sebastian and
	Ziegler, Anna},
	TITLE = {Shape uncertainty quantification of {M}axwell eigenvalues and
	-modes with application to {TESLA} cavities},
	JOURNAL = {Comput. Methods Appl. Mech. Engrg.},
	FJOURNAL = {Computer Methods in Applied Mechanics and Engineering},
	VOLUME = {428},
	YEAR = {2024},
	PAGES = {Paper No. 117108, 21},
	ISSN = {0045-7825,1879-2138},
	MRCLASS = {78A99},
	MRNUMBER = {4756441},
	DOI = {10.1016/j.cma.2024.117108},
	URL = {https://doi.org/10.1016/j.cma.2024.117108},
}

@ARTICLE{Evans:2013,
	author = {Evans, J. A. and Hughes, T. J. R.},
	title = {Isogeometric Divergence-conforming {B}-splines for the {U}nsteady
	{N}avier-{S}tokes {E}quations},
	journal = {J. Comput. Phys.},
	year = {2013},
	volume = {241},
	pages = {141 - 167},
	fjournal = {Journal of Computational Physics },
	issn = {0021-9991}
}

@article{Evans:2020,
  title={Hierarchical {B}-spline complexes of discrete differential forms},
  author={Evans, John A and Scott, Michael A and Shepherd, Kendrick M and Thomas, Derek C and V{\'a}zquez Hern{\'a}ndez, Rafael},
  journal={IMA Journal of Numerical Analysis},
  volume={40},
  number={1},
  pages={422--473},
  year={2020},
  publisher={Oxford University Press}
}

@article{Giannelli:2012,
  title={{THB}-splines: The truncated basis for hierarchical splines},
  author={Giannelli, Carlotta and J{\"u}ttler, Bert and Speleers, Hendrik},
  journal={Computer Aided Geometric Design},
  volume={29},
  number={7},
  pages={485--498},
  year={2012},
  publisher={Elsevier}
}

@article{Giannelli:2014,
  title={Strongly stable bases for adaptively refined multilevel spline spaces},
  author={Giannelli, Carlotta and J{\"u}ttler, Bert and Speleers, Hendrik},
  journal={Advances in Computational Mathematics},
  volume={40},
  number={2},
  pages={459--490},
  year={2014},
  publisher={Springer}
}

@book {Hatcher:2002,
	AUTHOR = {Hatcher, A.},
	TITLE = {Algebraic topology},
	publisher = {Cambridge University Press, Cambridge},
	YEAR = {2002},
	PAGES = {xii+544},
	ISBN = {0-521-79160-X; 0-521-79540-0},
	MRCLASS = {55-01 (55-00)},
	MRNUMBER = {1867354 (2002k:55001)},
	MRREVIEWER = {Donald W. Kahn},
    note={http://www. math. cornell. edu/\~{} hatcher/AT/ATpage. html},
}

@article{Hiemstra:2020,
  title={A {T}chebycheffian extension of multidegree {B}-splines: Algorithmic computation and properties},
  author={Hiemstra, Ren{\'e} R and Hughes, Thomas JR and Manni, Carla and Speleers, Hendrik and Toshniwal, Deepesh},
  journal={SIAM Journal on Numerical Analysis},
  volume={58},
  number={2},
  pages={1138--1163},
  year={2020},
  publisher={SIAM}
}

@article {Holderied:2021,
	AUTHOR = {Holderied, Florian and Possanner, Stefan and Wang, Xin},
	TITLE = {M{HD}-kinetic hybrid code based on structure-preserving finite
	elements with particles-in-cell},
	JOURNAL = {J. Comput. Phys.},
	FJOURNAL = {Journal of Computational Physics},
	VOLUME = {433},
	YEAR = {2021},
	PAGES = {110143},
	ISSN = {0021-9991},
	MRCLASS = {76X05 (65M75 76M28 82D10 82M10)},
	MRNUMBER = {4218536},
	DOI = {10.1016/j.jcp.2021.110143},
	nourl = {https://doi.org/10.1016/j.jcp.2021.110143},
}

@article {Holderied:2022,
	AUTHOR = {Holderied, Florian and Possanner, Stefan},
	TITLE = {Magneto-hydrodynamic eigenvalue solver for axisymmetric
	equilibria based on smooth polar splines},
	JOURNAL = {J. Comput. Phys.},
	FJOURNAL = {Journal of Computational Physics},
	VOLUME = {464},
	YEAR = {2022},
	PAGES = {Paper No. 111329, 34},
	ISSN = {0021-9991,1090-2716},
	MRCLASS = {65N25 (65N30 76W05)},
	MRNUMBER = {4432263},
	MRREVIEWER = {Zhongyi\ Huang},
	DOI = {10.1016/j.jcp.2022.111329},
	URL = {https://doi.org/10.1016/j.jcp.2022.111329},
}

@phdthesis{Holderied:2022b,
  title={{STRUPHY}: a structure-preserving hybrid {MHD}-kinetic code for the interaction of energetic particles with {A}lfv{\'e}n waves in magnetized plasmas},
  author={Holderied, Florian Erich},
  year={2022},
  school={Technische Universit{\"a}t M{\"u}nchen}
}

@article {Kamensky:2017,
	AUTHOR = {Kamensky, David and Hsu, Ming-Chen and Yu, Yue and Evans, John
	A. and Sacks, Michael S. and Hughes, Thomas J. R.},
	TITLE = {Immersogeometric cardiovascular fluid-structure interaction
	analysis with divergence-conforming {B}-splines},
	JOURNAL = {Comput. Methods Appl. Mech. Engrg.},
	FJOURNAL = {Computer Methods in Applied Mechanics and Engineering},
	VOLUME = {314},
	YEAR = {2017},
	PAGES = {408--472},
	ISSN = {0045-7825,1879-2138},
	MRCLASS = {65M60 (74F10 76Z05 92C35)},
	MRNUMBER = {3589879},
	DOI = {10.1016/j.cma.2016.07.028},
	URL = {https://doi.org/10.1016/j.cma.2016.07.028},
}

@article{Shepherd:2024,
  title={Locally-Verifiable Sufficient Conditions for Exactness of the Hierarchical {B}-spline Discrete de {R}ham Complex in $\mathbb{R}^n$},
  author={Shepherd, Kendrick and Toshniwal, Deepesh},
  journal={Foundations of Computational Mathematics},
  pages={1--43},
  year={2024},
  publisher={Springer}
}

@article{Toshniwal:2017,
  title={Multi-degree smooth polar splines: A framework for geometric modeling and isogeometric analysis},
  author={Toshniwal, Deepesh and Speleers, Hendrik and Hiemstra, Ren{\'e} R and Hughes, Thomas JR},
  journal={Computer Methods in Applied Mechanics and Engineering},
  volume={316},
  pages={1005--1061},
  year={2017},
  publisher={Elsevier}
}

@article{Toshniwal:2020,
  title={Multi-degree {B}-splines: Algorithmic computation and properties},
  author={Toshniwal, Deepesh and Speleers, Hendrik and Hiemstra, Ren{\'e} R and Manni, Carla and Hughes, Thomas JR},
  journal={Computer Aided Geometric Design},
  volume={76},
  pages={101792},
  year={2020},
  publisher={Elsevier}
}

@article{Speleers:2021,
  title={A general class of {C1} smooth rational splines: Application to construction of exact ellipses and ellipsoids},
  author={Speleers, Hendrik and Toshniwal, Deepesh},
  journal={Computer-Aided Design},
  volume={132},
  pages={102982},
  year={2021},
  publisher={Elsevier}
}

@article{Toshniwal:2021,
  title={Isogeometric discrete differential forms: Non-uniform degrees, {B}{\'e}zier extraction, polar splines and flows on surfaces},
  author={Toshniwal, Deepesh and Hughes, Thomas JR},
  journal={Computer Methods in Applied Mechanics and Engineering},
  volume={376},
  pages={113576},
  year={2021},
  publisher={Elsevier}
}

@ARTICLE{Vuong:2011,
	author = { Vuong,A.-V. and Giannelli, C. and J\"uttler, B. and Simeon, B.},
	title = { A hierarchical approach to adaptive local refinement in isogeometric
	analysis},
	journal = {Comput. Methods Appl. Mech. Engrg.},
	year = {2011},
	volume = {200},
	pages = {3554--3567},
	number = {49-52},
	coden = {CMMECC},
	fjournal = {Computer Methods in Applied Mechanics and Engineering},
	issn = {0045-7825}
}

@misc{Cabanas:2026,
  author       = {C. Cabanas, Diogo and
                  Dekker, Joey and
                  Palha, Artur and
                  Toshniwal, Deepesh},
  title        = {{M}antis.jl a Structure-Preserving Finite Element
                   Library
                  },
  month        = feb,
  year         = 2026,
  publisher    = {Zenodo},
  version      = {v0.3.0},
  doi          = {10.5281/zenodo.18697535},
  url          = {https://doi.org/10.5281/zenodo.18697535},
  swhid        = {swh:1:dir:a57a074e12792ac4f59eaa6e28a447e0c0468412
                   ;origin=https://doi.org/10.5281/zenodo.17495870;vi
                   sit=swh:1:snp:255b896f74e10b886efb6d9c30e2023dee03
                   1e95;anchor=swh:1:rel:7dc8ee6b866bc34fadfb85a6f099
                   4456b286046a;path=MantisFEM-Mantis.jl-fd03adf
                  },
}
